\documentclass[12pt]{article} 

\usepackage{comment}

\usepackage[a4paper]{geometry}

\usepackage{amsmath,amsfonts,amsthm,amssymb,mathabx}
\usepackage{upgreek}
\usepackage[titletoc]{appendix}
\usepackage[all]{xy}

\usepackage{caption}

\usepackage{accents}

\newtheorem{theorem}{Theorem}[section]
\newtheorem{lemma}[theorem]{Lemma}
\newtheorem{proposition}[theorem]{Proposition}
\newtheorem{corollary}[theorem]{Corollary}

\newtheorem{remark}[theorem]{Remark}
\newtheorem{definition}[theorem]{Definition}

\newcommand*{\defeq}{\mathrel{\vcenter{\baselineskip0.5ex \lineskiplimit0pt
                     \hbox{\scriptsize.}\hbox{\scriptsize.}}}%
                     =}
                     
\usepackage[none]{hyphenat}
\usepackage{tikz}
\usepackage{tikz-cd}
\usepackage{rotating}

\usetikzlibrary{matrix}
\usetikzlibrary{arrows}
\usetikzlibrary{decorations.pathmorphing}
\usepackage{wrapfig}

\newcommand{\bb}[1]{\mathbb{#1}}

\DeclareRobustCommand{\underT}[1]{\underaccent{\bar}{#1}}
\newcommand{\underZ}[1]{{#1}_Z}

\newcommand{\KZ}[1]{\underZ{K}^{#1}}

\newcommand{\JZ}[1]{\underZ{J}^{#1}}
\newcommand{\JZP}[1]{J^{#1}_{Z+}}

\newcommand{\rep}[1]{\phi^{#1}}
\newcommand{\repZ}[1]{\underZ{\phi}^{#1}}
\newcommand{\kap}[1]{\kappa^{#1}}
\newcommand{\kapZ}[1]{\underZ{\kappa}^{#1}}
\newcommand{\prim}[1]{{\rep{#1}}'}
\newcommand{\primZ}[1]{{\repZ{#1}}'}
\newcommand{\phihat}[1]{\widehat{\rep{#1}}}
\newcommand{\phihatZ}[1]{\widehat{\repZ{#1}}}

\newcommand{\red}{\mathcal{G}}
\newcommand{\redZ}{\underZ{\mathcal{G}}}
\newcommand{\yZ}{\underZ{y}}

\newcommand{\rhoZ}{\underZ{\rho}}
\newcommand{\phiZ}{\underZ{\phi}}
\newcommand{\SSZ}{\underZ{\mathcal{S}}}

\newcommand{\xiG}{\xi^{i}}
\newcommand{\xiZ}{\underZ{\xi}^{i}}
\newcommand{\weil}{\omega^{i}}
\newcommand{\weilZ}{\underZ{\omega}^{i}}

\newcommand{\WZ}{\underZ{W}^{i}}
\newcommand{\Heis}{\mathcal{H}^{i}}
\newcommand{\HeisZ}{\underZ{\mathcal{H}}^{i}}

\newcommand{\GZ}{\underZ{G}}
\newcommand{\SZ}{\underZ{S}}

\newcommand{\jhat}{\widehat{j}}

\newcommand{\thetaZ}{\underZ{\theta}}

\newcommand{\param}{\varphi}
\newcommand{\paramT}{\underT{\varphi}}
\newcommand{\piT}{\underT{\pi}}
\newcommand{\PhiSC}{\Phi_{\mathrm{sc}}}
\newcommand{\GT}{\underT{G}}
\newcommand{\ST}{\underT{S}}
\newcommand{\JT}{\underT{J}}
\newcommand{\jT}{\widehat{\underT{j}}}
\newcommand{\chiT}{\underT{\chi}_0}
\newcommand{\thetaT}{\underT{\theta}}
\newcommand{\CT}{\widehat{\underT{C}}}
\newcommand{\MT}{\widehat{\underT{M}}}

\DeclareMathOperator{\Fr}{Fr}

\DeclareMathOperator{\Ad}{Ad}
\DeclareMathOperator{\AdT}{\underline{\Ad}}
\DeclareMathOperator{\Res}{Res}
\DeclareMathOperator{\Ind}{Ind}

\DeclareMathOperator{\Gal}{Gal}

\newcommand{\Cent}{\mathrm{Cent}}

\newcommand{\g}{\mathfrak{g}}
\newcommand{\gZ}{\underZ{\mathfrak{g}}}

\newcommand{\G}{\mathrm{G}}

\DeclareMathOperator{\res}{\mathfrak{f}} 
\DeclareMathOperator{\resun}{\overline{\mathfrak{f}}} 

\DeclareMathOperator{\p}{\mathfrak{p}} 

\newcommand{\z}{\mathfrak{z}^i} 

\newcommand{\zstar}{\mathfrak{z}^{i,*}}
\newcommand{\zstarri}{\mathfrak{z}^{i,*}_{-r_i}} 
\newcommand{\zZ}{\underZ{\mathfrak{z}}^i} 
\newcommand{\zZstar}{\underZ{\mathfrak{z}}^{i,*}} 
\newcommand{\zZstarri}{(\underZ{\mathfrak{z}}^{i,*})_{-r_i}} 

\newcommand{\un}{\mathrm{un}}

\newcommand{\s}{\mathfrak{s}} 
 
\newcommand{\eZ}{\underZ{e}} 

\newcommand{\bigslant}[2]{{\raisebox{.2em}{$#1$}\left/\raisebox{-.2em}{$#2$}\right.}}

\newcommand{\quo}[2]{#1/#2}

\usepackage[nottoc,notlot,notlof]{tocbibind}

\usepackage{fancyhdr}

\makeatletter
\def\blfootnote{\gdef\@thefnmark{}\@footnotetext}
\makeatother

\usepackage{titlesec}

\titleclass{\subsubsubsection}{straight}[\subsection]

\newcounter{subsubsubsection}[subsubsection]
\renewcommand\thesubsubsubsection{\thesubsubsection.\arabic{subsubsubsection}}

\titleformat{\subsubsubsection}
  {\normalfont\normalsize\bfseries}{\thesubsubsubsection}{1em}{}
\titlespacing*{\subsubsubsection}
{0pt}{3.25ex plus 1ex minus .2ex}{1.5ex plus .2ex}

\makeatletter
\renewcommand\paragraph{\@startsection{paragraph}{5}{\z@}%
  {3.25ex \@plus1ex \@minus.2ex}%
  {-1em}%
  {\normalfont\normalsize\bfseries}}
\renewcommand\subparagraph{\@startsection{subparagraph}{6}{\parindent}%
  {3.25ex \@plus1ex \@minus .2ex}%
  {-1em}%
  {\normalfont\normalsize\bfseries}}
\def\toclevel@subsubsubsection{4}
\def\toclevel@paragraph{5}
\def\toclevel@paragraph{6}
\def\l@subsubsubsection{\@dottedtocline{4}{7em}{4em}}
\def\l@paragraph{\@dottedtocline{5}{10em}{5em}}
\def\l@subparagraph{\@dottedtocline{6}{14em}{6em}}
\makeatother

\begin{document}

\title{Functoriality for Supercuspidal $L$-packets}
\author{Ad\`ele Bourgeois and Paul Mezo}
\date{}

\maketitle

\begin{abstract}

Kaletha constructs $L$-packets for supercuspidal $L$-parameters of tame $p$-adic groups. These $L$-packets consist entirely of supercuspidal representations, which are explicitly described. Using the explicit descriptions, we show that Kaletha's $L$-packets satisfy a fundamental functoriality property desired for the Local Langlands Correspondence. 

\end{abstract}

\blfootnote{This research was supported by the Fields Institute for Research in Mathematical Sciences and the NSERC Discovery Grant RGPIN-2017-06361.


} 
 
\tableofcontents

\section{Introduction}

Let $G$ be a connected reductive algebraic group defined over a
nonarchimedean local field $F$.  The Local Langlands Correspondence
(LLC)  is a conjectural map
$$\varphi \mapsto \Pi_{\varphi}$$
from $L$\emph{-parameters} to $L$\emph{-packets} \cite[Chapter
  III]{Borel:1979}.  The latter are finite 
sets of (equivalence classes of) irreducible representations of the
group of $F$-points, $G(F)$. 
The LLC is 
expected to satisfy numerous additional properties which give it
content. We focus on only two properties. The first property concerns
\emph{supercuspidal} representations, for which packets were developed by Kaletha
\cite{Kaletha:Regular, 
  Kaletha:SCPackets}.  The second concerns a basic form of
\emph{functoriality} as listed by Borel in \cite[Desideratum 10.3(5)]{Borel:1979}. This form of functoriality was later refined by Solleveld in \cite[Corollary 2]{Solleveld:2020}.

To describe the expected properties for supercuspidal representations,
we recall that an $L$-parameter is an $L$-homomorphism
$$\varphi: W_{F} \times \mathrm{SL}_{2} \rightarrow {^L}G$$
from the Weil-Deligne group into the $L$-group of $G$
\cite[Section 8.2]{Borel:1979}.  Following \cite[Section
  4.1]{Kaletha:SCPackets}, the $L$-parameter $\varphi$ is defined to
be \emph{supercuspidal} if it is trivial on $\mathrm{SL}_{2}$,
\emph{i.e.}
$$\varphi: W_{F} \rightarrow {^L}G,$$
and its image is not contained in a proper parabolic subgroup of ${^L}G$
\cite[Section 3.3]{Borel:1979}.  As observed in \cite[Section
  4.1]{Kaletha:SCPackets}, ``compound'' $L$-packets (or $L$-packets when $G$ is quasisplit) consisting entirely of
supercuspidal representations are conjectured to correspond precisely
to supercuspidal $L$-parameters (\cite[Section 3.5]{DR}, \cite{AMS}).
In his recent works \cite{Kaletha:Regular, Kaletha:SCPackets}, Kaletha
provides an explicit construction for these conjectured
$L$-packets, under the additional assumptions that $G$ splits over a
tamely ramified extension, and that the residual characteristic $p$ of
$F$ does not divide the order of the the Weyl group of $G$.  He further
proves that the $L$-packets satisfy some important properties
(\emph{e.g.} stability).  

The first goal of this paper is to show that these packets satisfy 
the desired functorial property \cite[Desideratum
  10.3(5)]{Borel:1979}.  For this reason, and from now on, we work
under the aforementioned additional 
assumptions on $G$ and the residual characteristic of $F$.
For the sake of simplicity, we also assume that $G$ is
  quasisplit over $F$ (see the discussion surrounding (\ref{SinG})).
Let $\PhiSC(G)$ denote the set (of conjugacy classes) of
supercuspidal $L$-parameters of $G$. Given $\param\in\PhiSC(G)$, we
let $\Pi_\param$ denote the associated supercuspidal
$L$-packet obtained via Kaletha's construction. 

\begin{theorem}\label{th:desideratum}
Suppose $G$ splits over a tamely ramified extension, is quasisplit and that the residual characteristic $p$ of $F$ does not divide the order of the Weyl group of $G$.
Let $\eta: G \rightarrow \GT$ be an $F$-morphism of connected
reductive $F$-groups such that 
\begin{itemize}
\item[i)] the kernel of $d\eta : \mathrm{Lie}(G) \rightarrow \mathrm{Lie}(\GT)$ is central,
\item[ii)] the cokernel of $\eta$ is an abelian $F$-group.
\end{itemize} 
Let $\paramT\in \PhiSC(\GT)$ and set $\param =
{^L\eta}\circ\paramT$. Then for all $\piT\in \Pi_{\paramT}$,
$\piT\circ\eta$ is the direct sum of finitely many irreducible
supercuspidal representations belonging to $\Pi_{\param}$. 
\end{theorem}

We recall that the map $^L\eta$ is given as follows. We let $\widehat{\eta}: \widehat{\GT} \rightarrow \widehat{G}$ be the induced map on the Langlands dual groups \cite[Sections 1 and 2]{Springer:1979}, and define $^L\eta : {^L\GT}\rightarrow {^LG}$ by $^L\eta(g,w) = (\widehat{\eta}(g),w)$ for all $g\in\widehat{\GT}, w\in W_F$. 

Note that the above theorem is a modified version of \cite[Desideratum 10.3(5)]{Borel:1979}, in which $\eta$ is required to have abelian kernel and cokernel. The hypothesis on $\eta$ is precisely \cite[Condition 1]{Solleveld:2020}, and is stronger \cite[Lemma 5.1]{Solleveld:2020} than that of \cite[Desideratum 10.3(5)]{Borel:1979}. We require this stronger hypothesis to rule out purely inseparable homomorphisms, such as \cite[V.16.1]{Borel:LAG}, that arise in positive characteristic and to ensure that the root systems of $G$ and $\GT$ are identified through $\eta$. 

In addition to proving Theorem~\ref{th:desideratum}, we also provide a description of the components of $\underT{\pi}\circ \eta$. 
The supercuspidal representations that make up
the $L$-packets of Theorem~\ref{th:desideratum} are constructed from
\emph{tame $F$-non-singular elliptic pairs}, which consist of a
particular kind of torus and a character thereof \cite[Definition
  3.4.1]{Kaletha:SCPackets}. Given such a pair 
$(\ST,\thetaT)$ of $\GT$, we let $\pi_{(\ST,\thetaT)}$ denote the
attached supercuspidal representation of $\GT(F)$, which is obtained from the \emph{J.-K.~Yu construction} \cite{Yu:2001} after unfolding $(\ST,\thetaT)$ into an appropriate \emph{$\GT$-datum}. The representation
$\pi_{(\ST,\thetaT)}$ may be reducible, and its
irreducible components form part of an $L$-packet. The
first big result of this paper is writing $\pi_{(\ST,\thetaT)}\circ
\eta$ as a direct sum  of conjugates of $\pi_{(S,\theta)}$, where $(S,\theta)$ is a tame
$F$-non-singular elliptic pair of $G$ that satisfies $\eta(S)\subset
\ST$ and $\theta = \thetaT\circ\eta$. This is Theorem~\ref{th:fullTheorem}.

The composition  $\pi_{(\ST,\thetaT)}\circ\eta$ can be
viewed as the restriction of $\pi_{(\ST,\thetaT)}$ to $\eta(G(F))$.
Having abelian cokernel implies that $\eta(G)$ is a subgroup of
$\GT$ which contains the derived subgroup $[\GT,\GT]$. The kernel of $\eta$, which we denote by $Z$, is a central subgroup by \cite[Lemma 5.1]{Solleveld:2020}. It will be convenient to write $\GZ$ for $\eta(G)$. In this way, $\GZ \simeq G/Z$ and $\GT = Z(\GT)\GZ$.
We compute the restriction of $\pi_{(\ST,\thetaT)}$ to $\eta(G(F)) \simeq G(F)/Z(F)$ in
two steps. First, by restricting to $\eta(G)(F) =
\GZ(F)$, and second, by further restricting to $G(F)/Z(F)$. The group
$\eta(G(F))$ is a normal subgroup of $\GZ(F)$
(Corollary~\ref{cor:normal}) and their quotient is parameterized by a
subgroup of a
Galois cohomology group $H^1(F,Z)$ \cite[Proposition 12.3.4]{Springer:LAG}, which may be nontrivial. The restriction of
supercuspidal representations to subgroups that contain the derived
subgroup was extensively studied in \cite{thesisPaper}. As such, we can
apply the results therein to obtain a description for
$\pi_{(\ST,\thetaT)}|_{\GZ(F)}$ (Theorem~\ref{th:HG}). The second restriction can be computed via Mackey theory, as the quotient $\GZ(F)/(G(F)/Z(F))$ is compact and abelian \cite{Silberger:1979}.  The decomposition of $\pi_{(\ST,\thetaT)}\circ\eta$ is given
in Theorem  \ref{th:mainTheorem}.

In order to describe the supercuspidal representations in the
$L$-packets $\Pi_\param$ and 
$\Pi_{\paramT}$, one must know which tame $F$-non-singular elliptic
pairs to use. These pairs are provided by
\emph{supercuspidal $L$-packet data} \cite[Definition
  4.1.4]{Kaletha:SCPackets}.  The supercuspidal $L$-packet
data for $\param$ and $\paramT$ consist of tuples
$(S,\jhat,\chi_0,\theta)$ and $(\ST,\jT,\chiT,\thetaT)$,
respectively. Unlike the previous paragraph, $S$ and $\ST$ are not subtori of $G$ and $\GT$. Rather, they are embedded into subtori of the respective groups. 
The elements $\jhat$ and $\jT$
specify families of \emph{admissible embeddings} $S(F)\rightarrow
G(F)$ and $\ST(F)\rightarrow \GT(F)$, denoted $J(F)$ and $\JT(F)$,
respectively. Each embedding $j\in J(F)$ $(\underT{j}\in\JT(F))$ is
used to generate a tame $F$-non-singular elliptic pair $(jS,j\theta)$
($(\underT{j}\ST,\underT{j}\thetaT)$), for which we let the components
of $\pi_{(jS,j\theta)}$ ($\pi_{(\underT{j}\ST,\underT{j}\thetaT)}$) be
elements of $\Pi_\param$ ($\Pi_{\paramT}$).

In order to apply our decomposition formula for
$\pi_{(\underT{j}\ST,\underT{j}\thetaT)}\circ\eta$
and relate it to representations in $\Pi_\param$, we must first
establish an appropriate relationship between the supercuspidal $L$-packet data and
the admissible embeddings. This is given to us by
Theorem~\ref{th:data}, another key result of this paper, in which we
show that for all $\underT{j}\in\JT(F)$, there exists $j\in J(F)$ such
that $\eta(jS)\subset \underT{j}\ST$ and $j\theta =
\underT{j}\thetaT\circ\eta$. As such, we obtain a decomposition formula
for $\pi_{(\underT{j}\ST,\underT{j}\thetaT)}\circ\eta$ in terms of
certain conjugates of $\pi_{(jS,j\theta)}$.  This completes the
proof of Theorem~\ref{th:desideratum}.

The second goal of this paper is to provide a more detailed
  description of 
the decomposition of $\underT{\pi} \circ \eta$ in Theorem
\ref{th:desideratum} in the special case 
that both $\param$ and
$\paramT$ are \emph{regular} supercuspidal $L$-parameters
\cite[Definition 5.2.3]{Kaletha:Regular}.
The regularity assumptions on the $L$-parameters have several pleasant
consequences.  We list them for $\varphi$, with the understanding that
their analogues hold for $\underT{\varphi}$.  First, the representations
$\pi_{(jS,j\theta)}$ are irreducible for all $j \in J(F)$. From this it follows
that the set $J(F)$ parameterizes the
representations in $\Pi_{\varphi}$.
Second, the set $J(F)$ is in bijection 
with characters of the usual component group that one is accustomed to
seeing in Langlands correspondences.  More precisely, the component
group of the centralizer of the image of $\varphi$ is in bijection
with the Galois-fixed subgroup of some torus $\hat{S}$ \cite[Lemma
  5.3.4]{Kaletha:Regular}, and certain characters of this torus
are in bijection with $J(F)$ (see (\ref{kottwitz86})).  Since
$\hat{S}$ is abelian, so too is the component group.

The regularity assumptions consequently allow us to write
$$\underT{\pi} = \pi_{(\underT{j}\ST,\underT{j}\thetaT)} =
\pi_{(\paramT,\underT{\varrho})},$$ 
where $\underT{j}\in\underT{J}(F)$ corresponds to a character
$\underT{\varrho}$ of the component group for $\paramT$.  The map $\hat{\eta}$
sends the component group for $\underT{\varphi}$ to the one for
$\varphi$.  The precise description of $\underT{\pi} \circ \eta$
is given in Proposition  \ref{prop:otherpinning}.  We summarize it
here as follows. 
\begin{theorem}\label{th:Solleveld}
Let $\eta: G\rightarrow \GT$, $\paramT$ and $\param =
{^L\eta}\circ\paramT$ be as in Theorem
\ref{th:desideratum}. Assume that $\param$ and
$\paramT$  are regular. Then 
$$\pi_{(\paramT,\underT{\varrho})}\circ \eta \simeq \underset{\varrho}{\bigoplus}\
\mathrm{Hom}\left( \underT{\varrho}, (\varrho\circ \hat{\eta}) \right)
\otimes \pi_{(\param,\varrho)}$$
where $\underT{\varrho}$ and $\varrho$ are characters of the Langlands
component groups.
\end{theorem}

Theorem \ref{th:Solleveld} is the proof of a conjecture of Solleveld for
regular supercuspidal $L$-parameters \cite[Conjecture
  2]{Solleveld:2020}.  Solleveld  has proved his conjecture in a
variety of cases \cite[Theorem
  3]{Solleveld:2020}. The only overlap of Theorem \ref{th:Solleveld} with
these cases is when $G$ and $\underT{G}$ are inner forms of
$\mathrm{GL}_{n}$, $\mathrm{SL}_{n}$ or $\mathrm{PGL}_{n}$. 

One might hope that the regularity of $\param =
{^L\eta}\circ\paramT$ in Theorem \ref{th:Solleveld} would follow from
the regularity of $\underT{\varphi}$.  While this is not true in general, as illustrated with a counterexample at the end of Section \ref{sec:regularParams}, the converse implication holds (Corollary
\ref{cor:regularparams}).  Furthermore, as explained after
\cite[Definition 3.7.3]{Kaletha:Regular},  regular 
$L$-parameters are typical among all supercuspidal $L$-parameters.
 
Let us discuss how one might extend Theorem \ref{th:Solleveld} to non-regular supercuspidal $L$-parameters. In this case, $\underT{J}(F)$ is no longer a parameterizing set for $\Pi_{\paramT}$ since the representations $\pi_{(\underT{j}\ST,\underT{j}\thetaT)}, \underT{j}\in \underT{J}(F),$ may be reducible. For each $\underT{j}\in \underT{J}(F)$, the irreducible components of $\pi_{(\underT{j}\ST,\underT{j}\thetaT)}$ are parameterized by certain representations of $N(\underT{j}\ST,\GT)(F)_{\underT{j}\thetaT}$, the stabilizer of the pair $(\underT{j}\ST,\underT{j}\thetaT)$ in $N(\underT{j}\ST,\GT)(F)$ \cite[Corollary 3.4.7]{Kaletha:SCPackets}. It appears that \cite[Proposition 4.3.2]{Kaletha:SCPackets} serves as a bridge between $\{N(\underT{j}\ST,\GT)(F)_{\underT{j}\thetaT}: \underT{j}\in\underT{J}(F)\}$ and the component group of the centralizer of the image of $\paramT$. Another key step in the proof of Theorem \ref{th:Solleveld} is the decomposition formula for $\pi_{(\underT{j}\ST,\underT{j}\thetaT)}\circ \eta$. Removing the regularity hypothesis means one would need to derive the decomposition formula of $\underT{\pi}\circ \eta$, where $\underT{\pi}$ is an irreducible component of $\pi_{(\underT{j}\ST,\underT{j}\thetaT)}$. This would require a deeper study of the results in \cite[Section 3]{Kaletha:SCPackets}. Once one has such a decomposition formula, we believe that similar arguments as the ones in Section \ref{sec:SolleveldConj} could be applied.

Let us briefly indicate what is required to extend
Theorem \ref{th:desideratum} and Theorem  \ref{th:Solleveld} to 
non-quasisplit groups.  Every  connected reductive
algebraic $F$-group $G'$ is an inner form of a quasisplit form $G$.
When $\mathrm{char} F = 0$, the group $G'$ may be assigned to a class
of \emph{rigid inner twists} for $G$ 
\cite[Corollary 3.8 and Section 5.1]{Kaletha:2016}.  This class 
is an element in a set of the form $H^{1}(u \rightarrow W, Z'
\rightarrow G)$ which we shall not describe.  For any $j\in J(F)$,
there is a natural surjection 
\begin{equation}
  \label{SinG}
H^{1}(u \rightarrow W, Z' \rightarrow jS) \  \longrightarrow \ H^{1}(u \rightarrow W, Z'
\rightarrow G),
\end{equation}
where $jS$ is a maximal torus of $G$.
The elements in a supercuspidal $L$-packet of $G'$ are indexed by
the fibre in $H^{1}(u \rightarrow W, Z' \rightarrow jS)$ over the class
in $H^{1}(u \rightarrow W, Z' \rightarrow G)$ corresponding to $G'$
\cite[Section 5.3]{Kaletha:Regular}.
If one is only interested in the quasisplit form \emph{i.e.} $G' = G$, the
classes of rigid inner twists may be chosen to equal the usual Galois cohomology sets (which we recall more precisely below), and a supercuspidal packet is indexed by the fibre of
the more familiar map
\begin{align}\label{eq:Hmap}
H^{1}(F,jS) \rightarrow H^{1}(F, G)
\end{align}
over the trivial class.  This fibre is in bijection with the set of
admissible embeddings $J(F)$ above.  In general, the fibre of
(\ref{SinG}) corresponding to $G'$ is in bijection with the set
$J'(F)$ of admissible
embeddings into (the rigid inner twist for) $G'(F)$.  When
$\mathrm{char}F \neq 0$, a parallel picture is given in 
\cite{Dillery}.
The constructions
and results of Sections \ref{sec:supercuspidalRelations} and \ref{sec:packetRelations} apply to $G'(F)$ and $J'(F)$ 
in exactly the same manner as they do to $G(F)$ and $J(F)$.  More work is required to
accommodate $G'(F)$ and $J'(F)$ in the
constructions of Section \ref{sec:specializeReg}.  Rather than working with characters
of $\pi_{0}(\hat{S}^{\Gamma}/Z(\hat{G})^{\Gamma})$, one works with the
characters of the larger group $\pi_{0}([\widehat{\bar{S}}]^{+})$
which appears in \cite[Lemma 5.3.4]{Kaletha:Regular} and \cite[Corollary
  5.4]{Kaletha:2016}.   The admissible embeddings $J'(F)$ correspond
to certain characters of $\pi_{0}([\widehat{\bar{S}}]^{+})$ and
these characters correspond to the representations in the $L$-packet
for $G'$.  A discussion of such matters may be found in \cite[Section
  5.4]{Kaletha:2016}.  In view of the length of this paper, which
deals only with quasisplit groups, it seems prudent to leave the
treatment of non-quasisplit groups to some future work.

The paper is organized as follows. In
Section~\ref{sec:supercuspidalRelations}, we start by giving a summary
of the construction of supercuspidal representations from tame
$F$-non-singular elliptic pairs $(\ST,\thetaT)$ (Section~\ref{sec:summaryKaletha}).
We then move on to proving Theorem~\ref{th:fullTheorem}, which
describes the decomposition of $\pi_{(\ST,\thetaT)}\circ\eta$
(Section~\ref{sec:MainTheorem}). In Section~\ref{sec:packetRelations}, we provide a summary
of the construction of the supercuspidal $L$-packets
(Section~\ref{sec:constructPackets}). We then establish the
relationship between the supercuspidal $L$-packet data associated to $\param$ and
$\paramT$ (Section~\ref{sec:packetData}) and end with  the
proof of Theorem~\ref{th:desideratum}
(Section~\ref{sec:proofDesideratum}). In Section \ref{sec:specializeReg}, we start by describing the regular supercuspidal $L$-parameters and their corresponding $L$-packet structure, as well as a discussion on when one might expect both parameters $\param$ and $\paramT$ to be regular (Section \ref{sec:regularParams}). We then proceed to reparameterizing the $L$-packets in terms of characters of their corresponding component groups and proving Theorem \ref{th:Solleveld}  (Section \ref{sec:SolleveldConj}).

We close this introduction with a list of notation. Given the nonarchimedean local field $F$, we denote by $\mathcal{O}_F$ its ring of integers, $\p_F$ 
the unique maximal ideal of $\mathcal{O}_F$ and $\res$ its residue
field of prime characteristic $p$. Let $F^\un$ be a maximal unramified
extension of $F$. The residue field of $F^\un$ is
an algebraic closure of $\res$, so we denote it by $\resun$. The Galois group $\Gal(F^\un/F)$ is canonically isomorphic to $\Gal(\resun/\res)$, and we denote its generator by $\mathrm{Fr}$. Let
$\Gamma = \Gal(F^{\mathrm{sep}}/F)$ denote the absolute Galois groups of $F$, where $F^{\mathrm{sep}}$ is a separable closure of $F$. We use the notation $I_F$ and $P_F$ for the inertia subgroup and wild inertia
subgroup of the Weil group $W_F$, respectively. We also let $E$ denote the tamely ramified extension of $F$ over which $G$ splits.

In various instances of the paper, we will encounter different types of cohomology groups. Given an algebraic group $G'$ that is defined over $F$, we write $H^1(F,G')$ for $H^1(\Gamma,G'(F^{\mathrm{sep}}))$. Similarly, given an algebraic group $\mathcal{G}'$ that is defined over $\res$, we write $H^1(\res,\mathcal{G}')$ for $H^1(\Gal(\resun/\res),\mathcal{G}'(\resun))$. Furthermore, given a group $\tilde{G}$ with $\Gal(F^\un/F)$-action, we write $H^1(\mathrm{Fr},\tilde{G})$ for $H^1(\Gal(F^\un/F),\tilde{G})$ and $\tilde{G}^{\mathrm{Fr}}$ for $\tilde{G}^{\Gal(F^\un/F)}$.

Given a maximal torus $T$ of $G$, we let $R(G,T)$ denote the root
system of $G$ with respect to $T$. Given $\alpha\in R(G,T)$, we denote
the associated root subgroup by $U_\alpha$. Letting $\underT{T}$ denote
the maximal torus of $\GT$ such that $\eta(T) = \underT{T}\cap\eta(G)$
(given by \cite[Theorem 2.2]{thesisPaper}), the root systems $R(G,T)$
and $R(\GT,\underT{T})$ are canonically identified, and the Weyl groups
of $G$ and $\GT$ coincide. We use $\g$ and $\underT{\g}$ for the Lie
algebras of $G$ and $\GT$, respectively. 

Finally, given groups $H_1\subset H_2$, $h\in H_2$ and a representation $\gamma$ of $H_1$, we let $^hH_1 := \mathrm{Ad}(h)(H_1) = hH_1h^{-1}$ and $^h\gamma := \gamma \circ \mathrm{Ad}(h^{-1})$.

We would like to thank Tasho Kaletha for taking the time to answer our questions regarding the results of his papers, as well as providing
guidance in establishing the content of this manuscript.  We would
like to thank Maarten Solleveld for pointing out the missing
hypotheses for non-zero characteristic in an earlier draft. We would also like to thank Anne-Marie Aubert for her comments, which in particular allowed us to find a counterexample to the converse of Lemma \ref{lem:regularpairs}.
Finally, we are also grateful to Monica Nevins for her insights on the J.-K.~Yu
construction, which allowed us to finalize the proofs that lead to our
decomposition formula.

\section{The Decomposition of $\pi_{(\ST,\thetaT)}\circ\eta$}\label{sec:supercuspidalRelations}

In \cite{Kaletha:Regular, Kaletha:SCPackets}, Kaletha describes a way
to construct a supercuspidal representation $\pi_{(\ST,\thetaT)}$ of
$\GT$ from a tame $F$-non-singular elliptic pair $(\ST,\thetaT)$
\cite[Definition 3.4.1]{Kaletha:SCPackets}. Here $\ST$ is a
maximally unramified elliptic maximal torus and $\thetaT$  is a character of $\ST(F)$ satisfying a certain \emph{non-singularity}
condition. The irreducible components of these representations are
what make up the supercuspidal $L$-packets. As such, finding a decomposition formula for $\pi_{(\ST,\thetaT)}\circ\eta$ is crucial in proving Theorem~\ref{th:desideratum}. 

\subsection{The Construction of Supercuspidal Representations from Tame $F$-non-singular Elliptic Pairs}\label{sec:summaryKaletha}


Let us recall the construction of supercuspidal representations from tame $F$-non-singular elliptic pairs as per \cite{Kaletha:Regular,Kaletha:SCPackets}. For simplicity of notation, we will describe the construction over $G$, though it is also applied to $\GZ \simeq G/Z$ and $\GT$.

The reader is assumed to be familiar with the structure theory of $p$-adic groups. Following the notation from \cite{Kaletha:Regular}, we write $\mathcal{B}(G,F)$ for the reduced building of $G$ over $F$ and $\mathcal{A}(G,T,F)$ for the apartment associated to any maximal torus $T$ of $G$ which is maximally split. For each $x\in\mathcal{B}(G,F)$, we set $G(F)_x$ to be the stabilizer of $x$ in $G(F)$. Furthermore,
for $r > 0$, $G(F)_{x,r}$ denotes the Moy-Prasad filtration subgroup
of the parahoric subgroup $G(F)_{x,0}$. We will be using Kaletha and Prasad's definitions \cite[Definition 13.2.1]{KalethaPrasad}, which coincide with the ones of Moy and Prasad given our tameness assumption \cite[p.XXV]{KalethaPrasad}. We also set $G(F)_{x,r^+} =
\underset{t >r}{\bigcup}G(F)_{x,t}$. We use colons to abbreviate
quotients, that is $G(F)_{x,r:t} = \quo{G(F)_{x,r}}{G(F)_{x,t}}$ for
$t > r$. We have analogous filtrations of
$\mathcal{O}_F$-submodules at the level of the Lie algebra.  

For all $r>0$, the quotient $G(F)_{x,r:r^+}$ is an abelian group and
is isomorphic to its Lie algebra analog $\g(F)_{x,r:r^+}$ via Adler's
mock exponential map \cite{Adler:1998}. The quotient $G(F)_{x,0:0^+}$
is also very important, as it results in the $\res$-points of a
reductive group, which we refer to as the \emph{reductive quotient of $G$ at $x$}. 

The construction of the supercuspidal representation
$\pi_{(S,\theta)}$ of $G$ starts from a tame $F$-non-singular elliptic pair
$(S,\theta)$ in the sense of \cite[Definition
  3.4.1]{Kaletha:SCPackets}. The representation $\pi_{(S,\theta)}$ is
obtained in two steps. One starts by unfolding the pair $(S,\theta)$
into a $G$-datum $\Psi_{(S,\theta)} =
(\vec{G},y,\vec{r},\rho,\vec{\phi})$ in the sense of \cite[Section
  3]{Yu:2001}. We will refer to $\Psi_{(S,\theta)}$ as the \emph{corresponding $G$-datum} of $(S,\theta)$. The properties of $S$ and $\theta$ provided by \cite[Definition
  3.4.1]{Kaletha:SCPackets} allow us to go to the reductive quotient and use the theory of Deligne-Lusztig cuspidal representations in order to construct $\rho$, the so-called \emph{depth-zero} part of the datum $\Psi_{(S,\theta)}$. The second step consists of applying the J.-K.~Yu construction
\cite{Yu:2001} on the obtained $G$-datum. The unfolding of
the tame $F$-non-singular elliptic pair into a $G$-datum is given as follows. 

\paragraph{The twisted Levi sequence $\vec{G}$ and the sequence $\vec{r}$:} We recall how to construct a Levi sequence from $S$ as per \cite[Section 3.6]{Kaletha:Regular}. For all $r>0$, let ${E_r^\times = 1+ \mathfrak{p}_E^{\lceil er \rceil}}$, where $e$ denotes the ramification degree of $E/F$. We consider the set ${R_r \defeq \{\alpha\in R(G,S): \theta(N_{E/F}(\check{\alpha}(E^\times_r)))=1\}}$,where $\check{\alpha}$ is the coroot associated to $\alpha$ and $N_{E/F}$ is the norm of $E/F$, and set $R_{r^+} = \underset{s>r}{\cap}R_s$. There will be breaks in this filtration, $r_{d-1}>r_{d-2}>\cdots > r_0 > 0$, and we set $r_{-1} = 0$ and $r_d = \mathrm{depth}(\theta)$. Then for all $0\leq i\leq d$, $G^i\defeq \langle S,U_\alpha:\alpha\in R_{r_{i-1}^+} \rangle$ is a tamely ramified twisted Levi subgroup of $G$ \cite[Lemma 3.6.1]{Kaletha:Regular}. These twisted Levi subgroups are what we use to form the twisted Levi sequence $\vec{G} = (G^0,\dots,G^d)$. We also set $G^{-1} = S$ and $\vec{r} = (r_0,\dots,r_d)$. 

\paragraph{The character sequence $\vec{\phi}$:} By \cite[Proposition 3.6.7]{Kaletha:Regular} , given the character $\theta$ of $S(F)$, there exists a Howe factorization, that is a sequence of characters $\rep{i}: G^i(F) \rightarrow \mathbb{C}^\times$ for $i=-1,\dots,d$ such that
\begin{enumerate}
\item[1)] $\theta = \prod\limits_{i=-1}^d\rep{i}|_{S(F)}$;
\item[2)] for all $0\leq i\leq d$, $\rep{i}$ is trivial on the simply connected cover of $G^i$;
\item[3)] for all $0\leq i < d$, $\rep{i}$ is a $G^{i+1}(F)$-generic character of depth $r_i$ in the sense of \cite[Definition 3.9]{HM:2008}. For $i=d$, $\rep{d}$ is trivial if $r_d = r_{d-1}$ and has depth $r_d$ otherwise. For $i=-1$, $\rep{-1}$ is trivial if $G^0=S$ and otherwise satisfies $\rep{-1}|_{S(F)_{0^+}}=1$.
\end{enumerate}
Given such a factorization, we set $\vec{\phi} = (\rep{0},\dots,\rep{d})$.

\paragraph{The point $y$:} Since $(S,\theta)$ is a tame $F$-non-singular elliptic pair, we have that $S$ is a maximally unramified elliptic maximal $F$-torus in the sense of \cite[Definition 3.4.2]{Kaletha:Regular}. As such, we can associate to it a vertex $y$ of $\mathcal{B}(G,F)$ \cite[Lemma 3.4.3]{Kaletha:Regular}, which is the unique $\Gal(F^\un/F)$-fixed point of $\mathcal{A}(G,S,F^\un)$ \cite[Section 17.8]{KalethaPrasad}.  

\paragraph{The representation $\rho$:} Let $\red$ denote the reductive
quotient of $G^0$ at $y$, that is the connected reductive $\res$-group
such that $\red(\resun) \simeq G^0(F^\un)_{y,0:0^+}$ and $\red(\res)
\simeq G^0(F)_{y,0:0^+}$. The construction of $\red$ is
summarized in
\cite[Section 8.4.2]{KalethaPrasad}. One starts with the relative
identity component $\red^0_y$ of a $\mathcal{O}_{F}$-group scheme
associated to $y$,  whose
existence is guaranteed by \cite[Proposition
 8.3.1 and Section 9.2.5]{KalethaPrasad}.  One then takes the special
fibre
$\overline{\red^0_y}$ of  $\red^0_y$, and defines $\red$ to be the
quotient by its unipotent radical, $\red :=
\overline{\red^0_y}/R_u(\overline{\red^0_y})$. By \cite[Theorem
 8.3.13]{KalethaPrasad}, $G^0(F^\un)_{y,0} =
\red^0_y(\mathcal{O}_{F^\un})$. The projection map
\begin{equation}
 \label{projmap}
G^0(F^\un)_{y,0} = \red^0_y(\mathcal{O}_{F^\un}) \rightarrow
\overline{\red^0_y}(\resun)
\end{equation}
is surjective, and the
preimage of $R_u(\overline{\red^0_y})(\resun)$ under this map is equal
to $G^0(F^\un)_{y,0^+}$ \cite[Corollary 8.4.12]{KalethaPrasad}, whence
$\red(\resun)\simeq G^0(F^\un)_{y,0:0^+}$.

There is a natural action of  $\Gal(F^\un/F)$ on
$\red^0_y(\mathcal{O}_{F^\un})$ and a natural action of
$\Gal(\resun/\res)$ on $\overline{\red^0_y}(\resun)$ and the map
(\ref{projmap}) is Galois-equivariant with respect to these actions.
The Galois-equivariance passes to the isomorphism
$G^0(F^\un)_{y,0:0^+}\simeq \red(\resun)$.

By \cite[Lemma 3.4.4]{Kaletha:Regular}, there exists an elliptic maximal $\res$-torus $\mathcal{S}$ of $\red$ such that for every unramified extension $F'$ of $F$, the image of $S(F')_0$ in $G(F')_{y,0:0^+}$ is isomorphic to $\mathcal{S}(\res')$. For every character $\overline{\chi}$ of $\mathcal{S}(\res)$, one can construct a virtual character of $\mathcal{G}(\res)$ as per \cite{DL:1976}, which we denote by $R_{\mathcal{S},\overline{\chi}}$. When $\overline{\chi}$ is non-singular in the sense of \cite[Definition 5.15]{DL:1976}, $\pm R_{\mathcal{S},\overline{\chi}}$ is a Deligne-Lusztig cuspidal representation of $\mathcal{G}(\res)$ \cite[Proposition 7.4, Theorem 8.3]{DL:1976}. Note that the sign $\pm$ refers to $(-1)^{r_{\res}(\red)-r_{\res}(\mathcal{S})}$, where $r_{\res}(\red)$ and $r_{\res}(\mathcal{S})$ denote the $\res$-split ranks of $\red$ and $\mathcal{S}$, respectively. By \cite[Lemma 3.4.14]{Kaletha:Regular}, the character $\rep{-1}|_{S(F)_0}$ factors through to a character $\overline{\rep{-1}}$ of $\mathcal{S}(\res)$ which is non-singular, meaning that the virtual character $\pm R_{\mathcal{S},\overline{\rep{-1}}}$ is a (possibly reducible) Deligne-Lusztig cuspidal representation of $\red(\res)$.  
The pullback of $\pm R_{\mathcal{S},\overline{\rep{-1}}}$ to $G^0(F)_{y,0}$ then gets extended uniquely to a representation $\upkappa_{(S,\rep{-1})}$ of $S(F)G^0(F)_{y,0}$ and $\rho\defeq \Ind_{S(F)G^0(F)_{y,0}}^{G^0(F)_{y}}\upkappa_{(S,\rep{-1})}$. This construction is summarized in Figure~\ref{fig:summaryRho}. Note that we are following the notation from \cite{Kaletha:Regular} in the paragraph above. What we have denoted by $\rho$ is denoted by $\kappa_{(S,\rep{-1})}$ in \cite[Section 3.3]{Kaletha:SCPackets}.

\begin{figure}[!htbp]
\begin{center}
\begin{tikzcd}[column sep = 0]
{\left( \pm R_{\mathcal{S},\overline{\rep{-1}}}, \red(\res) \right)}\arrow[mapsto]{rrrrrrrr}{\text{pullback, extend}} 
     &{} &{} &{} &{} &{} &{} &{} &{\left( \upkappa_{(S,{\rep{-1}})}, S(F)G^0(F)_{y,0} \right)}\arrow[mapsto]{rrrrr}{\text{induce}} &{} &{} &{} &{} &{\left( \rho, G^0(F)_{y} \right)} 
\end{tikzcd}
\end{center}

%
\caption{Summary of the construction of $\rho$.}
\label{fig:summaryRho}
\end{figure}

Once we have the $G$-datum $\Psi_{(S,\theta)} = (\vec{G},y,\vec{r},\rho,\vec{\phi})$, we apply the J.-K.~Yu construction to obtain the supercuspidal representation $\pi_{(S,\theta)}$. We do not recall all the details of this construction, but provide a summary in the form of a diagram (Figure~\ref{fig:summaryJKYu}). We invite the reader to consult \cite[Section 3]{thesisPaper} for a brief description of the steps involved. We point out that it is sometimes convenient to write $\kappa_G(\Psi_{(S,\theta)})$ for $\kappa_{(S,\theta)}$ and $\pi_G(\Psi_{(S,\theta)})$ for $\pi_{(S,\theta)}$ to indicate that we are applying the J.-K.~Yu construction on the $G$-datum $\Psi_{(S,\theta)}$.  

\begin{figure}[!htbp]
\center \begin{tikzpicture}
  \matrix (m) [matrix of math nodes,row sep=0.8em,column sep=0.9em,minimum width=0.1em]
  {
     {} & \text{$\left( \rep{0}, K^0\right)$}    & \text{$\cdots$} &\text{$\left(\rep{d-2},K^{d-2}\right)$}  & \text{$\left(\rep{d-1}, K^{d-1}\right)$}  & \text{$\left(\rep{d}, K^d\right)$}\\
        {} & {} & {} & {} & {}  &{} & {} \\
    \text{$\left(\rho, K^0\right)$} &\text{$\left(\prim{0}, K^1\right)$} & \text{$\cdots$} &\text{$\left(\prim{d-2},K^{d-1}\right)$}& \text{$\left(\prim{d-1},K^d\right)$}  &  \text{$\left(\prim{d}, K^d\right)$}\\
      {} & {} & {} & {}  & {}  & {}  & {} & {} \\
     \text{$\left(\kap{-1}, K^d\right)$} & \text{$\left(\kap{0},K^d\right)$}  & \text{$\cdots$} &\text{$\left(\kap{d-2},K^d\right)$}&  \text{$\left(\kap{d-1},K^d\right)$}  &  \text{$\left(\kap{d}, K^d\right)$}\\
    };
  \path[-stealth]
    (m-1-2) edge node [right] {extend} (m-3-2)
    (m-1-4) edge node [right] {extend} (m-3-4)
    (m-1-5) edge node [right] {extend} (m-3-5)
    (m-1-6) edge node [right] {=} (m-3-6)
    (m-3-1) edge node [right] {inflate} (m-5-1)
    (m-3-2) edge node [right] {inflate} (m-5-2)
    (m-3-4) edge node [right] {inflate} (m-5-4)
    (m-3-5) edge node [right] {=} (m-5-5)
    (m-3-6) edge node [right] {=} (m-5-6)
;
\end{tikzpicture}
$\pi_{(S,\theta)} \defeq \Ind_{K^d}^G\kappa_{(S,\theta)}$, where
$$\kappa_{(S,\theta)} = \kap{-1} \otimes \kap{0} \otimes \kap{1} \otimes \kap{2} \cdots \otimes \kap{d-2} \otimes \kap{d-1} \otimes \kap{d}$$
\caption{Summary of the J.-K.~Yu construction for $\pi_{(S,\theta)}$, where $K^0 = G^0(F)_{y}$ and ${K^{i+1} = K^0G^1(F)_{y,r_0/2}\cdots G^{i+1}(F)_{y,r_i/2}}, 0\leq i\leq d-1$.}
\label{fig:summaryJKYu}
\end{figure}

The representation $\rho$ above may be reducible, and its irreducible components are given by \cite[Theorem 2.7.7]{Kaletha:SCPackets}. While the definition of a $G$-datum in \cite{Yu:2001} requires $\rho$ to be irreducible, we may still apply the steps of the J.-K.~Yu construction on $(\vec{G},y,\vec{r},\rho,\vec{\phi})$ to obtain $\pi_{(S,\theta)}$, which is a completely reducible supercuspidal representation independent of the chosen Howe factorization \cite[Corollary 3.4.7]{Kaletha:SCPackets}. We use the notation $[\pi_{(S,\theta)}]$ for the set of irreducible components of $\pi_{(S,\theta)}$.

From the pair $(S,\theta)$, one may also perform what we call a \emph{twisted} J.-K.~Yu construction. Indeed, following \cite[Section 4.1]{FKS}, let $\epsilon = \prod\limits_{i=1}^d\epsilon^{G^i/G^{i-1}}$, where $\epsilon^{G^i/G^{i-1}}$is the quadratic character of $K^d$ that is trivial on $G^1(F)_{y,r_0/2}\cdots G^d(F)_{y,r_d/2}$ and whose restriction to $K^0$ is given by $\epsilon^{G^i/G^{i-1}}_y$ defined in \cite[Definition 4.1.10]{FKS}. The so-called twisted representation then refers to $\Ind_{K^d}^{G}(\kappa_{(S,\theta)}\cdot \epsilon)$, which is equivalent to constructing $\pi_{(S,\theta\cdot \epsilon)}$ via the above steps. 

\subsection{Computing the Decomposition}\label{sec:MainTheorem}

The goal of this section is to write the decomposition formula for $\pi_{(\ST,\thetaT)}$ (Theorem \ref{th:fullTheorem}). In order to describe this formula, we must first describe how we allow elements of $\GT(F)$ to act on representations of $G(F)$.

%

Let $\underT{g}\in\GT(F)$. Using the equality $\GT = \GZ\, Z(\GT)$, write $\underT{g} = \underZ{g}\,z$ for some ${\underZ{g} \in \GZ,} {z\in Z(\GT)}$. Since $\GZ$ is the image of $\eta$, there exists $g\in G$ such that $\underZ{g} = \eta(g)$. It follows that $\Ad(\underT{g}) = \Ad(\eta(g))$. We also set $\AdT(\underT{g}) := \Ad(g)$, an automorphism of $G$.

\begin{lemma}\label{lem:AdMap}
For all $\underT{g}\in\GT(F)$, $\AdT(\underT{g}) \in \mathrm{Aut}(G(F))$ is defined over $F$. Furthermore, the map 
\begin{align*}
\AdT : \GT(F) &\rightarrow \mathrm{Aut}(G(F)) \\
\underT{g} &\mapsto \AdT(\underT{g})
\end{align*}
is a well-defined homomorphism.
\end{lemma}

\begin{proof}
It is clear that $\AdT(\underT{g})$ maps $G$ to $G$. To conclude that $\AdT(\underT{g})$ is defined over $F$ (and therefore maps $G(F)$ to $G(F)$), we show that $\AdT(\underT{g})\circ \sigma = \sigma\circ \AdT(\underT{g})$ for all $\sigma\in\Gamma$. Recall that $\AdT(\underT{g}) = \Ad(g)$, where $g\in G$ is such that $\underT{g} = \eta(g)z$ for some $z\in Z(\GT)$. Since $\eta$ is defined over $F$ and $\underT{g} \in \GT(F)$, we have the following chain of equalities:
\begin{align*}
\eta\circ\sigma\circ\AdT(\underT{g}) 
&= \sigma\circ \Ad(\eta(g))\circ \eta \\
&= \sigma\circ \Ad(\underT{g}) \circ \eta \\
&= \Ad(\underT{g})\circ \eta \circ \sigma \\
&= \Ad(\eta(g))\circ \eta \circ \sigma \\
&= \eta\circ \AdT(\underT{g})\circ \sigma.
\end{align*}

Given $x\in G$, the previous equality implies $(\sigma\circ \AdT(\underT{g}))(x) = (\AdT(\underT{g})\circ \sigma)(x)z_x$ for some $z_x\in Z$. Define the map
\begin{align*}
f: G&\rightarrow Z \\
x &\mapsto z_x.
\end{align*}
This is a homomorphism, and is trivial on $Z(G)$. Furthermore, because $Z$ is abelian, $f$ is also trivial on $[G,G]$. Thus, $f$ is trivial on $G = [G,G]\,Z(G)$, and $z_x = 1$ for all $x\in G$. We conclude that $\AdT(\underT{g})\circ \sigma = \sigma \circ \AdT(\underT{g})$, as desired.
To show that the map $\AdT$ is well-defined, assume $\underT{g} = \eta(g_1)z_1 = \eta(g_2)z_2$, where $g_1,g_2\in G$, $z_1,z_2\in Z(\GT)$. It follows that ${\eta(g_1g_2^{-1}) = z_1^{-1}z_2 \in Z(\GT)\cap \GZ \subset Z(\GZ)}$, and therefore $g_1g_2^{-1}\in Z(G)$. We conclude that $\Ad(g_1) = \Ad(g_2)$, and thus $\AdT(\underT{g})$ is well defined. It is straightforward to show that $\AdT$ is a homomorphism.
\end{proof}

\begin{corollary}\label{cor:normal}
The group $\eta(G(F))$ is normal in $\GT(F)$.
\end{corollary}

\begin{proof}
Let $h\in G(F)$ and $\underT{g}\in\GT(F)$. We show that $\Ad(\underT{g})(\eta(h)) \in \eta(G(F)).$ Following the notation above, we have that $\Ad(\underT{g}) = \Ad(\eta(g))$ for some $g\in G$. It follows that $$\Ad(\underT{g})(\eta(h)) = \eta\circ\Ad(g)(h) = \eta\circ\AdT(\underT{g})(h).$$ By the previous lemma, $\AdT(\underT{g})$ is defined over $F$, which implies $\AdT(\underT{g})(h)\in G(F)$. Thus, we conclude that $\Ad(\underT{g})(\eta(h)) \in \eta(G(F)).$
\end{proof}

The following lemma will also be useful in proving the main statements of this section.

\begin{lemma}\label{lem:conjThrough}
Let $\underZ{\pi}$ be a representation of $\GZ(F)$ and $\underT{g}\in \GT(F)$. Then $$^{\underT{g}}\underZ{\pi}\circ\eta = \underZ{\pi}\circ\eta\circ\AdT(\underT{g}^{-1}).$$
\end{lemma}

\begin{proof}
We have that $\AdT(\underT{g}) = \Ad(g)$, where $g\in G$ satisfies $\underT{g} = \eta(g)z$ for some $z\in Z(\GT)$. For all $h\in G(F)$,
$$\begin{aligned}
\underZ{\pi}\circ \eta\circ \AdT(\underT{g}^{-1})(h) &= \underZ{\pi}\circ\eta(g^{-1}hg)\\
& = \underZ{\pi}(\eta(g)^{-1}\eta(h)\eta(g))\\
& = {^{\eta(g)}\underZ{\pi}}\circ \eta(h) \\
&= {^{\underT{g}}\underZ{\pi}}\circ\eta(h).
\end{aligned}$$
\end{proof}

We are now ready to state the decomposition formula of $\pi_{(\ST,\thetaT)}\circ\eta$.

\begin{theorem}\label{th:fullTheorem}
Let $(\ST,\thetaT)$ and $(S,\theta)$ be tame $F$-non-singular elliptic pairs for $\GT$ and $G$, respectively. Assume that $\eta(S) \subset \ST$ and $\theta = \thetaT\circ\eta$. Then 
$$\pi_{(\ST,\thetaT)}\circ\eta \simeq \underset{\underT{c}\in \underT{C}}{\oplus} \pi_{(S,\theta)} \circ \AdT(\underT{c}),$$ where $\underT{C}$ is a set of coset representatives of $\eta(G(F))\setminus\GT(F)/\ST(F)$.
\end{theorem}

The proof of Theorem \ref{th:fullTheorem} is done in two steps. Indeed, by noting that $$\pi_{(\ST,\thetaT)}\circ\eta = \left(\Res^{\GT(F)}_{\GZ(F)}\pi_{(\ST,\thetaT)}\right)\circ \eta,$$ we first seek a decomposition formula for $\Res^{\GT(F)}_{\GZ(F)}\pi_{(\ST,\thetaT)}$.
 The results from \cite{thesisPaper} grant us such a formula, as $\GZ$ is a normal subgroup of $\GT$ that contains $[\GT,\GT]$. This is stated in Theorem \ref{th:HG}. The second step is describing the composition of a supercuspidal representation of $\GZ(F)$ with $\eta$. This is given by Theorem \ref{th:mainTheorem}. The rest of this section is organized as follows. We start by presenting Theorem \ref{th:HG} and its proof. We then provide the statement of Theorem \ref{th:mainTheorem}  and prove Theorem \ref{th:fullTheorem}. The proof of Theorem \ref{th:mainTheorem}, being lengthy and involved, is postponed to its own subsection, Section \ref{sec:proofMainTheorem}. 

\begin{theorem}\label{th:HG}
Let $(S,\theta)$ be a tame $F$-non-singular elliptic pair of $G$ and let $y$ be the vertex of $\mathcal{B}(G,F)$ associated to $S$. Let $H$ be a closed connected $F$-subgroup of $G$ that contains $[G,G]$. Set $S_{H} = S\cap H$ and $\theta_H = \theta|_{S_H}$. Then $(S_H,\theta_H)$ is a tame $F$-non-singular elliptic pair of $H$ and
$$\pi_{(S,\theta)}|_{H(F)} = \underset{d\in D}{\oplus} {^d\pi_{(S_H,\theta_H)}},$$ where $D$ is a set of coset representatives of $H(F)\setminus G(F) / S(F)$.
\end{theorem}

We note that the results of \cite{thesisPaper} express the decomposition formula of $\pi_{(S,\theta)}|_{H(F)}$ in terms of a double sum over two sets of coset representatives. It turns out that one can combine these two sets to form a new set of coset representatives over a single quotient. This combination of coset representatives is discussed in Appendix \ref{sec:appendix}. We also make use of the following lemma.

\begin{lemma}\label{lem:redundancyParahoric}
Let $G$ be a reductive $F$-group, and $H$ be an $F$-subgroup that contains $[G,G]$. Let $S$ be a maximally unramified elliptic maximal torus of $G$, and let $S_H = S\cap H$, so that $S_H$ is a maximally unramified elliptic maximal torus of $H$. Denote by $y$ the point of the reduced building associated to $S$
 (and $S_H$) via \cite[Lemma 3.4.3]{Kaletha:Regular}. Let $G^0$ and $G^0_H$ denote the first Levi subgroups of the Levi sequences obtained from $S$ and $S_H$, respectively, and recall that $G^0_H = G^0\cap H$. Then
 $$G^0(F)_{y,0} = S(F)_0\, H^0(F)_{y,0}.$$
 \end{lemma}
 
\begin{proof}
By \cite[Definition 3.4.2]{Kaletha:Regular}, we have that $S$ is the centralizer of a maximal $F^{\un}$-split torus of $G$. Similarly, $S_H$ is the centralizer of a maximal $F^{\un}$-split torus of $H$. By definition (e.g. \cite[pp.35-36]{thesis} or \cite[Definition 13.2.1]{KalethaPrasad}),
\begin{align}\label{eq:parahoric} G^0(F^{\un})_{y,0} &=\langle S(F^{\un})_0, U_\alpha(F^{\un}):\alpha\in \Phi^{\mathrm{aff}}_{F^{\un}}, \langle \alpha,y \rangle \geq 0  \rangle, 
\end{align}
where $\Phi^{\mathrm{aff}}_{F^\un}$ is a set of affine roots and $U_\varphi(F^\un)$ is an associated affine root subgroup. The affine root subgroups are normalized by $S(F^\un)_0$, allowing us to write
\begin{align*}
G^0(F^\un)_{y,0} &= S(F^{\un})_0\, \langle S_H(F^{\un})_0, U_\alpha(F^{\un}):\alpha\in \Phi^{\mathrm{aff}}_{F^{\un}}, \langle \alpha,y \rangle \geq 0  \rangle \nonumber \\
&= S(F^{\un})_0\,H^0(F^{\un})_{y,0},\nonumber
\end{align*} 
and 
$$G^0(F)_{y,0} = (G^0(F^{\un})_{y,0})^{\Fr} = (S(F^{\un})_0\,H^0(F^{\un})_{y,0})^{\Fr}.$$ Using \cite[Lemma 3.4.6]{Kaletha:Regular} and the definition of $S_H$, we have that $S(F^{\un})_0\cap H^0(F^{\un})_{y,0} = S_H(F^{\un})_0$. Furthermore, $H^1(\Fr,S_H(F^{\un})_0)$ is trivial (see for instance the proof of \cite[Lemma 3.4.10]{Kaletha:Regular}). It follows from the usual sequence of Galois cohomology \cite[Proposition 12.3.4]{Springer:LAG} that 
$$(S(F^{\un})_0\,H^0(F^{\un})_{y,0})^{\Fr} = S(F^{\un})^{\Fr}_0\,(H^0(F^{\un})_{y,0})^{\Fr} = S(F)_0\,H^0(F)_{y,0},$$ and therefore
$$G^0(F)_{y,0} = S(F)_0\,H^0(F)_{y,0}.$$
 \end{proof}

We now have all the tools to prove Theorem \ref{th:HG}.

\begin{proof}[Proof of  Theorem \ref{th:HG}]
Let $\Psi_{(S,\theta)} = (\vec{G},y,\vec{r},\rho,\vec{\phi})$ be the $G$-datum obtained from the pair $(S,\theta)$ as in Section~\ref{sec:summaryKaletha}. Recall that we may write $\pi_{G}(\Psi_{(S,\theta)})$ for $\pi_{(S,\theta)}$ and $\kappa_{G}(\Psi_{(S,\theta)})$ for $\kappa_{(S,\theta)}$ to indicate that we are applying the J.-K.~Yu construction to $\Psi_{(S,\theta)}$. Set $K^i_H = K^i\cap H$ for all $0\leq i\leq d$ and $\Psi^{H}_{(S,\theta)} = (\vec{H},y,\vec{\tilde{r}}, \rho|_{K^0_H},\vec{\phi}_{H})$, where $\vec{H}, \vec{\tilde{r}}$ and $\vec{\phi}_H$ are as per \cite[Theorem 4.1]{thesisPaper}. Then, it follows from \cite[Theorems 5.7 and 5.8]{thesisPaper} that
$$\pi_{(S,\theta)}|_{H(F)} = \pi_G(\Psi_{(S,\theta)})|_{H(F)} \simeq \underset{l\in L}{\oplus} {^l\pi_H(\Psi^H_{(S,\theta)})},$$ where $L$ is a set of coset representatives of $H(F)\setminus G(F) / K^d$. It remains to compare $\Psi^H_{(S,\theta)}$ and $\Psi_{(S_H,\theta_H)}$. We have that $\vec{H}$ is the twisted Levi sequence associated to $S_H$ by \cite[Theorem 2.3]{thesisPaper} and the discussion preceding it, and the point $y$ is the vertex of $\mathcal{B}(H,F)$ associated to $S_H$ by \cite[Lemma 7.1]{thesisPaper}. The character sequence $\vec{\phi}_H$ clearly satisfies the first two axioms to be a Howe factorization of $\theta_H$, and genericity is given by \cite[Proposition 4.7]{thesisPaper}. Therefore, assembling these pieces along with the construction from Figure~\ref{fig:summaryRho}, we have $\Psi_{(S_H,\theta_H)} = (\vec{H},y,\vec{\tilde{r}},\Ind_{S_H(F)H^0(F)_{y,0}}^{K^0_H}\upkappa_{(S_H,\theta_H)},\vec{\phi}_H)$. Following the steps in the proof of \cite[Proposition 7.5]{thesisPaper}, we have
$$\rho|_{K^0_H} \simeq \underset{\ell\in \mathcal{L}}{\oplus} {^\ell \Ind_{S_H(F)H^0(F)_{y,0}}^{K^0_H}\upkappa_{(S_H,\theta_H)}},$$ where $\mathcal{L}$ is a set of coset representatives of $K^0_H\setminus K^0 / S(F)G^0(F)_{y,0}$. Therefore, we conclude that $$\pi_H(\Psi^H_{(S,\theta)}) \simeq \underset{\ell\in \mathcal{L}}{\oplus} {^\ell \pi_H(\Psi_{(S_H,\theta_H)})} = \underset{\ell\in \mathcal{L}}{\oplus}{^\ell\pi_{(S_H,\theta_H)}},$$ which implies
$$\pi_{(S,\theta)}|_{H(F)} \simeq \underset{l\in L}{\oplus}\underset{\ell \in \mathcal{L}}{\oplus} {^{l\ell}\pi_{(S_H,\theta_H)}}.$$

Using Lemma \ref{lem:redundancyParahoric}, one rewrites $\mathcal{L}$ as $K^0_H \setminus K^0 /S(F)$. We claim that $L\mathcal{L} = \{l\ell : l\in L,\ell\in \mathcal{L}\}$ is a set of coset representatives of $H(F)\setminus G(F) / S(F)$, which we denote by $D$, allowing us to write $$\pi_{(S,\theta)} \simeq \underset{d\in D}{\oplus} {^d\pi_{(S_H,\theta_H)}}.$$

To prove this last claim, we set $A = H(F),B=G(F),C=K^d,\bar{A}=K^0_H,\bar{B}=K^0$ and $\bar{C} = S(F)$, and show that $A,B,C,\bar{A},\bar{B},\bar{C}$ satisfy the hypotheses of part 1) of Lemma \ref{lem:combineCosets}. It is clear that $A$ and $\bar{A}$ are normal subgroups of $B$ and $\bar{B}$, respectively, and that $\bar{B}\subseteq C$ and $A\cap \bar{B} = \bar{A}$. It remains to show that $\bigslant{C}{(A\cap C)\bar{C}} \simeq \bigslant{\bar{B}}{\bar{A}\bar{C}}.$

Setting $J_H= H^1(F)_{y,r_0/2}\cdots H^d(F)_{y,r_{d-1}/2}$, we have $K^d_H = K^0_HJ_H$ by definition, and ${K^d = K^0J_H}$ as per \cite[Proof of Proposition 5.1]{thesisPaper}. It follows that 
\begin{align*}
\bigslant{C}{(A\cap C)\bar{C}} &= \bigslant{K^d}{K^d_HS(F)} = \bigslant{K^0J_H}{K^0_HJ_HS(F)}. 
\end{align*}
Given that $S(F)$ is in the stabilizer of $y$, we have that $^sH^i(F)_{y,r} = H^i(F)_{s\cdot y, r} = H^i(F)_{y,r}$ for all $s\in S(F), r\geq 0, 0\leq i\leq d$. Therefore, $S(F)$ normalizes $J_H$ and $J_HS(F) = S(F)J_H$. This, in combination with our modified third isomorphism theorem (Lemma \ref{lem:modifiedisothm}) allows us to obtain
$$\bigslant{C}{(A\cap C)\bar{C}} \simeq \bigslant{K^0}{K^0_HS(F)(K^0\cap J_H)} = \bigslant{K^0}{K^0_HS(F)},$$
where the last equality follows from the fact that $K^0\cap J_H \subset K^0_H$. Thus, ${\bigslant{C}{(A\cap C)\bar{C}}\simeq \bigslant{\bar{B}}{\bar{A}\bar{C}}}$.
\end{proof}

\begin{theorem}\label{th:mainTheorem}
Let $(S,\theta)$ and $(\SZ,\thetaZ)$ be tame $F$-non-singular elliptic pairs of $G$ and $\GZ$, respectively. Assume that $\eta(S) = \SZ$ and $\theta = \thetaZ\circ\eta$. Then
$$\pi_{(\SZ,\thetaZ)}\circ \eta \simeq \underset{\underZ{d}\in \underZ{D}}{\oplus} \pi_{(S,\theta)} \circ \AdT(\underZ{d}),$$ where $\underZ{D}$ is a set of coset representatives of $\eta(G(F))\setminus \GZ(F)/\SZ(F)$.
\end{theorem}

The strategy that we use to prove Theorem~\ref{th:mainTheorem} is similar to the one from \cite{thesisPaper} in the sense that we go through the constructions of $\pi_{(S,\theta)}$ and $\pi_{(\SZ,\thetaZ)}$ step-by-step to make comparisons along the way. It is a lengthy process, so we dedicate Section~\ref{sec:proofMainTheorem} to proving this statement. The comparison between the depth-zero parts is particularly delicate. With the statements of Theorems \ref{th:HG} and \ref{th:mainTheorem} in hand, we provide the proof of Theorem~\ref{th:fullTheorem}.

{\begin{proof}[Proof of Theorem~\ref{th:fullTheorem}]
Setting $G= \GT$ and $H=\GZ$ in Theorem \ref{th:HG}, we have
\begin{align*}
\pi_{(\ST,\thetaT)}\circ\eta &= \left(\Res^{\GT(F)}_{\GZ(F)}\pi_{(\ST,\thetaT)}\right) \circ\eta \\
&\simeq \left( \underset{\underT{c}\in \underT{C}}{\oplus} {^{\underT{c}}\pi_{(\SZ,\thetaZ)}} \right) \circ\eta,
\end{align*}
where $\SZ = \ST\cap \GZ, \thetaZ = \theta|_{\SZ}$ and $\underT{C}$ is a set of coset representatives of $\GZ(F)\setminus \GT(F)/\ST(F)$. 

By Lemma \ref{lem:conjThrough}, $^{\underT{c}}\pi_{(\SZ,\thetaZ)} \circ \eta = \pi_{(\SZ,\thetaZ)}\circ\eta\circ\AdT(\underT{c}^{-1})$. 
Using this last equality, and applying Theorem \ref{th:mainTheorem}, it follows that
\begin{align*}
\pi_{(\ST,\thetaT)}\circ\eta &\simeq \underset{\underT{c}\in \underT{C}}{\oplus} (\pi_{(\SZ,\thetaZ)}\circ \eta) \circ \AdT(\underT{c}^{-1})\\
&\simeq \underset{\underT{c}\in \underT{C}}{\oplus}  \left( \underset{\underZ{d}\in \underZ{D}}{\oplus}  \pi_{(S,\theta)}\circ \AdT(\underZ{d}) \right)\circ \AdT(\underT{c}^{-1}) \\
&= \underset{\underT{c}\in \underT{C}}{\oplus} \underset{\underZ{d}\in \underZ{D}}{\oplus}  \pi_{(S,\theta)}\circ \AdT(\underZ{d}\,\underT{c}^{-1})
\end{align*}
where $\underZ{D}$ is a set of coset representatives of $\eta(G(F))\setminus \GZ(F) /\SZ(F)$. Setting $A=\GZ(F), B=\GT(F), C = \ST(F), \bar{A} = \eta(G(F)), \bar{B} = \GZ(F), \bar{C} = \SZ(F)$, one sees from part 2) of Lemma \ref{lem:combineCosets} that $\{\underZ{d}\,\underT{c}^{-1} : \underT{c}\in \underT{C}, \underZ{d}\in \underZ{D}\}$ is a set of coset representatives of $\eta(G(F))\setminus \GT(F) /\ST(F)$. The decomposition formula thus follows.
\end{proof}}

\subsubsection{The Proof of Theorem~\ref{th:mainTheorem}}\label{sec:proofMainTheorem}

As previously mentioned, we will prove Theorem \ref{th:mainTheorem} by going through the constructions of $\pi_{(S,\theta)}$ and $\pi_{(\SZ,\thetaZ)}$ step-by-step. Given that the first step of the constructions is inducing J.-K.~Yu data from $(S,\theta)$ and $(\SZ,\thetaZ)$, we begin by establishing a relationship between both data.

\subsubsubsection{Matching the J.-K.~Yu Data}\label{sec:matchingInducedData}

In this section, $G, \GZ = \eta(G) \simeq G/Z$ and $\eta$ are as in the introduction. We let $(S,\theta)$ and $(\SZ,\thetaZ)$ be tame $F$-non-singular elliptic pairs of $G$ and $\GZ$, respectively, which satisfy $\eta(S) = \SZ$ and $\theta = \thetaZ\circ\eta$. The goal of this section is to show that the corresponding J.-K.~Yu data are also related via the map $\eta$, a statement we will make precise with Theorem~\ref{th:matchingData} and illustrate in Figure~\ref{fig:summaryJKData} below. First, we note that $\eta:G\rightarrow \GZ$ induces an equivariant isomorphism $\eta_\mathcal{B}: \mathcal{B}(G,F) \rightarrow \mathcal{B}(\GZ,F)$ by \cite[Axiom 4.1.1]{KalethaPrasad}. In particular, this map satisfies $\eta_{\mathcal{B}}(gx) = \eta(g)\cdot\eta_\mathcal{B}(x)$ for all $g\in G(F), x\in \mathcal{B}(G,F)$, where $\cdot$ refers to the action of $G(F)$ on $\mathcal{B}(G,F)$.

\begin{theorem}\label{th:matchingData}
Let $(S,\theta)$ and $(\SZ,\thetaZ)$ be tame $F$-non-singular elliptic pairs of $G$ and $\GZ$, respectively, such that $\eta(S) = \SZ$ and $\theta = \thetaZ\circ\eta$. Let $(\vec{G},y,\vec{r},\rho,\vec{\phi})$ and $(\vec{\GZ},\yZ,\vec{\underZ{r}},\rhoZ,\vec{\phiZ})$ be the corresponding J.-K.~Yu data as described in Section~\ref{sec:summaryKaletha}. Then, $\vec{r} = \vec{\underZ{r}}$, $\eta(\vec{G}) = \vec{\GZ}$, $\yZ = \eta_{\mathcal{B}}(y)$, $\vec{\phi} = \vec{\phiZ}\circ\eta$ and $\rhoZ\circ\eta \simeq \underset{\underZ{c}\in \underZ{C}}{\oplus}\rho\circ \AdT(\underZ{c}),$ where $\underZ{C}$ is a set of coset representatives of $\eta(K^0) \setminus \KZ{0}/\SZ(F)$.
\end{theorem}

\begin{figure}[!htbp]
\begin{center}
\begin{tikzcd}[column sep = 0]
{}  &{(\SZ,} &{} &{\thetaZ)}\arrow[bend right = 20, swap]{rrrrrrr}{\circ \eta} &{}  &{} &{} &{} &{(S,}\arrow[bend right = 20, swap]{lllllll}{\eta}  &{}&{\theta)} &{}\\[-30pt]
{} &{} &{}\arrow[squiggly]{d} &{} &{} &{} &{} &{} &{} &{}\arrow[squiggly]{d} &{} &{} \\
{} &{} &{} &{} &{} &{} &{} &{} &{} &{} &{}\\[-20pt]
{({\vec{\GZ}},} &{{\yZ},} &{{\vec{r}}},&{{\rhoZ}}\arrow[dashed,bend right = 20, swap]{rrrrrrr}{\circ \eta} &{\vec{\phiZ})}\arrow[bend right = 20, swap]{rrrrrrr}{\circ \eta} &{} &{} & {(\vec{G},} \arrow[bend right = 20, swap]{lllllll}{\eta} &{y,}\arrow[bend right = 20, swap]{lllllll}{\eta_\mathcal{B}}  &{\vec{r},} &{\rho,} &{\vec{\phi})}
\end{tikzcd}
\end{center}
\caption{Relationship between the corresponding J.-K.~Yu datum given the relationship between the tame $F$-non-singular elliptic pairs.}
\label{fig:summaryJKData}
\end{figure}

The proof of this theorem will be divided into four parts. Lemma~\ref{lem:LeviSeq} shows that ${\eta(\vec{G}) = \vec{\GZ}}$ and $\vec{r} = \vec{\underZ{r}}$. Lemma~\ref{lem:yyZ} gives us $\yZ = \eta_\mathcal{B}(y)$. Proposition~\ref{prop:HoweFact} allows us to set $\vec{\phi} = \vec{\phiZ}\circ\eta$. Finally, we obtain the decomposition formula of $\rhoZ\circ\eta$ from Proposition~\ref{prop:rhoeta}.

\begin{lemma}\label{lem:LeviSeq}
Let $\vec{G} = (G^0,\dots, G^d)$ and $\vec{\GZ} = (\GZ^0,\dots,\GZ^{\underZ{d}})$ be the twisted Levi sequences obtained from $S$ and $\SZ$, respectively, as per Section~\ref{sec:summaryKaletha}. Then $d = \underZ{d}$ and $\eta(G^i) = \GZ^i$ for all $0\leq i\leq d$.
\end{lemma}

The previous lemma easily follows from the fact that the root systems $R(G,S)$ and $R(\GZ,\SZ)$ are canonically identified. Furthermore, the induced sequence of numbers $\vec{r}$ is the same for both $S$ and $\SZ$. 

\begin{lemma}\label{lem:intersectGi}
For all $0\leq i\leq d$, $\GZ^i(F)\cap \eta(G(F)) = \eta(G^i(F)).$
\end{lemma}

\begin{proof}
Recall that $\GZ^i \simeq G^i/Z$. Given $\underZ{g}\in \GZ^i(F)$, we may think of $\underZ{g}$ as an element of the form $gZ$, where $g\in G^i$ is such that $\sigma(g)Z = gZ$ for all $\sigma\in\Gamma$. Assume that $\underZ{g}\in \eta(G(F))$ (so that $\underZ{g}\in \GZ^i(F)\cap \eta(G(F))$). Then $\underZ{g} = hZ$ for some $h\in G(F)$. It follows that $h=gz$ for some $z\in Z\subseteq Z(G) \subseteq G^i$. Therefore, $h\in G^i\cap G(F) = G^i(F),$ and thus $\underZ{g} = hZ \in \eta(G^i(F))$.
\end{proof}

\begin{lemma}\label{lem:yyZ}
Let $y$ be the vertex of $\mathcal{B}(G,F)$ associated to $S$ and $\yZ$ be the vertex of $\mathcal{B}(\GZ,F)$ associated to $\SZ$ as per \cite[Lemma 3.4.3]{Kaletha:Regular}. Then $\yZ = \eta_\mathcal{B}(y)$.
\end{lemma}

\begin{proof}
The bijection $\eta_\mathcal{B}$ maps $\mathcal{A}(G,S,F^{\un})$ to $\mathcal{A}(\GZ,\SZ,F^{\un})$. Recall that $y$ is fixed by $\Gal(F^\un/F)$ by definition. Since $\eta_{\mathcal{B}}$ is also defined over $F$, $\eta_{\mathcal{B}}(y)$ is a $\Gal(F^\un/F)$-fixed point. Since $\SZ$ is elliptic, this fixed point is unique \cite[Section 17.8]{KalethaPrasad}, and thus $\yZ = \eta_{\mathcal{B}}(y)$.
\end{proof}


Before moving on to the other parts of the data, it is also important to establish a relationship between the Moy-Prasad filtration subgroups of $G^i(F)$ and those of $\GZ^i(F)$, as they are crucial in the steps of the J.-K.~Yu construction. In particular, we have the following lemma.

\begin{lemma}\label{lem:filtrations}
For all $r>0$ we have $\eta(G(F)_{y,r}) = \GZ(F)_{\yZ,r}$.
\end{lemma}

\begin{proof}
Let $r>0$. Following the proof of \cite[Lemma 3.3.2]{Kaletha:Regular}, use \cite[Lemma 6.4.48]{BT:1972} to write $G(F)_{y,r}$ as the direct product of (topological spaces) $T(F)_r$ and the appropriate affine root subgroups. Here $T$ is a maximally unramified maximally split maximal torus, whose existence is guaranteed by \cite[Corollary 5.1.12]{BT:1984}. Since $\eta$ induces an isomorphism on the affine root subgroups, it suffices to show that $\eta(T(F)_r) = \underZ{T}(F)_r$, where $\underZ{T} = \eta(T) \simeq T/Z$. To do so, let $Z^\circ$ denote the identity component of $Z$. The map $\eta$ factors as follows:
\begin{center}
\begin{tikzcd}
{T}\arrow{rr}{\eta}\arrow{rd}{\eta^\circ} &{} &{\underZ{T}\simeq (T/Z^\circ)/(Z/Z^\circ)}\\
{} &{T/Z^\circ}\arrow{ru}{\overline{\eta}} &{}\\
\end{tikzcd}
\end{center}
We have that $Z^\circ$ is a torus by \cite[Theorem 16.2]{Humphreys:LAG} as it is a closed and connected subgroup of the torus $Z(G)^\circ$. By \cite[Lemma 3.1.3]{Kaletha:Regular}, we have an exact sequence 
$$1\rightarrow Z^\circ(F)_r \rightarrow T(F)_r \rightarrow (T/Z^\circ)(F)_r \rightarrow 1,$$ implying that $(T/Z^\circ)(F)_r \simeq T(F)_r/Z^\circ(F)_r\simeq \eta^\circ(T(F)_r)$. Furthermore, since $\overline{\eta}: T/Z^\circ \rightarrow \underZ{T}$ is an isogeny, \cite[Lemma 3.1.3]{Kaletha:Regular} tells us that $\overline{\eta}((T/Z^\circ)(F)_r) = \underZ{T}(F)_r$. Combining these two equations allows us to conclude that $\eta(T(F)_r) = \underZ{T}(F)_r$ which concludes the proof.
\end{proof}

\begin{remark}\label{rem:filtrations}
Given that $Z$ is a central subgroup of $G^i, 0\leq i\leq d$, we apply Lemma~\ref{lem:filtrations} and obtain $\eta(G^i(F)_{y,r_i}) = \GZ^i(F)_{\yZ,r_i}$. It follows that $\eta$ induces isomorphisms ${G^i(F)_{y,r_i:r_i^+}\simeq \GZ^i(F)_{\yZ,r_i:r_i^+}}$, $0\leq i\leq d$. Using a similar argument, we also have ${\eta(J^{i+1}) = \JZ{i+1}}$, $\eta(J^{i+1}_+) = \JZP{i+1}$ and $J^{i+1}/J^{i+1}_+\simeq \JZ{i+1}/\JZP{i+1}$ for all $0\leq i\leq d-1$, where $J^{i+1} = (G^i,G^{i+1})(F)_{y,(r_i,r_i/2)}$ and $J^{i+1}_+ =(G^i,G^{i+1})(F)_{y,(r_i,{r_i/2}^+)}$ as per \cite[Section 1]{Yu:2001}, and $\JZ{i+1}$ and $\JZP{i+1}$ are defined analogously.

At the depth-zero level, we can only guarantee an inclusion, that is ${\eta(G^0(F))_{y,0} \subset \GZ^0(F)_{\yZ,0}}$. This induces a homomorphism $G^0(F)_{y,0:0^+}\rightarrow \GZ^0(F)_{\yZ,0:0^+}$. This homomorphism is discussed in great detail after the statement of Proposition \ref{prop:rhoeta}. 
\end{remark}

\begin{lemma}\label{lem:etaKs}
Let $\KZ{0} = \GZ(F)_{\yZ}$ and $\KZ{i} = \KZ{0}\,\GZ^1(F)_{\yZ,r_0/2}\,\cdots\, \GZ^{i}(F)_{\yZ,r_{i-1}/2}$ for all ${0\leq i\leq d-1}$. Then $\eta(K^i) = \KZ{i} \cap \eta(G(F))$ for all $0\leq i\leq d$. Furthermore, $\eta(K^i)$ is normalized by $\KZ{0}$ for all $0\leq i\leq d$.
\end{lemma}

\begin{proof}
Let us begin with the case $i=0$. We start by showing that $\eta(K^0)$ is a normal subgroup of $\KZ{0}$. Let $k\in K^0$. Then $k\cdot y = y$ by definition. It follows that $\eta_\mathcal{B}(k\cdot y) = \eta_\mathcal{B}(y)$, or equivalently $\eta(k)\cdot \yZ = \yZ$. Thus $\eta(K^0) \subset \KZ{0}$. For normality, take $k\in K^0$ and $\underZ{k}\in \KZ{0}$. Since $\eta(G^0(F))$ is a normal subgroup of $\GZ^0(F)$, it follows that $\underZ{k}\,\eta(k)\,\underZ{k}^{-1} = \eta(g)$ for some $g\in G^0(F)$. By what precedes, we also have $\eta(g)\in \KZ{0}$, or equivalently $\eta(g)\cdot\yZ = \yZ$. Since the map $\eta_{\mathcal{B}}$ is a bijection, it follows that $g\cdot y = y$. Thus, $g\in K^0$ and $\underZ{k}\,\eta(k)\,\underZ{k}^{-1}\in \eta(K^0)$. For the intersection, it is clear that $\KZ{0}\cap \eta(G(F)) \supset \eta(K^0)$. Conversely, take $\underZ{k}\in \KZ{0}\cap \eta(G(F))$. In particular, $\underZ{k}\in \GZ^0(F) \cap \eta(G(F)) = \eta(G^0(F))$ (Lemma~\ref{lem:intersectGi}). Then $\underZ{k} = \eta(g)$ for some $g\in G^0(F)$  and $\eta(g)\cdot\yZ = \yZ$. Using the bijectivity of $\eta_{\mathcal{B}}$, it follows that $g\in K^0$, and thus $\underZ{k} \in \eta(K^0)$. 

Fix $0 < i \leq d$. Let ${J = G^1(F)_{y,r_0/2}\,\cdots\, G^i(F)_{y,r_{i-1}/2}}$ and ${\underZ{J} = \GZ^1(F)_{\yZ,r_0/2}\,\cdots\, \GZ^i(F)_{\yZ,r_{i-1}/2}}$. By what precedes in Remark~\ref{rem:filtrations}, we have $\eta(J) = \underZ{J}$. It follows that $\KZ{i} =\KZ{0}\,\underZ{J} = \KZ{0}\,\eta(J)$, and therefore $$\KZ{i} \cap \eta(G(F)) = (\KZ{0}\cap \eta(G(F)))\, \eta(J) = \eta(K^0)\,\eta(J) = \eta(K^i).$$
Using the fact that $^g\GZ^j(F)_{\yZ,r} = \GZ^j(F)_{g\cdot\yZ,r}$ for all $g\in \GZ^j(F)$, one sees that $\eta(J) = \underZ{J}$ is normalized by $\KZ{0}$. We conclude, with what precedes for $i=0$, that $\eta(K^i)$ is normalized by $\KZ{0}$. 
\end{proof}

\begin{corollary}\label{cor:AdK}
Let $0\leq i\leq d$. Then $\AdT(\underZ{k})(K^i) = K^i$ for all $\underZ{k}\in \KZ{0}$.
\end{corollary}

\begin{proof}
By the previous lemma, we have $\underZ{k}\,\eta(K^i)\,\underZ{k}^{-1} = \eta(K^i)$ for all $\underZ{k}\in \KZ{0}$. Given ${\underZ{k} \in \KZ{0}\subset \GZ^0(F)}$, $\underZ{k}= \eta(g)$ for some $g\in G^0$ such that $\sigma(g)g^{-1}\in Z$ for all $\sigma\in \Gamma$. It follows that $\eta(gK^ig^{-1}) = \eta(K^i)$, with $gK^ig^{-1} \subset G(F)$ (or equivalently, $\sigma(gK^ig^{-1}) = gK^ig^{-1}$ for all $\sigma\in\Gamma$). As a map on $G(F)$, the kernel of $\eta$ is $Z(F)$, which implies $gK^iZ(F)g^{-1} = K^iZ(F)$. By \cite[Lemma 3.3]{Yu:2001}, one has $K^0 = N_{G^0(F)}(G^0(F)_{y,0})$, hence $Z(F) \subset K^0 \subset K^i$. Thus, we conclude that $gK^ig^{-1} = K^i$, or equivalently, $\AdT(\underZ{k})(K^i) = K^i$.
\end{proof}

\begin{proposition}\label{prop:HoweFact}
Let $(\repZ{-1},\repZ{0},\dots,\repZ{d})$ be a Howe factorization for $\thetaZ$. For each ${-1\leq i\leq d}$, set $\rep{i} = \repZ{i}\circ \eta$. Then $(\rep{-1},\rep{0},\dots,\rep{d})$ is a Howe factorization for $\theta$.
\end{proposition}

\begin{proof}
One sees that $(\rep{-1},\rep{0},\dots,\rep{d})$ satisfies the two first axioms to be a Howe factorization of $\theta$, so it remains to verify the third axiom.
Let $0\leq i < d$. Verifying the genericity condition requires some additional notation. Following \cite[Section 3.1]{HM:2008}, we let $\z$ denote the centre of $\g^i = \mathrm{Lie}(G^i)$ and $\zstar$ its dual. We also set $\s = \mathrm{Lie}(S), \z_{r_i} = \mathfrak{z}^i(F)\cap \s(F)_{r_i}$ and define $$\mathfrak{z}^{i,*}_{-r_i} = \{X^*\in \zstar(F) : X^*(Y) \in \mathfrak{p}_F \text{ for all }Y\in \z_{r_i^+}\}.$$ We establish similar notation for $\GZ^i$ by adding subscript $Z$. 

Fix a character $\psi$ of $F$ which is nontrivial on $\mathcal{O}_F$ and trivial on $\mathfrak{p}_F$. By definition of genericity \cite[Definition 3.9]{HM:2008}, $\repZ{i}$ is a character of depth $r_i$, and its restriction to $\GZ^i(F)_{\yZ,r_i}$ is realized by a $\GZ^{i+1}(F)$-generic element of $\zZstarri$ of depth $-r_i$. That is, there exists a $\GZ^{i+1}(F)$-generic element $\underZ{X}^*\in \zZstarri$ of depth $-r_i$ in the sense of \cite[Definition 3.7]{HM:2008} such that $\repZ{i}\left(\eZ\left(Y+\gZ^i(F)_{\yZ,r_i^+}\right)\right) = \psi(X^*(Y))$ for all $Y\in \gZ^i(F)_{\yZ,r_i}$, where $$\eZ: \gZ^i(F)_{\yZ,r_i:r_i^+}\rightarrow \GZ^i(F)_{\yZ,r_i:r_i^+}$$ is the isomorphism from \cite{Adler:1998} referred to as the mock exponential map. Note that given our underlying hypothesis on $p$, one may simplify \cite[Definition 3.9]{HM:2008} to \cite[Definition 3.2]{thesisPaper} and omit so-called condition (GE2) from \cite[Definition 3.7]{HM:2008}.

To show that $\rep{i}$ is ${G^{i+1}}(F)$-generic of depth $-r_i$, we must show that $\rep{i}$ is trivial on ${G^i}(F)_{y,r_i^+}$ and that its restriction to ${G^i}(F)_{y,r_i}$ is realized by a ${G^{i+1}}(F)$-generic element of $\zstarri$ of depth $-r_i$. The quotient map $\eta: G^i \rightarrow \GZ^i$ induces a map on the Lie algebras $d\eta: \g^i \rightarrow \gZ^i$. Similarly to Remark~\ref{rem:filtrations}, we have that $d\eta(\g^i(F)_{y,r_i}) = \gZ^i(F)_{\yZ,r_i}$, which induces an isomorphism $\g^i(F)_{y,r_i:r_i^+}\simeq\gZ^i(F)_{\yZ,r_i:r_i^+}$ for all $0\leq i\leq d$. We also have $d\eta(\z) \subset \zZ$ and $d\eta(\z_{r_i}) \subset (\zZ)_{r_i}$. Therefore $d\eta$ induces a dual map
\begin{align*}
d\eta^*: \zZstar &\rightarrow \zstar\\
\underZ{X}^* &\mapsto \underZ{X}^*\circ d\eta,
\end{align*}
which satisfies $d\eta^*(\zZstarri) \subset \zstarri$.
Setting $X^* = d\eta^*(\underZ{X}^*)$, we then see from \cite[Definition 3.7]{HM:2008} that $X^*$ is a ${G^{i+1}}(F)$-generic element of $\zstarri$ of depth $-r_i$, as $R(G^i,S)$ canonically identifies with $R(\GZ^i, \SZ)$. We show that $X^*$ realizes $\rep{i}|_{{G^i}(F)_{y,r_i}}$. Let $e$ denote the mock exponential map from $\g^i(F)_{y,r_i:r_i^+}$ to ${G^i}(F)_{y,r_i:r_i^+}$. One can verify that the following diagram commutes.

\begin{center}
\begin{tikzcd}
{\g^i(F)_{y,r_i:r_i^+}}\arrow{d}{d\eta}\arrow{r}{e} &{{G^i}(F)_{y,r_i:r_i^+}}\arrow{d}{\eta}\\
{\gZ^i(F)_{\yZ,r_i:r_i^+}}\arrow{r}{\eZ} &{{\GZ^i}(F)_{\yZ,r_i:r_i^+}}\\
\end{tikzcd}
\end{center}
From this relationship between the mock exponential maps, it follows that for all $Y\in \g^i(F)_{y,r_i}$,
\begin{align*}
\rep{i}\left(e\left(Y+\g^i(F)_{y,r_i^+}\right)\right) &= \repZ{i}\circ\eta\left(e\left(Y+\g^i(F)_{y,r_i^+}\right)\right) \\
&=\repZ{i}\left(\eZ\left(d\eta(Y)+\gZ^i(F)_{\yZ,r_i^+}\right)\right) \\
&= \psi(\underZ{X}^*(d\eta(Y)))\\
&= \psi(X^*(Y)).
\end{align*}
Thus, we conclude that $\rep{i}$ is ${G^{i+i}}(F)$-generic of depth $r_i$.

For $i=d$, we see that $\rep{d}$ is trivial whenever $\repZ{d}$ is. When $\repZ{d} \neq 1$, $\rep{d}$ must be of the same depth, as ${G}(F)_{y,r_d:r_d^+}\simeq \GZ(F)_{\yZ,r_d:r_d^+}$.

Finally, for $i=-1$, it is clear that $\rep{-1}$ is trivial in the case where $\repZ{-1}$ is trivial, and that $\rep{-1}|_{S(F)_{0^+}} = 1$ whenever $\repZ{-1}|_{\SZ(F)_{0^+}} = 1$ as $\eta(S(F)_{0^+})\subset \SZ(F)_{0^+}$.

\end{proof}

\begin{proposition}\label{prop:rhoeta}
Let $\rho$ and $\rhoZ$ be the representations of $G^0(F)_{y}$ and $\GZ^0(F)_{\yZ}$ constructed from $(S,\theta)$ and $(\SZ,\thetaZ)$, respectively, as per Section~\ref{sec:summaryKaletha}. Then $$\rhoZ\circ\eta = \underset{\underZ{c}\in \underZ{C}}{\oplus}\rho\circ \AdT(\underZ{c}),$$ where $\underZ{C}$ is a set of coset representatives of $\eta(K^0) \setminus \KZ{0}/\SZ(F)$.
\end{proposition}

Note that by Corollary~\ref{cor:AdK}, we have that $\AdT(\underZ{c})(K^0) = K^0$ for all $\underZ{c}\in \KZ{0}$. Therefore, the direct sum decomposition above makes sense as a representation of $K^0$. In order to prove this proposition, we will need to introduce additional notation and state some intermediate results.

Recall from Figure~\ref{fig:summaryRho} that $\rho = \Ind_{S(F)G^0(F)_{y,0}}^{G^0(F)_{y}}\upkappa_{(S,\rep{-1})}$, where $\upkappa_{(S,\rep{-1})}$ is constructed from the Deligne-Lusztig cuspidal representation $\pm R_{\mathcal{S},\overline{\rep{-1}}}$ of $\red(\res)$ with $\mathcal{S}$ a maximal torus of $\red$ which satisfies $\mathcal{S}(\res) \simeq S(F)_{0:0^+}$. Adopting similar notation for $\GZ^0$, we have that $\rhoZ = \Ind_{{\GZ^0}(F)_{\yZ,0}}^{{\GZ^0}(F)_{\yZ}}\upkappa_{(\SZ,{\repZ{-1}})}$, where $\upkappa_{(\SZ,\repZ{-1})}$ is constructed from the Deligne-Lusztig cuspidal representation $\pm R_{\SSZ,\overline{\repZ{-1}}}$ of $\redZ(\res)$, the reductive quotient of ${\GZ^0}$ at $\yZ$, with $\SSZ$ a maximal torus of $\redZ$ which satisfies $\SSZ(\res) \simeq \SZ(F)_{0:0^+}$. 

To understand the relationship between $\rho$ and $\rhoZ$, we must understand the relationship between $\pm R_{\mathcal{S},\overline{\rep{-1}}}$ and $\pm R_{\SSZ,\overline{\repZ{-1}}}$. We first describe how $\eta$ induces a map between $\red$ and $\redZ$. Recall from Remark \ref{rem:filtrations} that $\eta$ induces a homomorphism $G^0(F^\un)_{y,0}\rightarrow \GZ^0(F^\un)_{\yZ,0}$, which we can compose with the quotient map to obtain a homomorphism
$$\begin{aligned}G^0(F^\un)_{y,0}&\rightarrow \GZ^0(F^\un)_{\yZ,0:0^+}\\
g&\mapsto \eta(g)\GZ^0(F^\un)_{y,0^+}.\end{aligned}$$ The kernel of this homomorphism is $(Z\cap G^0(F^\un)_{y,0})\,G^0(F^\un)_{y,0^+},$ 
resulting in an embedding
$$\bigslant{G^0(F^\un)_{y,0}}{(Z\cap G^0(F^\un)_{y,0})\,G^0(F^\un)_{y,0^+}}\hookrightarrow \GZ^0(F^\un)_{\yZ,0:0^+}.$$ By the third isomorphism theorem, the domain of the previous embedding is isomorphic to
$$\bigslant{G^0(F^\un)_{y,0:0^+}}{\Big(\bigslant{ (Z\cap G^0(F^\un)_{y,0})G^0(F^\un)_{y,0^+}}{G^0(F^\un)_{y,0^+}}\Big)}.$$ Given that $Z\subset Z(G)$, it follows that $\bigslant{(Z\cap G^0(F^\un)_{y,0})\,G^0(F^\un)_{y,0^+}}{G^0(F^\un)_{y,0^+}}$ is a closed central subgroup of $G^0(F^\un)_{y,0:0^+}$. As such, it corresponds to a closed central subgroup of $\red(\resun)$, which we denote by $\mathcal{Z}(\resun)$. Given that we can identify reductive groups with their $\resun$-points, what we have just described is an embedding 
$$
\overline{\eta}: \red/\mathcal{Z} \hookrightarrow \redZ.$$ 
Note that this map is $\mathrm{Gal}(\resun/\res)$-equivariant, as $\eta$ is $\mathrm{Gal}(F^\un/F)$-equivariant. Furthermore, $\mathcal{Z}$ is defined over $\res$ by \cite[Corollary 12.1.3]{Springer:LAG}. Thus, the map between $\red$ and $\redZ$ (which is defined over $\res$) is given by
\begin{center}
\begin{tikzcd}[column sep = 0]
\red\arrow{rrr}{q}\arrow[leftrightarrow]{d}{\simeq}  &{} &{} &\red/\mathcal{Z}\arrow[hookrightarrow]{rrr}{\overline{\eta}}\arrow[leftrightarrow]{d}{\simeq} &{} &{} &\redZ\arrow[leftrightarrow]{d}{\simeq} \\
 G^0(F^\un)_{y,0:0^+}\arrow{rrr} &{} &{}   &\bigslant{G^0(F^\un)_{y,0}}{(Z\cap G^0(F^\un)_{y,0})G^0(F^\un)_{y,0^+}}\arrow[hookrightarrow]{rrr} &{} &{} &{\GZ^0(F^\un)_{\yZ,0:0^+}}  \\
 [-20pt] gG^0(F^\un)_{y,0^+} \arrow[mapsto]{rrrrrr} &{} &{} &{} &{} &{} &\eta(g)\GZ^0(F^\un)_{\yZ,0^+},
\end{tikzcd}
\end{center}
where $q$ is the obvious quotient map. To alleviate notation, we will keep the isomorphisms implicit and say that an element of $\red(\resun)$ (or $\red(\res)$) is of the form $gG^0(F^\un)_{y,0^+}$ for some $g\in G^0(F^\un)_{y,0}$ (or $gG^0(F)_{y,0^+}$ for some $g\in G^0(F)_{y,0}$). We start by proving two elementary results involving the maps $\overline{\eta}$ and $q$.

\begin{lemma}\label{lem:etaTori}
Using the above notation, one has $\overline{\eta}\circ q(\mathcal{S})=\SSZ \cap \overline{\eta}(\red/\mathcal{Z})$.
\end{lemma}

\begin{proof}
We identify the reductive groups with their $\resun$-points. Based on the definitions above, we have
$$
\begin{aligned}
\overline{\eta}&\circ q(\mathcal{S}(\resun)) = \overline{\eta}\circ q\left( \bigslant{S(F^\un)_0G^0(F^\un)_{y,0^+}}{G^0(F^\un)_{y,0^+}} \right) \\
&= \bigslant{\eta(S(F^\un)_0)\GZ^0(F^\un)_{\yZ,0^+}}{\GZ^0(F^\un)_{\yZ,0^+}} \\
&\subseteq \Big(\bigslant{\SZ(F^\un)_0\GZ^0(F^\un)_{\yZ,0^+}}{\GZ^0(F^\un)_{\yZ,0^+}} \Big) \bigcap \Big( \bigslant{\eta(G^0(F^\un)_{y,0})\GZ^0(F^\un)_{\yZ,0^+}}{\GZ^0(F^\un)_{\yZ,0^+}}\Big) \\
&= \SSZ(\resun) \cap \overline{\eta}(\red/\mathcal{Z})(\resun).
\end{aligned}$$ Given that both tori are maximal, the equality follows.
\end{proof}

\begin{lemma}\label{lem:redQuotients}
Let $\overline{\eta}$ and $\mathcal{Z}$ be as above. Then $\overline{\eta}(\red/\mathcal{Z}) \supseteq [\redZ,\redZ]$.
\end{lemma}

\begin{proof}
We identify the reductive groups with their $\resun$-points. Based on the definitions above, we have
$$\overline{\eta}(\red(\resun)/\mathcal{Z}(\resun)) = \bigslant{\eta(G^0(F^\un)_{y,0})\GZ^0(F^\un)_{\yZ,0^+}}{\GZ^0(F^\un)_{\yZ,0^+}}.$$
Since $\GZ^0(F^{\un})_{\yZ,0}$ is generated by $\SZ(F^{\un})_0$ and certain root subgroups (cf. (\ref{eq:parahoric})), and root subgroups are normalized by toral elements, it follows that $[\GZ^0(F^{\un})_{\yZ,0},\GZ^0(F^{\un})_{\yZ,0}]$ consists only of products of root subgroup elements. Since $\eta$ induces an isomorphism on the affine root subgroups, we conclude that $[\GZ^0(F^{\un})_{\yZ,0},\GZ^0(F^{\un})_{\yZ,0}] \subseteq \eta(G^0(F^\un)_{y,0})$. It follows that $$[\redZ(\resun),\redZ(\resun)] = \bigslant{[\GZ^0(F^\un)_{\yZ,0},\GZ^0(F^\un)_{\yZ,0}]\GZ^0(F^\un)_{\yZ,0^+}}{\GZ^0(F^\un)_{\yZ,0^+}}\subseteq \overline{\eta}\left(\redZ(\resun)/\mathcal{Z}(\resun)\right).$$
\end{proof}

\begin{lemma}\label{lem:redundancyParahoricQuotient}
We have the equality $\GZ^0(F)_{y,0} = \SZ(F)_0\,\eta(G^0(F)_{y,0})$.
\end{lemma}

\begin{proof}
Using Lemma \ref{lem:redQuotients}, we have that $\redZ = \SSZ \,\overline{\eta}(\red/\mathcal{Z})$. Furthermore, the intersection of $\SSZ$ with $\overline{\eta}(\red/\mathcal{Z})$ is a maximal torus of $\overline{\eta}(\red/\mathcal{Z})$. As a consequence of Lang's theorem, $H^1(\res,\SSZ\cap \overline{\eta}(\red/\mathcal{Z})) = 1$. Combining this with the usual Galois cohomology sequence, $${\redZ(\res) = \SSZ(\res)\,\overline{\eta}(\red/\mathcal{Z})(\res) = \SSZ(\res)\,\overline{\eta}\left(\left({\red}/{\mathcal{Z}}\right)(\res)\right)}.$$ We have that $\left({\red}/{\mathcal{Z}}\right)(\res) = C \left( \red(\res)/\mathcal{Z}(\res)\right),$ where $C$ is a set of coset representatives of $\bigslant{(\red/\mathcal{Z})(\res)}{\left( \red(\res)/\mathcal{Z}(\res) \right)}$. Using Lemma \ref{lem:cohomology}, we may assume without loss of generality that $C\subseteq q(\mathcal{S})(\res)$. Therefore,
$$\redZ(\res) = \SSZ(\res)\, \overline{\eta}(C)\, \overline{\eta}(\red(\res)/\mathcal{Z}(\res)) = \SSZ(\res)\, \overline{\eta}(\red(\res)/\mathcal{Z}(\res)).$$ Given that $\red(\res) = G^0(F)_{y,0:0^+}$, it follows that $\mathcal{Z}(\res) = \bigslant{\left(Z\cap G^0(F)_{y,0}\right)G^0(F)_{y,0^+}}{G^0(F)_{y,0^+}}$. Therefore,
$$\overline{\eta}(\red(\res)/\mathcal{Z}(\res)) = \bigslant{\eta(G^0(F)_{y,0})\GZ^0(F)_{y,0^+}}{\GZ^0(F)_{y,0^+}}.$$ Using the identifications $\redZ(\res) \simeq \GZ^0(F)_{y,0:0^+}$ and $\SSZ(\res) \simeq \SZ(F)_0\GZ^0(F)_{y,0^+}/\GZ^0(F)_{y,0^+}$, the above equality over the $\res$-points implies the following equation
$$\GZ^0(F)_{y,0:0^+} =\left( \bigslant{\SZ(F)_0\GZ^0(F)_{y,0^+}}{\GZ^0(F)_{y,0^+}}\right) \left( \bigslant{\eta(G^0(F)_{y,0})\GZ^0(F)_{y,0^+}}{\GZ^0(F)_{y,0^+}}\right),$$ from which we conclude that
$$\GZ^0(F)_{y,0} = \SZ(F)_0\,\eta(G^0(F)_{y,0})\,\GZ^0(F)_{y,0^+}.$$ Finally, we have that $\GZ^0(F)_{y,0^+} = \eta(G^0(F)_{y,0^+})$. Indeed, using Lemma \ref{lem:filtrations}, we have
$$\GZ^0(F)_{y,0^+} = \underset{r > 0}{\cup}\GZ^0(F)_{y,r} = \underset{r > 0}{\cup}\eta(G^0(F)_{y,0}) = \eta\left(\underset{r > 0}{\cup} G^0(F)_{y,0}\right) = \eta(G^0(F)_{y,0^+}).$$ Thus, we conclude that $$\GZ^0(F)_{y,0} = \SZ(F)_0\,\eta(G^0(F)_{y,0}).$$
\end{proof}

\begin{remark}
One can apply an argument similar to the one in the proof of Lemma \ref{lem:redundancyParahoric} to obtain $\GZ^0(F^\un)_{y,0} = \SZ(F^\un)_0\,\eta(G^0(F^\un)_{y,0})$. However, we do not see a quick Galois cohomology argument to get the equality at the level of the $F$-points.
\end{remark}

Using the map $\overline{\eta}\circ q$, and Lemma \ref{lem:redQuotients}, we can establish the following relationship between the virtual characters $\pm R_{\mathcal{S},\overline{\rep{-1}}}$ and $\pm R_{\SSZ,\overline{\repZ{-1}}}$.

\begin{proposition}\label{prop:virtualChars}
Given the above notation, one has $\pm R_{\mathcal{S},\overline{\rep{-1}}} = \pm R_{\SSZ,\overline{\repZ{-1}}} \circ (\overline{\eta}\circ q).$
\end{proposition}

\begin{proof}
We start by showing that $R_{\mathcal{S},\overline{\rep{-1}}} = R_{\SSZ,\overline{\repZ{-1}}} \circ (\overline{\eta}\circ q).$
Let $\bar{g}\in \red(\res)$. One has
$$R_{\SSZ,\overline{\repZ{-1}}} \circ (\overline{\eta}\circ q)(\bar{g}) = R_{\SSZ,\overline{\repZ{-1}}}|_{\overline{\eta}(\red/\mathcal{Z})(\res)} (\overline{\eta}\circ q(\bar{g})).$$ By Lemma \ref{lem:redQuotients}, one has $[\redZ,\redZ]\subseteq \overline{\eta}(\red/\mathcal{Z})$. Thus, we can apply \cite[Theorem A.2]{thesisPaper} to obtain
$$R_{\SSZ,\overline{\repZ{-1}}} \circ (\overline{\eta}\circ q)(\bar{g}) = R_{\SSZ\cap \overline{\eta}(\red/\mathcal{Z}),\overline{\repZ{-1}}|_{\SSZ\cap \overline{\eta}(\red/\mathcal{Z})}}(\overline{\eta}\circ q(\bar{g})).$$ Given that virtual characters remain constant under isomorphism (see e.g. \cite[Corollary 6.18]{thesis}), we modify the right-hand side of the previous equation accordingly and obtain
$$R_{\SSZ,\overline{\repZ{-1}}} \circ (\overline{\eta}\circ q)(\bar{g}) = R_{\overline{\eta}^{-1}(\SSZ\cap \overline{\eta}(\red/\mathcal{Z})),\overline{\repZ{-1}}|_{\SSZ\cap \overline{\eta}(\red/\mathcal{Z})}\circ \overline{\eta}}(q(\bar{g})).$$ By Lemma \ref{lem:etaTori}, we have $\overline{\eta}^{-1}(\SSZ\cap \overline{\eta}(\red/\mathcal{Z})) = q(\mathcal{S})$. Next, we claim that ${\overline{\repZ{-1}}|_{\SSZ\cap \overline{\eta}(\red/\mathcal{Z})}\circ \overline{\eta} = \overline{\rep{-1}}.}$ Indeed, the restriction $\overline{\eta}\circ q|_{\red(\res)}$ maps $gG^0(F)_{y,0^+}$ to $\eta(g)\GZ^0(F)_{y,0^+}$ for all $g\in G^0(F)_{y,0}$, and therefore
$$\overline{\repZ{-1}}\circ\overline{\eta}\circ q (gG^0(F)_{y,0^+}) = \overline{\repZ{-1}}(\eta(g)\GZ^0(F)_{y,0^+}) = \repZ{-1}(\eta(g)).$$ As in Proposition \ref{prop:HoweFact}, we set $\repZ{-1}\circ \eta = \rep{-1}$, which allows us to conclude that 
$$\overline{\repZ{-1}}\circ\overline{\eta}\circ q (gG^0(F)_{y,0^+}) = \rep{-1}(g) = \overline{\rep{-1}}(gG(F)_{y,0^+}).$$ Getting back to our virtual character expression, we have
$$R_{\SSZ,\overline{\repZ{-1}}} \circ (\overline{\eta}\circ q)(\bar{g}) = R_{q(\mathcal{S}),\overline{\repZ{-1}}|_{\SSZ\cap \overline{\eta}(\red/\mathcal{Z})}\circ \overline{\eta}}(q(\bar{g})),$$ with $\overline{\repZ{-1}}|_{\SSZ\cap \overline{\eta}(\red/\mathcal{Z})}\circ \overline{\eta} = \overline{\rep{-1}}.$ Thus, we apply Theorem \ref{th:virtualCharQuotient} to the right-hand side of the above equation and obtain
$$R_{\SSZ,\overline{\repZ{-1}}} \circ (\overline{\eta}\circ q)(\bar{g}) = R_{\mathcal{S},\overline{\rep{-1}}}(\bar{g}).$$

Now, recall that $\pm R_{\mathcal{S},\overline{\rep{-1}}} = (-1)^{r_{\res}(\red)-r_{\res}(\mathcal{S})}R_{\mathcal{S},\overline{\rep{-1}}},$ where $r_{\res}(\red)$ and $r_{\res}(\mathcal{S})$ denote the $\res$-split ranks of $\red$ and $\mathcal{S}$, respectively. Thus, the proof will be complete once we show that ${r_{\res}(\red)-r_{\res}(\mathcal{S})} = {r_{\res}(\redZ)-r_{\res}(\SSZ)}$. By \cite[Corollary 22.7]{Borel:LAG}, we have that $$r_{\res}(\red) - r_{\res}(\mathcal{S}) = r_{\res}(\red/\mathcal{Z}) - r_{\res}(\mathcal{S}/\mathcal{Z}).$$ Given that the restriction of $\overline{\eta}$ to its image is an isomorphism, it follows that 
$$r_{\res}(\red) - r_{\res}(\mathcal{S}) = r_{\res}(\overline{\eta}(\red/\mathcal{Z})) - r_{\res}(\overline{\eta}(\mathcal{S}/\mathcal{Z})) = r_{\res}(\overline{\eta}(\red/\mathcal{Z}))  - r_{\res}(\SSZ\cap \overline{\eta}(\red/\mathcal{Z})),$$ where the last equality follows from Lemma \ref{lem:etaTori}. Since $\overline{\eta}(\red/\mathcal{Z}) \supseteq [\redZ,\redZ]$ (Lemma \ref{lem:redQuotients}), we have
$$r_{\res}(\overline{\eta}(\red/\mathcal{Z})) - r_{\res}(\SSZ\cap \overline{\eta}(\red/\mathcal{Z})) = r_{\res}(\redZ) - r_{\res}(\SSZ)$$ as a consequence of \cite[Proposition 4.27]{BorelTits:1965} (see for instance \cite[Proposition 2.1.45]{thesis}). Thus, we conclude that
$${r_{\res}(\red)-r_{\res}(\mathcal{S})} = {r_{\res}(\redZ)-r_{\res}(\SSZ)}.$$
\end{proof}

\begin{corollary}\label{cor:kappaeta}
Let $\upkappa_{(S,\rep{-1})}$ and $\upkappa_{(\SZ,\repZ{-1})}$ be as defined earlier in this section. Then $\upkappa_{(S,\rep{-1})} = \upkappa_{(\SZ,\repZ{-1})} \circ \eta.$
\end{corollary}

\begin{proof}
Given a representation $\overline{\chi}$ of a quotient group, let $_p\overline{\chi}$ denote its pullback. In light of Proposition \ref{prop:virtualChars}, we have 
$$_p{R_{\mathcal{S},\overline{\rep{-1}}}} = {_p({R_{\SSZ,\overline{\repZ{-1}}}\circ \overline{\eta}\circ q})}.$$ It is simple to verify $$_p({R_{\SSZ,\overline{\repZ{-1}}}\circ \overline{\eta}\circ q}) = {_p{R_{\SSZ,\overline{\repZ{-1}}}}}\circ \eta.$$ 
It follows that $\upkappa_{(\SZ,\repZ{-1})}\circ \eta |_{G^0(F)_{y,0}} = \upkappa_{(S,\rep{-1})}|_{G^0(F)_{y,0}}$, so that $\upkappa_{(\SZ,\repZ{-1})}\circ \eta$ is an extension of $_p{R_{\mathcal{S},\rep{-1}}}$. Given that the extensions are unique, 
we conclude that $\upkappa_{(S,\rep{-1})} = \upkappa_{(\SZ,\repZ{-1})}\circ\eta.$
\end{proof}

%

We now have all the tools to prove Proposition \ref{prop:rhoeta}.

\begin{proof}[Proof of Proposition~\ref{prop:rhoeta}]
From the Mackey Decomposition, we have
\begin{align*}
\rhoZ\circ \eta &= \left( \Ind_{\SZ(F)\GZ^0(F)_{\yZ,0}}^{\KZ{0}}\upkappa_{(\SZ,\phiZ^{-1})} \right) \circ\eta \\
&= \left( \Res_{\eta(K^0)}^{\KZ{0}}\Ind_{\SZ(F)\GZ^0(F)_{\yZ,0}}^{\KZ{0}}\upkappa_{(\SZ,\phiZ^{-1})} \right) \circ\eta \\
&= \underset{\underZ{c}\in \underZ{C}}{\oplus}\left[ \left( \Ind_{\eta(K^0)\cap {^{\underZ{c}}(\SZ(F)\GZ^0(F)_{\yZ,0})}}^{\eta(K^0)} \Res_{\eta(K^0)\cap {^{\underZ{c}}(\SZ(F)\GZ^0(F)_{\yZ,0})}}^{^{\underZ{c}}(\SZ(F)\GZ^0(F)_{\yZ,0})} {^{\underZ{c}}\upkappa_{(\SZ,\phiZ^{-1})}}\right) \circ \eta \right],
\end{align*}
where $\underZ{C}$ is a set of coset representatives of $\eta(K^0)\setminus \KZ{0}/ \SZ(F)\GZ^0(F)_{\yZ,0}$, which is equal to $\eta(K^0)\setminus \KZ{0}/\SZ(F)$ by Lemma \ref{lem:redundancyParahoricQuotient}. Given that $\eta(K^0)$ is a normal subgroup of $\KZ{0}$ (Lemma~\ref{lem:etaKs}) and that $\eta(K^0) \cap \SZ(F)\GZ^0(F)_{\yZ,0} = \eta(S(F)G^0(F)_{y,0})$ (as a consequence of Lemma~\ref{lem:redundancyParahoricQuotient}), it follows that 
\begin{align*}
\eta(K^0)\cap {^{\underZ{c}}(\SZ(F)\GZ^0(F)_{\yZ,0})} &= {^{\underZ{c}}(\eta(K^0)\cap \SZ(F) \GZ^0(F)_{\yZ,0})} \\ &= {^{\underZ{c}}\eta(S(F)G^0(F)_{y,0})} \\
&= \eta(\AdT(\underZ{c})(S(F)G^0(F)_{y,0})).
\end{align*}
As a result, we may simplify the above expression and apply Proposition \ref{prop:induction} to obtain 
\begin{align*}
\rhoZ\circ\eta &= \underset{\underZ{c}\in \underZ{C}}{\oplus} \left[ \left( \Ind_{\eta(\AdT(\underZ{c})(S(F)G^0(F)_{y,0}))}^{\eta(K^0)}\Res_{\eta(\AdT(\underZ{c})(S(F)G^0(F)_{y,0}))}^{^{\underZ{c}}(\SZ(F)\GZ^0(F)_{\yZ,0})}{^{\underZ{c}}\upkappa_{(\SZ,\phiZ^{-1})} }\right) \circ \eta \right] \\
&\simeq \underset{\underZ{c}\in \underZ{C}}{\oplus}\Ind_{\AdT(\underZ{c})(S(F)G^0(F)_{y,0})}^{\AdT(\underZ{c})(K^0)}\left( {^{\underZ{c}}\upkappa_{(\SZ,\phiZ^{-1})}}\circ\eta \right).
\end{align*}
Applying Lemma \ref{lem:conjThrough}, followed by Proposition \ref{prop:kappaeta} on the previous expression, we obtain
\begin{align*}
\rhoZ\circ\eta &\simeq \underset{\underZ{c}\in \underZ{C}}{\oplus}\Ind_{\AdT(\underZ{c})(S(F)G^0(F)_{y,0})}^{\AdT(\underZ{c})(K^0)}\left[\left( \upkappa_{(\SZ,\phiZ^{-1})}\circ\eta \right)\circ \AdT(\underZ{c}^{-1}) \right] \\
&=\underset{\underZ{c}\in \underZ{C}}{\oplus}\Ind_{\AdT(\underZ{c})(S(F)G^0(F)_{y,0})}^{\AdT(\underZ{c})(K^0)} \left( \upkappa_{(S,\phi^{-1})} \circ \AdT(\underZ{c}^{-1}) \right).
\end{align*}
Finally, we apply Proposition \ref{prop:conjInd} to extract the $\AdT$ map from the induction and get
\begin{align*}
\rhoZ\circ\eta &\simeq \underset{\underZ{c}\in \underZ{C}}{\oplus} \left(\Ind_{S(F)G^0(F)_{y,0}}^{K^0} \upkappa_{(S,\phi^{-1})}\right) \circ \AdT(\underZ{c}^{-1})\\
&= \underset{\underZ{c}\in \underZ{C}}{\oplus} \rho\circ \AdT(\underZ{c}^{-1}).
\end{align*} The conclusion follows by noting that $\{\underZ{c}^{-1} : \underZ{c}\in \underZ{C}\}$ is another set of coset representatives of $\eta(K^0)\setminus \KZ{0}/\SZ(F)$. 
\end{proof}

Proposition \ref{prop:rhoeta} completes the proof of Theorem~\ref{th:matchingData}.

\subsubsubsection{Going Through the Steps of the J.-K.~Yu Construction}\label{sec:stepsCommute}

Let $(S,\theta)$ and $(\SZ,\thetaZ)$ be tame $F$-non-singular elliptic pairs of $G$ and $\GZ$, respectively, such that $\eta(S) = \SZ$ and $\theta=\thetaZ\circ\eta$. In the previous section, we have established the relationship between the corresponding J.-K.~Yu data, $(\vec{G},y,\vec{r},\rho,\vec{\phi})$ and $(\vec{\GZ},\yZ,\vec{r},\rhoZ,\vec{\phiZ})$, respectively. It is from these data that we construct the representations $\pi_{(S,\theta)}$ and $\pi_{(\SZ,\thetaZ)}$ following the steps of the J.-K.~Yu construction as outlined in Figure~\ref{fig:summaryJKYu}. To be consistent with notation, we will keep using subscript $Z$ to differentiate between the construction over $\GZ$ from that of $G$. Since we have that $\rep{i} = \repZ{i}\circ\eta$ for all $0\leq i\leq d$ and $\rhoZ\circ\eta = \underset{\underZ{c}\in\underZ{C}}{\oplus}\rho\circ \AdT(\underZ{c})$ with $\underZ{C}$ a set of coset representatives of $\eta(K^0)\setminus\KZ{0}/\SZ(F)$, it is natural to expect that we also have $\prim{i} = \primZ{i}\circ\eta$, $\kapZ{-1}\circ\eta = \underset{\underZ{c}\in\underZ{C}}{\oplus} \kap{-1} \circ\AdT(\underZ{c})$ and $\kap{i} = \kapZ{i}\circ\eta$ for all $0\leq i\leq d$. Indeed, we prove these equalities and inclusion with Propositions~\ref{prop:phieta} and \ref{prop:kappaeta}, and illustrate them in Figure~\ref{fig:etaCommute}. In particular, one can say that the J.-K.~Yu construction commutes with the map $\eta$.  The above results allow us to complete the proof of Theorem~\ref{th:mainTheorem} at the end of this section. 

 \begin{figure}[htbp!]
\center \begin{tikzpicture}
  \matrix (m) [matrix of math nodes,row sep=0.1em,column sep=1.5em,minimum width=0.1em]
  {
    {} &{} &{} & \text{$\left(\repZ{i}, {\KZ{i}}\right)$} &{} &\text{$\left(\rep{i}, {K^i}\right)$} \\
    {} &{} &{} &{} &\text{$\circlearrowleft$} &{} \\
     \text{$\left(\rhoZ, {\KZ{0}}\right)$} &{} &\text{$\left( \underset{\underZ{c}\in\underZ{C}}{\oplus}\rho\circ\AdT(\underZ{c}), {K^0} \right)$} &\text{$\left(\primZ{i}, {\KZ{i+1}}\right)$} &{} &\text{$\left(\prim{i}, {K^{i+1}}\right)$} \\
      {} &\text{$\circlearrowleft$} &{} &{} &\text{$\circlearrowleft$} &{} \\
     \text{$\left(\kapZ{-1}, {\KZ{d}}\right)$} &{} &\text{$\left( \underset{\underZ{c}\in\underZ{C}}{\oplus}\kap{-1} \circ \AdT(\underZ{c}), {K^d}\right)$} & \text{$\left(\kapZ{i}, {\KZ{d}}\right)$} &{} &\text{$\left( \kap{i}, {K^d}\right)$}\\
    };
    
\path[-stealth]
   (m-3-1) edge node [left] {inflate} (m-5-1)    
    (m-3-3) edge node [right] {inflate} (m-5-3)
    (m-1-4) edge node [above] {$\circ\eta$} (m-1-6)
   (m-1-4) edge node [left] {extend} (m-3-4)
    (m-3-4) edge node [below] {$\circ\eta$} (m-3-6)    
    (m-1-6) edge node [right] {extend} (m-3-6)
    (m-3-4) edge node [left] {inflate} (m-5-4)
    (m-3-6) edge node [right] {inflate} (m-5-6)    
    (m-5-4) edge node [below] {$\circ\eta$} (m-5-6);
    
\path[-stealth]
	(m-3-1) edge node [above] {$\circ \eta$} (m-3-3)
	(m-5-1) edge node [below] {$\circ\eta$} (m-5-3);

\end{tikzpicture}
\caption{Commutativity of the composition with $\eta$ with extension and inflation.}
\label{fig:etaCommute}
\end{figure}

\begin{proposition}\label{prop:phieta}
For all $0\leq i\leq d$ we have $\prim{i} = \primZ{i}\circ\eta$.
\end{proposition}

In order to define the extension $\prim{i}$ of $\rep{i}$, we require the groups $J^{i+1}$ and $J^{i+1}_+$, which were previously mentioned in Remark~\ref{rem:filtrations}. The extension process is divided in two steps: the first step consists of extending $\rep{i}$ to a character $\phihat{i}$ of $K^iG^{i+1}(F)_{y,s_i^+}$, where $s_i = r_i/2$. The character $\phihat{i}$ is the unique character of $K^iG^{i+1}(F)_{y,s_i^+}$ that agrees with $\rep{i}$ on $K^i$ and is trivial on $(G^i,G^{i+1})(F)_{y,(r_i^+,s_i^+)}$ \cite[Section 3.1]{HM:2008}. When $J^{i+1}\neq J^{i+1}_+$, a second step is required to extend the character $\phihat{i}$ a little further to a representation of $K^{i+1}$ by means of a Heisenberg-Weil lift. We adopt analogous notations to describe the extension $\primZ{i}$ of $\repZ{i}$. We note that $J^{i+1} = J^{i+1}_+$ if and only if $\JZ{i+1} = \JZP{i+1}$ (as a consequence of Remark~\ref{rem:filtrations}), which ensures that the construction of $\prim{i}$ requires a Heisenberg-Weil lift if and only if that of $\primZ{i}$ does. 

\begin{proof}[Proof of Proposition~\ref{prop:phieta}]
We have that $\phihat{i} = \phihatZ{i}\circ\eta$. Indeed, given that $\eta(K^i)\subset \KZ{i}$ and $\eta((G^i,G^{i+1})(F)_{y,(r_i^+,s_i^+)}) = (\GZ^i,\GZ^{i+1})(F)_{\yZ,(r_i^+,s_i^+)}$ (Remark~\ref{rem:filtrations}), one sees that $\phihatZ{i}\circ\eta$ agrees with $\rep{i}$ on ${K^i}$ and that it is trivial on $(G^i,G^{i+1})(F)_{y,(r_i^+,s_i^+)}$. 

In the case where $J^{i+1} = J^{i+1}_+$, we have $\prim{i} = \phihat{i}$ and $\primZ{i} = \phihatZ{i}$ and we are done. In the case where $J^{i+1}\neq J^{i+1}_+$, we have that $\prim{i}$ is constructed using a Heisenberg-Weil lift $\weil$, which is a representation of $K^i\ltimes \Heis$, where $\Heis = J^{i+1}/\ker(\xiG)$ and $\xiG = \phihat{i}|_{J^{i+1}_+}$. We then have $\prim{i}(kj) = \phihat{i}(k)\weil(k,j\ker(\xiG))$ for all $k\in K^i, j\in J^{i+1}$. Since $J^{i+1} \neq J^{i+1}_+$ if and only if $\JZ{i+1} \neq \JZP{i+1}$ (Remark~\ref{rem:filtrations}), we also require a Heisenberg-Weil lift $\weilZ$, which is a representation of $\KZ{i}\ltimes\HeisZ$, where $\HeisZ = \JZ{i+1}/\ker(\xiZ)$ and $\xiZ = \phihatZ{i}|_{\JZP{i+1}}$, and have that $\primZ{i}(\underZ{k}\underZ{j}) = \phihatZ{i}(k)\weilZ(\underZ{k},\underZ{j}\ker(\xiZ))$ for all $\underZ{k}\in\KZ{i}, \underZ{j}\in\JZ{i+1}$. 

Since we already know that $\phihat{i} = \phihatZ{i}\circ\eta$, it then suffices to show that $\weil = \weilZ\circ\eta$. Note that the map $\eta$ induces isomorphisms $\Heis\simeq \HeisZ$ and $W\simeq \WZ$ by Remark~\ref{rem:filtrations}, where $W = J^{i+1}/J^{i+1}_+$ and $\WZ = \JZ{i+1}/\JZP{i+1}$. We then obtain that $\weil = \weilZ\circ\eta$ as an application of \cite[Proposition 3.2]{Nevins:2015}, in which we set $H_1 = \Heis$, $H_2= \HeisZ$, $W_1 = W$, $W_2 = \WZ$, $T_1 = K^i$, $T_2 = \KZ{i}$, $\alpha = \delta = \eta$, $\nu_1$ and $\nu_2$ the corresponding special isomorphisms from \cite[Lemma 2.35]{HM:2008}, and $f_1$ and $f_2$ the homomorphisms coming from the actions by conjugation of $K^i$ and $\KZ{i}$ on $J^{i+1}$ and $\JZ{i+1}$, respectively.
\end{proof}

\begin{proposition}\label{prop:kappaeta}
For all $0\leq i\leq d$ we have $\kap{i} = \kapZ{i}\circ\eta$. Furthermore, $$\kapZ{-1}\circ\eta = \underset{\underZ{c}\in \underZ{C}}{\oplus}\kap{-1}\circ \AdT(\underZ{c}),$$ where $\underZ{C}$ is a set of coset representatives of $\eta(K^0)\setminus\KZ{0}/\SZ(F)$.
\end{proposition}

\begin{proof}
Let $0\leq i\leq d-1$. Let us briefly recall the process of inflation. We have that $K^d = K^{i+1}J$, where $J = J^{i+2}\cdots J^d$. Then, for all $k\in K^{i+1}, j\in J$, $\kap{i}(kj) = \prim{i}(k)$. Similarly, we have $\KZ{d} = \KZ{i+1}\underZ{J}$, where $\underZ{J} = \JZ{i+2}\cdots \JZ{d}$, and $\kapZ{i}(\underZ{k}\,\underZ{j}) = \primZ{i}(\underZ{k})$ for all ${\underZ{k}\in \KZ{i+1}}$, $\underZ{j}\in \underZ{J}$.

Using these definitions, for all $k\in {K^{i+1}}, j\in J$, we have ${\kap{i}(kj) = \prim{i}(k) = \primZ{i}(\eta(k))}$. By Remark~\ref{rem:filtrations}, we have that $\eta(k)\in \KZ{i+1}$ and $\eta(j)\in \underZ{J}$. Therefore, ${\primZ{i}(\eta(k)) = \kapZ{i}(\eta(k)\,\eta(j)) = \kapZ{i}\circ\eta(kj)}$. Thus, we conclude that $\kap{i} = \kapZ{i}\circ\eta$.

By a similar argument, we have that $$\underZ{\kappa}^{-1} \circ \eta = \underset{\underZ{c}\in \underZ{C}}{\oplus} \kappa^{-1}\circ \AdT(\underZ{c})$$ as a consequence of having (Proposition \ref{prop:rhoeta}) $$\rhoZ\circ\eta = \underset{\underZ{c}\in \underZ{C}}{\oplus}\rho\circ \AdT(\underZ{c}),$$ where $\underZ{C}$ is a set of coset representatives of $\eta(K^0)\setminus\KZ{0}/\SZ(F)$.
\end{proof}

We are now in a position to complete the proof of our main theorem.

\begin{proof}[Proof of Theorem~\ref{th:mainTheorem}]
We start by showing that $\kappa_{(\SZ,\thetaZ)}\circ \eta = \underset{\underZ{l}\in \underZ{L}}{\oplus}\kappa_{(S,\theta)}\circ \AdT(\underZ{l}),$ where $\underZ{L}$ is a set of coset representatives of $\eta(K^0)\setminus \KZ{0} / \SZ(F)$. For all $\underZ{l}\in \underZ{L}$, we have that $^{\underZ{l}^{-1}}\phiZ^i = \phiZ^i$, as $\KZ{0}\subset \GZ^i(F)$ for all $i$. It follows that $^{\underZ{l}^{-1}}\underZ{\kappa}^i \simeq \underZ{\kappa}^i$, which implies $^{\underZ{l}^{-1}}\kappa^i_Z\circ \eta \simeq \kappa^i_Z \circ \eta$. Using Proposition \ref{prop:kappaeta}, we conclude that $\kappa^i\circ \AdT(\underZ{l}) \simeq \kappa_i$ for all $\underZ{l}\in \underZ{L}$. Furthermore, Proposition \ref{prop:kappaeta} also tells us that $\kappa^{-1}_Z\circ \eta = \underset{\underZ{l}\in \underZ{L}}{\oplus}\kappa^{-1}\circ \AdT(\underZ{l})$. We thus obtain the following chain of equalities:
\begin{align*}
\kappa_{(\SZ,\thetaZ)}\circ\eta &= \left[\kappa^{-1}_Z \otimes \left( \bigotimes_{i=0}^d \kappa^i_Z \right)\right] \circ \eta \\
&= \kappa^{-1}_Z\circ \eta \otimes \left( \bigotimes_{i=0}^d \kappa^i_Z\circ\eta \right) \\
&= \left( \underset{\underZ{l}\in \underZ{L}}{\oplus} \kappa^{-1}\circ\AdT(\underZ{l}) \right) \otimes \bigotimes_{i=0}^d \kappa^i \\
&= \underset{\underZ{l}\in \underZ{L}}{\oplus}\left( \kappa^{-1} \circ \AdT(\underZ{l})\otimes \bigotimes_{i=0}^d\kappa^i \right) \\
&\simeq \underset{\underZ{l}\in \underZ{L}}{\oplus}  \left(\kappa^{-1} \otimes \bigotimes_{i=0}^d\kappa^i \right)\circ \AdT(\underZ{l})  \\
&= \underset{\underZ{l}\in \underZ{L}}{\oplus} \kappa_{(S,\theta)}\circ \AdT(\underZ{l}).
\end{align*}

Note that by Corollary~\ref{cor:AdK}, we have that $\AdT(\underZ{l})(K^d) = K^d$ for all $\underZ{l}\in \underZ{L}\subseteq \KZ{0}$. Therefore, the direct sum decomposition above makes sense as a representation of $K^d$.

From the Mackey Decomposition, we have
\begin{align*}
\pi_{(\SZ,\thetaZ)}\circ \eta &= \left(\Ind_{\KZ{d}}^{\GZ(F)}\kappa_{(\SZ,\thetaZ)}\right)\circ\eta \\
&=\left(\Res^{\GZ(F)}_{\eta(G(F))}\Ind_{\KZ{d}}^{\GZ(F)}\kappa_{(\SZ,\thetaZ)}\right)\circ\eta \\
&= \underset{\underZ{\ell}\in \underZ{\mathcal{L}}}{\oplus} \left( \Ind_{\eta(G(F))\cap {^{\underZ{\ell}} \KZ{d}}}^{\eta(G(F))}\Res^{^{\underZ{\ell}} \KZ{d}}_{\eta(G(F))\cap {^{\underZ{\ell}} \KZ{d}}} {^{\underZ{\ell}}\kappa_{(\SZ,\thetaZ)}}\right) \circ \eta,
\end{align*}
where $\underZ{\mathcal{L}}$ is a set of coset representatives of $\eta(G(F))\setminus \GZ(F) / \KZ{d}$. Given that $\eta(G(F))$ is a normal subgroup of $\GZ(F)$ and that $\eta(G(F))\cap \KZ{d} = \eta(K^d)$ (Lemma~\ref{lem:etaKs}), it follows that 
\begin{align*}
\eta(G(F)) \cap {^{\underZ{\ell}} \KZ{d}} &= {^{\underZ{\ell}}(\eta(G(F)) \cap \KZ{d})} \\ &= {^{\underZ{\ell}} \eta(K^d)} \\
&= \eta(\AdT(\underZ{\ell})(K^d)).
\end{align*} As a result, we may simplify the above expression and apply Proposition \ref{prop:induction} to obtain
\begin{align*}
\pi_{(\SZ,\thetaZ)}\circ \eta &= \underset{\underZ{\ell}\in \underZ{\mathcal{L}}}{\oplus}\left( \Ind_{\eta(\AdT(\underZ{\ell})(K^d))}^{^{\underZ{\ell}} \eta(G(F))}\Res_{\eta(\AdT(\underZ{\ell})(K^d))}^{^{\underZ{\ell}} \KZ{d}} {^{\underZ{\ell}}\kappa_{(\SZ,\thetaZ)}} \right) \circ\eta \\
&\simeq \underset{\underZ{\ell}\in \underZ{\mathcal{L}}}{\oplus}\Ind_{\AdT(\underZ{\ell})(K^d)}^{\AdT(\underZ{\ell})(G(F))}\left({^{\underZ{\ell}}\kappa_{(\SZ,\thetaZ)}\circ\eta}\right).
\end{align*}
We then apply Lemma \ref{lem:conjThrough}, followed by Proposition \ref{prop:conjInd} on the previous expression and get
\begin{align*}
\pi_{(\SZ,\thetaZ)}\circ \eta &\simeq \underset{\underZ{\ell}\in \underZ{\mathcal{L}}}{\oplus}\Ind_{\AdT(\underZ{\ell})(K^d)}^{\AdT(\underZ{\ell})(G(F))}\left[\left(\kappa_{(\SZ,\thetaZ)}\circ\eta\right)\circ \AdT(\underZ{\ell}^{-1})\right] \\
&\simeq \underset{\underZ{\ell}\in \underZ{\mathcal{L}}}{\oplus}\left( \Ind_{K^d}^{G(F)} \kappa_{(\SZ,\thetaZ)}\circ \eta \right) \circ \AdT(\underZ{\ell}^{-1}).
\end{align*}
By what precedes, it follows that
\begin{align*}
\pi_{(\SZ,\thetaZ)}\circ\eta &= \underset{\underZ{\ell}\in \underZ{\mathcal{L}}}{\oplus}\Ind_{K^d}^{G(F)}\left( \underset{\underZ{l}\in \underZ{L}}{\oplus} \kappa_{(S,\theta)} \circ \AdT(\underZ{l})\right) \circ \AdT(\underZ{\ell}^{-1})\\
&= \underset{\underZ{\ell}\in \underZ{\mathcal{L}}}{\oplus} \underset{\underZ{l}\in \underZ{L}}{\oplus} \Ind_{\AdT(\underZ{l})(K^d)}^{\AdT(\underZ{l})(G(F))}\left( \kappa_{(S,\theta)}\circ \AdT(\underZ{l}) \right) \circ \AdT(\underZ{\ell}^{-1}) \\
&\simeq \underset{\underZ{\ell}\in \underZ{\mathcal{L}}}{\oplus}\underset{\underZ{l}\in \underZ{L}}{\oplus} \left(\Ind_{K^d}^{G(F)}\kappa_{(S,\theta)} \right) \circ \AdT(\underZ{l})\circ \AdT(\underZ{\ell}^{-1})\\
&= \underset{\underZ{\ell}\in \underZ{\mathcal{L}}}{\oplus} \underset{\underZ{l}\in \underZ{L}}{\oplus} \pi_{(S,\theta)}\circ \AdT(\underZ{l}\,\underZ{\ell}^{-1}).
\end{align*}

Finally, we claim that $\{\underZ{l}\,\underZ{\ell}^{-1} : \underZ{l}\in \underZ{L}, \underZ{\ell}\in \underZ{\mathcal{L}}\}$ is a set of coset representatives of $\eta(G(F))\setminus \GZ(F)/\SZ(F)$, which we denote by $\underZ{D}$, allowing us to write
$$\pi_{(\SZ,\thetaZ)}\circ\eta = \underset{\underZ{d}\in \underZ{D}}{\oplus} \pi_{(S,\theta)}\circ \AdT(\underZ{d}).$$

To prove this last claim, we set $A =\eta(G(F)), B=\GZ(F), C=\KZ{d},\bar{A} = \eta(K^0), {\bar{B} = \KZ{0}}$ and $\bar{C} = \SZ(F)$ and show that $A,B,C,\bar{A},\bar{B},\bar{C}$ satisfy the hypotheses of part 1) of Lemma~\ref{lem:combineCosets}. It is clear that $A$ and $\bar{A}$ are normal subgroups of $B$ and $\bar{B}$, respectively, and that $\bar{B}\subseteq C$ and $A\cap \bar{B} \subseteq \bar{A}$. It remains to show that $\bigslant{C}{(A\cap C)\bar{C}}\simeq \bigslant{\bar{B}}{\bar{A}\bar{C}}.$

Recall that $\KZ{d} = \KZ{0}J_Z$ and $\eta(K^d) = \eta(K^0J) = \eta(K^0)J_Z$, where $J$ and $J_Z$ are as in the proof of Lemma~\ref{lem:etaKs}. Therefore,
$$
\bigslant{C}{(A\cap C)\bar{C}} = \bigslant{\KZ{d}}{\eta(K^d)\bar{C}} 
=  \bigslant{\KZ{0}J_Z}{\eta(K^0)J_Z\bar{C}}.
$$
Given that $\bar{C}$ is in the stabilizer of $\yZ$, it follows that $\bar{C}$ normalizes $J_Z$ and thus $J_Z\bar{C} = \bar{C}J_Z$. This, in combination with our modified third isomorphism theorem (Lemma~\ref{lem:modifiedisothm}) allows us to obtain
$$
\bigslant{C}{(A\cap C)\bar{C}} \simeq \bigslant{\KZ{0}}{\eta(K^0)\bar{C}(\KZ{0}\cap J_Z)}= \bigslant{\KZ{0}}{\eta(K^0)\bar{C}},
$$ where the last equality follows from the fact that $\KZ{0}\cap J_Z = \KZ{0}\cap \eta(J) \subset \eta(K^0)$ (Lemma \ref{lem:etaKs}). Thus, we conclude from part 1) of Lemma \ref{lem:combineCosets} that 
$\{\underZ{l}\, \underZ{\ell}^{-1}: \underZ{l}\in \underZ{L}, \underZ{\ell}\in \underZ{\mathcal{L}}\}$ is a set of coset representatives of $\eta(G(F))\setminus \GZ(F)/\SZ(F)$.
\end{proof}

\section{Functoriality for Supercuspidal $L$-packets}\label{sec:packetRelations}

Recall that $\PhiSC(G)$ denotes the set (of conjugacy classes) of supercuspidal $L$-parameters of $G$. Given our hypothesis on $p$, every $\param\in\PhiSC(G)$ has the property that $\param(P_F)$ is contained in a maximal torus of $\widehat{G}$ \cite[Lemma 4.1.3]{Kaletha:SCPackets}. Such parameters are called \emph{torally wild} in \cite{Kaletha:SCPackets}. Since all supercuspidal parameters we consider in this paper are torally wild, we will omit these adjectives. 

Given $\param\in\PhiSC(G)$, we let $\Pi_\param$ denote the associated $L$-packet of \cite{Kaletha:SCPackets}. Kaletha provides an explicit parameterization for $\Pi_\param$, and elements therein consist entirely of supercuspidal representations obtained from the construction outlined in Section~\ref{sec:summaryKaletha}. Thus, when ${\param\in\PhiSC(G)}$, we shall refer to $\Pi_\param$ as a \emph{supercuspidal} $L$-packet.

\subsection{The Construction of Supercuspidal $L$-packets}\label{sec:constructPackets}

 In order to describe Kaletha's construction of supercuspidal $L$-packets, we must first familiarize ourselves with his notion of a supercuspidal $L$-packet datum. We start this section by recalling the definition below.

\begin{definition}[{\cite[Definition 4.1.4]{Kaletha:SCPackets}}]\label{def:packetdatum}
A supercuspidal $L$-packet datum of $G$ is a tuple $(S,\jhat,\chi_0,\theta)$, where
\begin{enumerate}
\item[1)] $S$ is a torus of dimension equal to the absolute rank of $G$, defined over $F$ and split over a tame extension of $F$;
\item[2)] $\jhat: \widehat{S} \rightarrow \widehat{G}$ is an embedding of complex reductive groups whose $\widehat{G}$-conjugacy class is $\Gamma$-stable;
\item[3)] $\chi_0 = (\chi_{\alpha_0})_{\alpha_0}$ is tamely ramified $\chi$-data for $R(G,S^0)$, where $S^0$ is a particular subtorus of $S$ defined from $R_{0+}$ as explained in \cite[p.41]{Kaletha:SCPackets};
\item[3)] and $\theta :S(F)\rightarrow \bb{C}^\times$ is a character;
\end{enumerate}
subject to the condition that $(S,\theta)$ is a tame $F$-non-singular elliptic pair in the sense of \cite[Definition 3.4.1]{Kaletha:SCPackets}.
\end{definition}

Despite appearances, the torus $S$ does not actually live inside
$G$. It is an abstract torus that will be embedded into $G$ below.

The notion of $\chi$-data was introduced in \cite{LS:1987} and is recalled in \cite[Section 4.6]{Kaletha:Regular}. It is not necessary for the reader to be familiar with $\chi$-data in what follows. For our purposes, one can think of $\chi$-data for $R(G,S^0)$ simply as a set of characters of subfields of $F$ which are indexed by the root system.

By \cite[Proposition 4.1.8]{Kaletha:SCPackets}, there is a one-to-one correspondence between the $\widehat{G}$-conjugacy classes of supercuspidal $L$-parameters for $G$ and isomorphism classes of supercuspidal $L$-packet data. Following the proof of \cite[Proposition 4.1.8]{Kaletha:SCPackets}, given $\param\in\PhiSC(G)$, one constructs a representative $(S,\jhat,\chi_0,\theta)$ of the corresponding isomorphism class of supercuspidal $L$-packet data as follows:
\begin{itemize}
\item $S$: Let $\widehat{M} = \Cent(\param(P_F),\widehat{G})^\circ$, $\widehat{C} = \Cent(\param(I_F),\widehat{G})^\circ$ and $\widehat{S} = \Cent(\widehat{C},\widehat{M})$. By \cite[Lemma 5.2.2]{Kaletha:Regular} and \cite[Lemma 4.1.3]{Kaletha:SCPackets}, $\widehat{M}$ is Levi subgroup of $\widehat{G}$, $\widehat{C}$ is a torus of $\widehat{G}$ and $\widehat{S}$ is a maximal torus of $\widehat{G}$. The action of $W_F$ (which extends to $\Gamma$) on $\widehat{S}$ is defined as $\Ad(\param(-))$. The torus $S$ is then the torus dual to $\widehat{S}$.
\item $\jhat$: one simply takes $\jhat$ to be the set inclusion $\widehat{S} \hookrightarrow \widehat{G}$.
\item $\chi_0$: one chooses a tame $\chi$-data $\chi_0$ for $S^0$, which extends to a $\chi$-data for $S$ by \cite[Remark 4.1.5]{Kaletha:SCPackets}. 
\item $\theta$: Following \cite[Section 2.6]{LS:1987}, the $\chi$-data allows one to extend $\jhat$ to an embedding $^Lj: {^LS} \rightarrow {^LG}$. The image of $^Lj$ contains the image of $\param$ so that we may write $\param = {^Lj}\circ\param_S$ for some $L$-parameter $\param_S$ of $S$. We let $\theta$ be the corresponding character of $S(F)$ via the LLC for tori.
\end{itemize} 

We will say that $(S,\jhat,\chi_0,\theta)$ is the \emph{supercuspidal
$L$-packet datum associated to} $\param\in\PhiSC(G)$. The embedding
$\jhat$ belongs to a $\Gamma$-stable $\widehat{G}$-conjugacy class
$\widehat{J}$ of embeddings $\widehat{S} \rightarrow \widehat{G}$, and
from $\widehat{J}$ we obtain a $\Gamma$-stable
$G(F^{\mathrm{sep}})$-conjugacy class $J$ of embeddings $S\rightarrow
G$ (called \emph{admissible embeddings}) as per \cite[Section
  5.1]{Kaletha:Regular} and \cite[Sections 6.1 and 7.1]{Dillery}. We denote by $J(F)$ the set of $G(F)$-conjugacy classes of elements of $J$ which are defined over $F$. For each $j\in J(F)$, we consider the torus $jS = j(S)$ and let $j\theta = \theta\circ j^{-1}\cdot\epsilon_j$, where $\epsilon_j$ is the specific character from \cite[Section 4.1]{FKS}, described at the end of Section~\ref{sec:summaryKaletha}. Each pair $(jS,j\theta)$ is a tame $F$-non-singular elliptic pair from which we can construct a supercuspidal representation $\pi_{(jS,j\theta)}$ as described in Section \ref{sec:summaryKaletha}. The supercuspidal $L$-packet $\Pi_\param$ is then defined as
$$\Pi_\param \defeq \{[\pi_{(jS,j\theta)}]:j\in J(F)\},$$
where $j$ is identified with its $G(F)$-conjugacy class and $\pi_{(jS,j\theta)}$ is identified with its equivalence class.
Similarly, given $\paramT\in \PhiSC(\GT)$, we denote the associated supercuspidal $L$-packet datum by $(\ST,\jT,\chiT,\thetaT)$, and let $\JT(F)$ be the set of $\GT(F)$-conjugacy classes of admissible embeddings obtained from the $\Gamma$-stable $\widehat{\GT}$-conjugacy class of $\jT$ which are defined over $F$, so that
$$\Pi_{\paramT} = \{[\pi_{(\underT{j}\ST,\underT{j}\thetaT)}]:\underT{j}\in \JT(F)\}.$$

What we have denoted by $\Pi_\param$ is what Kaletha denotes as $\Pi_\param(G)$ in \cite{Kaletha:SCPackets}. Kaletha assigns the notation $\Pi_\param$ to his ``compound packet'' which encompasses rigid inner forms of $G$.


\subsection{Matching the Supercuspidal $L$-packet Data}\label{sec:packetData}

The relationship between $\Pi_{\paramT}$
and $\Pi_\param$, where $\paramT\in\PhiSC(\GT)$ and
$\param={^L\eta}\circ\param$, is given by matching
their corresponding parameterizing data through ${^L\eta}$.  The map
${^L\eta}$ is defined by $^L\eta(g,w) =
(\widehat{\eta}(\underT{g}),w)$ for all $\underT{g}\in\widehat{\GT}, w\in W_F$, and so
we begin with a review of the map  $\widehat{\eta}:
\widehat{\GT} \rightarrow \widehat{G}$ \cite[Sections 1 and
  2]{Springer:1979}.  Recall that we are assuming that $G$ is
quasisplit.  We may therefore fix a Borel subgroup $B \subset G$
defined over $F$, and fix a maximally $F$-split maximal torus $T$ of $G$
with $T\subset B$.    
This choice of Borel pair $(B,T)$ is equivalent to fixing a based root datum
for $G$.  The 
$\Gamma$-equivariant map $\eta$ carries $T$ to a maximally
$F$-split torus $\eta(T)$ of $\eta(G) \supset [G,G]$.  Set $\underT{T}
= \eta(T)Z(\GT)$ and $\underT{B} = \eta(B) Z(\GT)$.  Then
$(\underT{B}, \underT{T})$ is a Borel pair for $\GT$ which is
defined over $F$
\cite[Corollary 8.1.6]{Springer:LAG}.   Thus, we have fixed based root
data for both $G$ and $\GT$, and $\eta$ induces a homomorphism
between them (\emph{cf.} \cite[1.13]{Jantzen}).   This homomorphism of
root data determines a homomorphism from the dual based root datum of
$\GT$ to the dual based root datum of $G$.  These two
dual based data may be identified with the  based root
data arising from $\Gamma$-invariant Borel pairs $(\widehat{\underT{B}},
\widehat{\underT{T}})$ and $(\widehat{B}, \widehat{T})$ which are implicit in the
definition of ${^L\GT}$ and ${^LG}$
\cite[2.3]{Borel:1979}. Under this identification 
and the assumptions of Theorem \ref{th:desideratum},
the homomorphism of dual based root data produces the homomorphism
$\widehat{\eta}: \widehat{\GT} \rightarrow \widehat{G}$ of algebraic
groups with abelian kernel and cokernel \cite[Proposition
  1.14]{Jantzen}.  It is only unique up to conjugation by $\widehat{T}$.

A change in the choice of
based root data above has the effect of conjugating $\widehat{\eta}$ by an
element of $\widehat{G}$.  This change has no effect on the equivalence
classes of the objects under consideration.  However, even for fixed
based root data and a fixed choice of $\widehat{\eta}$, there remains an
ambiguity between $\eta$ and 
$\widehat{\eta}$.  Indeed, conjugating $\eta$ by an element $\underT{t} \in
\underT{T}$ such that $\underT{t}Z(\GT) \in
(\underT{T}/Z(\GT))(F)$ produces another $F$-homomorphism whose
dual is identical to $\widehat{\eta}$.   As observed in
\cite[Section 3]{Solleveld:2020}, this ambiguity is consequential in
matching the parameters of the $L$-packets under consideration.

In order to be more precise about these matters, we fix some pinnings.  We
extend the $\Gamma$-invariant Borel pair $(B,T)$ to a
$\Gamma$-invariant pinning  $(T,B,\{X_\alpha\})$ of $G$.  The
application of $\eta$ to this pinning fixes a $\Gamma$-invariant
pinning $(\underT{T}, \underT{B}, \{\eta(X_{\alpha})\})$ of
$\GT$.   It follows from \cite[Theorem
  3.2]{Solleveld:2020} that 
$\eta: G \rightarrow \GT$ is the unique $F$-homomorphism in its
$(\GT/Z(\GT))(F)$-conjugacy class which
carries the former pinning to the latter (see the discussion
preceding \cite[Theorem 6.2]{Solleveld:2020}).
Underlying the definition of ${^LG}$ and
${^L\GT}$ are  respective  pinnings
$(\widehat{T},\widehat{B},\{Y_{\widehat{\alpha}}\})$ of $\widehat{G}$
and $(\widehat{\underT{T}},\widehat{\underT{B}},
\{\underT{Y}_{\widehat{\alpha}}\})$ of $\widehat{\GT}$.  The fixed pinnings determine
$\widehat{\eta}$ uniquely by the requirement that
$\{\widehat{\eta}(Y_{\hat{\alpha}})\} = \{\underT{Y}_{\hat{\alpha}}\}$. 
Furthermore,  if $\eta' : G \rightarrow
 \GT$ is another $F$-homomorphism satisfying the assumptions of
 Theorem  \ref{th:desideratum} and $\widehat{\eta'} =
\widehat{\eta}$, then there is a unique element
$\underT{t}'Z(\GT) \in 
(\underT{T}/Z(\GT))(F)$ such that $\eta' = \Ad(\underT{t}') \circ
\eta$  \cite[Proposition 3.4]{Solleveld:2020}.

Having set the foundation for the comparison between $G$ and $\GT$,
let us return to the matching of the apposite parameters.   They are
to be matched as in Figure~\ref{fig:summaryData}.

\begin{theorem}\label{th:data}
Let $\paramT\in \PhiSC(\GT)$, $\param = {^L\eta}\circ\paramT$ and $(\ST,\jT,\chiT,\thetaT)$ and $(S,\jhat,\chi_0,\theta)$ be the associated supercuspidal $L$-packet data. Let $J(F)$ and $\JT(F)$ be the sets of embeddings which parameterize $\Pi_\param$ and $\Pi_{\paramT}$, respectively. Then $\widehat{\eta}(\widehat{\ST}) \subset \widehat{S}$, $\chiT = \chi_0$, $\widehat{\eta}\circ\jT = \jhat\circ\widehat{\eta}$ and $\theta=\thetaT\circ\eta_S$, where $\eta_S$ is the dual map of $\widehat{\eta}|_{\widehat{\ST}}: \widehat{\ST}\rightarrow \widehat{S}$. Furthermore, for all $\underT{j}\in \JT(F)$, there exists $j\in J(F)$ such that $\eta(jS) \subset \underT{j}\ST$ and $j\theta = \underT{j}\thetaT\circ\eta$.
\end{theorem} 

\begin{figure}[!htbp]
\begin{center}
\begin{tikzcd}[column sep = 0]
{} &{\paramT}\arrow{rrrrrr}{^L\eta\circ} &{} &{} &{} &{} &{} &{} &{\param} &{}\\[-30pt]
{} &{}\arrow[leftrightarrow]{d} &{} &{} &{} &{} &{} &{} &{}\arrow[leftrightarrow]{d} &{}\\
{} &{} &{} &{} &{} &{} &{} &{} &{}\\[-20pt]
{({\ST},} &{{\jT},}\arrow[rounded corners, to path={ 
      --node[left]{\scriptsize$\widehat{\eta}\circ$} ([xshift=-10ex]\tikztotarget.west)   
      -- (\tikztotarget.west)}]{drrr}{\widehat{\eta}\circ} &{\chi_0,} &{{\thetaT})}\arrow[bend left = 20]{rrrrrrr}{\circ \eta_S} &{} &{} &{} & {({S},} \arrow[bend left = 20]{lllllll}{\eta_S} &{{\jhat},}\arrow[rounded corners, to path={ 
      --node[right]{\scriptsize$\circ\widehat{\eta}$} ([xshift=10ex]\tikztotarget.east)   
      -- (\tikztotarget.east)}]{dllll}  &{\chi_0,} &{{\theta})}\\
{} &{} &{} &{} &{=} &{} &{} &{} &{} &{}\\
\end{tikzcd}
\end{center}
\caption{Summary of the relationship between the supercuspidal $L$-packet data associated to $\paramT$ and $\param$.}
\label{fig:summaryData}
\end{figure}


The proof of this theorem will be divided into three propositions. Proposition~\ref{prop:SSprime} will give the relationship between the tori, and consequently the embeddings and $\chi$-data. Proposition~\ref{prop:characters} will give the relationship between the characters and Proposition \ref{prop:embeddings} will provide the statement regarding the sets $J(F)$ and $\JT(F)$.

\begin{proposition}\label{prop:SSprime}
Let $\ST$ and $S$ be as in Theorem~\ref{th:data}. Then $\widehat{\eta}(\widehat{\ST}) \subset \widehat{S}$.
\end{proposition}

In order to prove this proposition, we will make use of the following result, which can be shown using \cite[Theorem 2.2]{Humphreys:Conjugacy} and the presentation of reductive groups in terms of generators from \cite[Theorem 26.3]{Humphreys:LAG}.

\begin{proposition}\label{lem:centralizers}
Let $T$ be a maximal torus of a connected reductive group $G'$, and assume $H$ is a subtorus of $T$. Then $T = \Cent(H,G')$ if and only if for every $\alpha\in R(G',T)$ there exists $h_\alpha\in H$ such that $\alpha(h_\alpha)\neq 1$.
\end{proposition}

%


\begin{proof}[Proof of Proposition~\ref{prop:SSprime}] Recall that $\widehat{S} = \Cent(\widehat{C},\widehat{M})$, where $\widehat{M} = \Cent(\param(P_F),\widehat{G})^\circ$ and $\widehat{C} = \Cent(\param(I_F),\widehat{G})^\circ$. Similarly, we have $\widehat{\ST} = \Cent(\CT,\MT)$, where $\MT = \Cent(\paramT(P_F),\widehat{\GT})^\circ$ and $\CT = \Cent(\paramT(I_F),\widehat{\GT})^\circ$.
We start by showing that $\widehat{\eta}(\MT) = \widehat{M}\cap \widehat{\eta}(\widehat{\GT})$. We have that $\paramT(P_F)$ is contained in some maximal torus $\widehat{\underT{\mathcal{T}}}$ of $\widehat{\GT}$ \cite[Lemma 4.1.3]{Kaletha:SCPackets}, and therefore $\param(P_F)$ is contained in the maximal torus $\widehat{\eta}(\widehat{\underT{\mathcal{T}}})$ of $\widehat{\eta}(\widehat{\GT})$. Since $\widehat{\eta}(\widehat{\GT})$ contains $[\widehat{G},\widehat{G}]$, we have that $\widehat{\eta}(\widehat{\underT{\mathcal{T}}}) = \widehat{\mathcal{T}} \cap \widehat{\eta}(\widehat{\GT})$ for some maximal torus $\widehat{\mathcal{T}}$ of $\widehat{G}$ \cite[Theorem 2.2]{thesisPaper}. By definition, we have
$$\MT = \mathrm{Cent}(\paramT(P_F),\widehat{\GT})^\circ = \left(\underset{s\in\paramT(P_F)}{\bigcap}\mathrm{Cent}(s,\widehat{\GT})\right)^\circ.$$ Using the description from \cite[Theorem 2.2]{Humphreys:Conjugacy} for each set $\mathrm{Cent}(s,\widehat{\GT}), s\in\paramT(P_F)$, it follows that
$$\MT = \langle \widehat{\underT{\mathcal{T}}}, U_\beta: \beta(s) = 1 \text{ for all } s\in\paramT(P_F)\rangle,$$ where $U_\alpha$ denotes the root subgroup of $\widehat{\GT}$ associated to the root $\alpha\in R(\widehat{\GT},\widehat{\underT{\mathcal{T}}})$.

Using a similar argument, and given that the root systems of $\widehat{\GT}$ and $\widehat{G}$ are canonically identified, we have that $$\widehat{M} = \langle \widehat{\mathcal{T}}, \widehat{\eta}(U_\beta) : \beta(s)=1 \text{ for all }s\in\param(P_F) \rangle.$$ We then deduce from \cite[Section 2B]{thesisPaper} that $\widehat{\eta}(\MT) = \widehat{M} \cap \widehat{\eta}(\widehat{\GT})$. Analogously, one has $\widehat{\eta}(\CT) = \widehat{C}\cap \widehat{\eta}(\widehat{\GT})$.

Since $\widehat{\ST} = \mathrm{Cent}(\CT,\MT)$ it follows from Proposition~\ref{lem:centralizers} that there exists $c_\alpha\in\CT$ such that $\alpha(c_\alpha)\neq 1$ for all $\alpha\in R(\MT,\widehat{\ST})$. Applying $\widehat{\eta}$, we have that for all $\alpha\in R(\widehat{\eta}(\MT),\widehat{\eta}(\widehat{\ST}))$, there exists $\widehat{\eta}(c_\alpha) \in \widehat{\eta}(\CT)$ such that $\alpha({\widehat{\eta}(c_\alpha)})\neq 1$. Reapplying Proposition~\ref{lem:centralizers}, we obtain $\widehat{\eta}(\widehat{\ST}) = \Cent(\widehat{\eta}(\CT), \widehat{\eta}(\MT))$.  It follows that 
$$\widehat{S} \cap \widehat{\eta}(\widehat{\GT}) = \Cent(\widehat{C},\widehat{M})\cap \widehat{\eta}(\widehat{\GT}) = \Cent(\widehat{C},\widehat{\eta}(\MT)) \subset \Cent(\widehat{\eta}(\CT),\widehat{\eta}(\MT)) = \widehat{\eta}(\widehat{\ST}).$$ Since both $\widehat{S}\cap \widehat{\eta}(\widehat{\GT})$ and $\widehat{\eta}(\widehat{\ST})$ are maximal tori of $\widehat{\eta}(\widehat{\GT})$, we conclude that they are equal, and therefore $\widehat{\eta}(\widehat{\ST})\subset \widehat{S}$.

\end{proof}

Having $\widehat{\eta}(\widehat{\ST}) \subset \widehat{S}$ implies that the root systems $R(\widehat{\GT},\widehat{\ST})$ and $R(\widehat{G},\widehat{S})$, together with their $\Gamma$-actions, are canonically identified, which allows us to choose $\chi_0 = \chiT$, as the $\chi$-data are parameterized by roots. Furthermore,  $\jT: \widehat{\ST} \rightarrow \widehat{\GT}$ and $\jhat: \widehat{S} \rightarrow \widehat{G}$ are simply inclusions. This means we have the following commutative diagram:

\begin{center}
\begin{tikzcd}
\widehat{\ST} \arrow[rightarrow]{r}{\jT} \arrow[rightarrow]{d}{\widehat{\eta}}
  & \widehat{\GT}\arrow[rightarrow]{d}{\widehat{\eta}} \\
\widehat{S} \arrow[rightarrow]{r}{\jhat}
  &  \widehat{G}
\end{tikzcd}
\end{center}

\begin{proposition}\label{prop:characters}
Let $\theta, \thetaT$ and $\eta_S$ be as in Theorem~\ref{th:data}. Then, $\theta = \thetaT\circ\eta_S$.
\end{proposition}

\begin{proof}
Using the $\chi$-data as in \cite[Section 2.6]{LS:1987}, the above diagram extends into another commutative diagram:
\begin{center}
\begin{tikzcd}
\widehat{\ST}\rtimes W_F \arrow[rightarrow]{r}{^L\underT{j}} \arrow[rightarrow]{d}{^L\eta}
  & \widehat{\GT}\rtimes W_F\arrow[rightarrow]{d}{^L\eta} \\
\widehat{S}\rtimes W_F \arrow[rightarrow]{r}{^Lj}
  &  \widehat{G}\rtimes W_F
\end{tikzcd},
\end{center}
where $^L\eta(g,w) = (\widehat{\eta}(g),w)$ for all $g\in\widehat{\GT}, w\in W_F$.  

Following \cite[Proposition 4.1.8]{Kaletha:SCPackets}, $\mathrm{Im}\paramT\subset \mathrm{Im}{^L\underT{j}}$ and $\mathrm{Im}\phi\subset\mathrm{Im}{^Lj}$, meaning that $\paramT = {^L\underT{j}\circ\param_{\ST}}$ and $\param = {^Lj\circ\param_{S}}$ for some $L$-parameters $\param_{\ST}$ and $\param_S$ of $\ST$ and $S$, respectively. We claim that $\param_{S} = {^L\eta}\circ\param_{\ST}$. Indeed, by definition, $\param = {^L\eta}\circ\paramT$, which implies $^Lj\circ\param_{S} = {^L\eta}\circ {^L\underT{j}}\circ \param_{\ST}$. Using the commutative diagram above, it follows that $^Lj\circ\param_{S} = {^Lj}\circ {^L\eta}\circ\param_{\ST}$. Given that $^Lj$ is an embedding, it is injective by definition, which implies that $\param_{S} = {^L\eta}\circ\param_{\ST}$ as claimed.

By definition, $\thetaT$ and $\theta$ are the characters which correspond to $\param_{\ST}$ and $\param_{S}$, respectively, under the LLC for tori. Since $L$-packets of tori are singletons, we apply the functoriality property for the LLC of tori to conclude that $\theta = \thetaT\circ\eta_S$.
\end{proof}

We now arrive to the final statement of Theorem~\ref{th:data} which
matches the embeddings in $J(F)$ and $\JT(F)$.   The description
of these embeddings depends on our fixed pinnings.

\begin{proposition}\label{prop:embeddings}
For all $\underT{j}\in \JT(F)$, there exists $j\in J(F)$ such that $\eta(jS) \subset \underT{j}\ST$ and ${j\theta = \underT{j}\thetaT\circ\eta}$.
\end{proposition}

\begin{proof}
Our fixed $\Gamma$-invariant pinnings satisfy $\eta(T) =
\underT{T}\cap\eta(G)$ and $\widehat{\eta}(\widehat{\underT{T}})=
\widehat{T}\cap\widehat{\eta}(\widehat{\GT})$.
Following \cite[Section 5.1]{Kaletha:Regular}, we may describe $\JT$ and $J$ as follows. Choose $\widehat{\underT{i}}$ in $\widehat{\JT}$ such that $\widehat{\underT{i}}(\widehat{\ST}) = \widehat{\underT{T}}$ and define $\underT{i}$ to be the inverse of the isomorphism $\underT{T}\rightarrow \ST$ induced by $\widehat{\underT{i}}$. We have that $\widehat{\underT{i}} = \mathrm{Ad}(\widehat{g})\circ\jT$ for some $\widehat{g} \in \widehat{\GT}$. Let $\widehat{i}\in\widehat{J}$ be defined by $\widehat{i} = \mathrm{Ad}(\widehat{\eta}(\widehat{g}))\circ \jhat$. Since $\widehat{\eta}\circ\jT = \jhat\circ\widehat{\eta}$, we have the following commutative diagram.
\begin{equation}\label{diagram}
\begin{tikzcd}
\widehat{\ST} \arrow[rightarrow]{r}{\widehat{\underT{i}}} \arrow[rightarrow]{d}{\widehat{\eta}}
  & \widehat{\underT{T}}\arrow[rightarrow]{d}{\widehat{\eta}} \\
\widehat{S} \arrow[rightarrow]{r}{\widehat{i}} 
  &  \widehat{T}
\end{tikzcd}
\end{equation}

It follows that 
$$\widehat{T}\cap \widehat{\eta}(\widehat{\GT}) = \widehat{\eta}(\widehat{\underT{T}}) = \widehat{\eta}(\widehat{\underT{i}}(\widehat{\ST})) = \widehat{i}(\widehat{\eta}(\widehat{\ST}))\subset \widehat{i}(\widehat{S}).$$
Since we know $\widehat{i}(\widehat{S})$ has to be a maximal torus of $\widehat{G}$, we conclude from \cite[Theorem 2.2]{thesisPaper} that $\widehat{i}(\widehat{S}) = \widehat{T}$. Therefore, $\JT$ corresponds to the $\GT(F^{\mathrm{sep}})$-conjugacy class of $\underT{i}$, and $J$ corresponds to the $G(F^{\mathrm{sep}})$-conjugacy class of $i$. 

Now, given $\underT{j}\in \JT(F)$, we have that $\underT{j} = \mathrm{Ad}(\underT{g})\circ \underT{i}$ for some $\underT{g}\in \GT(F^{\mathrm{sep}})$. Using the fact that $\GT = Z(\GT)^\circ\GZ$, we may assume without loss of generality that $\underT{g}\in \GZ(F^{\mathrm{sep}})$. Let $g$ be any preimage of $\underT{g}$ in $G(F^{\mathrm{sep}})$ by $\eta$ and set $j = \mathrm{Ad}(g)\circ i$. By taking the dual of diagram~(\ref{diagram}), we have
\begin{equation}\label{diagram2}
\begin{tikzcd}
{\ST} \arrow[rightarrow]{r}{{\underT{i}}}
  & {\underT{T}} \\
{S} \arrow[rightarrow]{r}{{i}} \arrow[rightarrow]{u}{\eta_S}
  &  {T}\arrow[rightarrow]{u}{\eta}
\end{tikzcd}
\end{equation}
Here, $\eta_S$ is the dual map of $\widehat{\eta}|_{\widehat{\ST}}: \widehat{\ST} \rightarrow \widehat{S}$. It follows that
$$\eta(jS) = \eta(g\, iS\, g^{-1}) = \eta(g)\, \eta(iS)\, \eta(g)^{-1} = \underT{g}\, \underT{i}(\eta_S(S))\, \underT{g}^{-1} = \underT{j}(\eta_S(S)) \subset \underT{j}\ST.$$ Since $\underT{j}$ and $\eta$ are defined over $F$, we have that $j\in J(F)$. Indeed, by \cite[Lemma 6.2]{Dillery}, which generalizes \cite[Corollary 2.2]{Kottwitz:1982} to arbitrary local fields, there exists $h\in G(F^{\mathrm{sep}})$ such that $\Ad(h)\circ j$ is defined over $F$. Then $\Ad(\eta(h))\circ\underT{j}$ is also defined over $F$, implying that $\sigma(\Ad(\eta(h))\circ\underT{j})\sigma^{-1}=\Ad(\eta(h))\circ\underT{j}$ for all $\sigma\in\Gamma$.
Equivalently, $\eta(h)^{-1}\sigma(\eta(h)) = \eta(h^{-1}\sigma(h))\in \underT{j}\ST$ for all $\sigma\in\Gamma$. This implies $h^{-1}\sigma(h)\in jS$, and therefore $\Ad(h)\circ j = \Ad(\sigma(h)) \circ j$ for all $\sigma\in\Gamma$. Using the fact that $\Ad(h)\circ j$ is defined over $F$, we rewrite this last equality as $\sigma(\Ad(h)\circ j)\sigma^{-1} = \Ad(\sigma(h))\circ j$ for all $\sigma\in\Gamma$. It follows that $\Ad(\sigma(h))\circ \sigma j\sigma^{-1} = \Ad(\sigma(h))\circ j$, and therefore $\sigma j \sigma^{-1}=j$ for all $\sigma\in\Gamma$.

We now show that $j\theta = \underT{j}\thetaT\circ \eta$. We have that $j\theta = \theta\circ j^{-1}\cdot \epsilon_j$ and $\underT{j}\thetaT = \thetaT\circ \underT{j}^{-1}\cdot \epsilon_{\underT{j}}$, where $\epsilon_j$ and $\epsilon_{\underT{j}}$ are the characters from \cite[Section 4.1]{FKS} which we briefly recalled at the end of Section~\ref{sec:summaryKaletha}. By what precedes, we have that $\eta(jS) \subset \underT{j}\ST$ and $\thetaT\circ \underT{j}^{-1}\circ\eta = \thetaT\circ\eta_S\circ j^{-1} = \theta \circ j^{-1}$. We claim that $\epsilon_j = \epsilon_{\underT{j}}\circ\eta$. Indeed, let $(G^0_j,\dots,G^d_j)$ and $(\GT^0_{\underT{j}},\cdots, \GT^d_{\underT{j}})$ be the twisted Levi sequences obtained from $jS$ and $\underT{j}\ST$, respectively. Recall that $\epsilon_j = \prod\limits_{i=1}^d\epsilon^{G^i_j/G^{i-1}_j}$, where $\epsilon^{G^i_j/G^{i-1}_j}$ is the quadratic character of $K^d_j$ that is trivial on $G^1_j(F)_{y_j,r_0/2}\cdots G^d_j(F)_{y_j,r_d/2}$ and whose restriction to $K^0_j$ is given by $\epsilon^{G^i_j/G^{i-1}_j}_{y_j}$ defined in \cite[Definition 4.1.10]{FKS}. The character $\epsilon^{G^i_j/G^{i-1}_j}_{y_j}$ is essentially just a composition of a sign character constructed from the adjoint groups of $G^i_j$ and $G^{i-1}_j$, and the adjoint map of $G^i_j$. The character $\epsilon_{\underT{j}}$ is defined similarly. Given that $\eta(G^i_j) = \GT^i_{\underT{j}}\cap\eta(G)$ (Lemma~\ref{lem:LeviSeq} and Theorem~\ref{th:HG}), it follows that $G^i_j$ and $\GT^i_{\underT{j}}$ have the same adjoint group and that the adjoint map of $G^i_j$ is the composition of the adjoint map of $\GT^i_{\underT{j}}$ with $\eta$. It follows that  $\epsilon^{G^i_j/G^{i-1}_j}_{y_j} = \epsilon^{\GT^i_{\underT{j}}/\GT^{i-1}_{\underT{j}}}_{\underT{y}_{\underT{j}}} \circ \eta$ for all $1\leq i\leq d$ and therefore $\epsilon_j = \epsilon_{\underT{j}}\circ\eta$. This completes the proof.
\end{proof}

\subsection{The Proof of Theorem~\ref{th:desideratum}}\label{sec:proofDesideratum}


\begin{proof}[Proof of Theorem~\ref{th:desideratum}]
Let $(\ST,\jT,\chi_0,\thetaT)$ and $(S,\jhat,\chi_0,\theta)$ be the supercuspidal $L$-packet data associated to $\paramT$ and $\param$, respectively. By construction of $\Pi_{\paramT}$, we have that $\piT \subset \pi_{(\underT{j}\ST,\underT{j}\thetaT)}$ for some $\underT{j}\in \JT(F)$. By Theorem~\ref{th:data}, there exists $j\in J(F)$ such that $\eta(jS) \subset \underT{j}\ST$ and $j\theta = \underT{j}\thetaT\circ\eta$.  
By Theorem~\ref{th:fullTheorem}, it follows that
$$\piT\circ\eta \subset \pi_{(\underT{j}\ST,\underT{j}\thetaT)}\circ\eta = \underset{\underT{c}\in \underT{C}}{\oplus}\pi_{(jS,j\theta)}\circ \AdT(\underT{c}),$$
where $\underT{C}$ is a set of coset representatives of $\eta(G(F))\setminus \GT(F)/\underT{j}\ST(F)$. Recall that $\AdT(\underT{c}) = \Ad(c),$ where $c\in G$ satisfies $\underT{c} = \eta(c) z$ for some $z\in Z(\GT)$. Following the steps of the construction from Section~\ref{sec:summaryKaletha}, one sees that $\pi_{(jS,j\theta)}\circ\AdT(\underT{c}) = {^{c^{-1}}\pi_{(jS,j\theta)}} \simeq \pi_{({^{c^{-1}}jS}, {^{c^{-1}}j\theta})}$ as a consequence of \cite[Section 5.1.1]{HM:2008}, where $^{c^{-1}}jS = (\mathrm{Ad}(c^{-1})\circ j)S$ and 
\begin{align*}
^{c^{-1}}j\theta &=  j\theta\circ \Ad(c) \\
&= ((\theta\circ j^{-1})\cdot\epsilon_j )\circ \Ad(c) \\
&= \left(\theta\circ j^{-1}\circ\Ad(c)\right) \cdot \left(\epsilon_j\circ\Ad(c)\right)\\
&= \left( \theta\circ (\Ad(c^{-1})\circ j)^{-1}\right) \cdot \epsilon_{\Ad(c^{-1})\circ j}\\
&= (\Ad(c^{-1})\circ j)\theta
\end{align*} Since $\Ad(c^{-1})$ is defined over $F$ (Lemma \ref{lem:AdMap}), $\mathrm{Ad}(c^{-1})\circ j \in J(F)$, and therefore $$[\pi_{(jS,j\theta)}\circ\AdT(\underT{c})] = [\pi_{((\Ad(c^{-1})\circ j)S,(\Ad(c^{-1})\circ j)\theta)}]\subset \Pi_{\param}$$ for all $\underT{c}\in\underT{C}$. Thus, all irreducible components of $\underT{\pi}\circ\eta$ belong to $\Pi_\param$.

\end{proof}

\section{Specializing to Regular Supercuspidal Parameters}\label{sec:specializeReg}

Part of the local Langlands conjectures is a correspondence between the irreducible representations in an $L$-packet $\Pi_\varphi$ and the irreducible representations of the component group of $\mathrm{Cent}(\varphi(W_F), \widehat{G})$ \cite[Conjecture 1.15]{Vogan:1993}. In this section, we review this correspondence for \emph{regular} supercuspidal $L$-parameters.

The regular supercuspidal $L$-parameters \cite[Definition 5.2.3]{Kaletha:Regular} are an important subclass of the supercuspidal $L$-parameters. They are also easier to study, as their corresponding $L$-packets are simpler to describe (cf. (\ref{eq:simplerpacket}) below), and their corresponding component groups are abelian \cite[Lemma 5.3.4]{Kaletha:Regular}.

Under the assumption that $\param$ and $\paramT$ are both regular, we show how to reparameterize the elements of $\Pi_\varphi$ and $\Pi_{\paramT}$ in terms of characters of their respective component groups. From this reparameterization, we obtain an alternate formulation for the decomposition formula for $\underT{\pi}\circ \eta, \underT{\pi}\in \Pi_{\paramT}$, obtained from Theorem \ref{th:fullTheorem}. This reformulation amounts to a proof of \cite[Conjecture 2]{Solleveld:2020} for regular supercuspidal $L$-packets of quasisplit groups (Theorem \ref{th:Solleveld}, Proposition \ref{prop:otherpinning}).

Let us begin this section by describing explicitly the regular supercuspidal $L$-parameters, their corresponding $L$-packet structure, and the relationship between the regularity of $\param$ and $\paramT$.

\subsection{Regular Packets and Conditions for Regularity of $\param$ and $\paramT$}\label{sec:regularParams}

One way to describe the regular supercuspidal $L$-parameters is via the notion of regular supercuspidal $L$-packet data \cite[Definition 5.2.4]{Kaletha:Regular}.  A regular supercuspidal $L$-packet datum of $G$ is a supercuspidal $L$-packet datum $(S,\jhat,\chi_0,\theta)$ (cf. Definition \ref{def:packetdatum}), with the stronger condition that $(S,\theta)$ is an extra regular elliptic pair in the sense of \cite[Definition 3.7.5]{Kaletha:Regular}. In particular, this means that the character $\theta|_{S(F)_0}$ has trivial stabilizer for the action of $\Omega(S,G)(F) := (N_G(S)/S)(F)$. 

By \cite[Proposition 5.2.7]{Kaletha:Regular}, there is a one-to-one
correspondence between the $\widehat{G}$-conjugacy classes of regular supercuspidal $L$-parameters for $G$ and isomorphism classes of regular supercuspidal $L$-packet data. Given a regular supercuspidal $L$-parameter $\varphi$ of $G$ with associated regular supercuspidal $L$-packet datum $(S,\hat{j},\chi_0,\theta)$, the representations $\pi_{(jS,j\theta)}$ are irreducible for all $j\in J(F)$ \cite[Lemma 3.4.20]{Kaletha:Regular}. 
Thus, the corresponding $L$-packet is 
\begin{align}\label{eq:simplerpacket}
\Pi_\varphi = \{ \pi_{(jS,j\theta)} : j\in J(F)\},
\end{align}
where $j$ is identified with its $G(F)$-conjugacy class and $\pi_{(jS,j\theta)}$ is identified with its equivalence class. Furthermore, as stated in \cite[Section 5.3]{Kaletha:Regular} and \cite[Section 4.2]{Kaletha:SCPackets}, the elements of $\Pi_\varphi$ are in one-to-one correspondence with the elements of $J(F)$. The following lemma is a proof of this statement.

\begin{lemma}\label{lem:simplerpackets}
Let $\varphi$ be a regular supercuspidal $L$-parameter of $G$ with associated regular $L$-packet datum $(S,\hat{j},\chi_0,\theta)$. Then the map $j \mapsto \pi_{(jS,j\theta)}$ induces a bijection
$J(F) \rightarrow \Pi_\varphi.$
\end{lemma}

\begin{proof}
We prove the equivalent statement: $\pi_{(j_1S,j_1\theta)} \simeq \pi_{(j_2S,j_2\theta)}$ if and only if $j_1$ and $j_2$ are $G(F)$-conjugate. 

Assume $\pi_{(j_1S,j_1\theta)} \simeq \pi_{(j_2S,j_2\theta)}$. Then, by \cite[Corollary 3.7.10]{Kaletha:Regular}, there exists $g\in G(F)$ such that $j_1S = \Ad(g)j_2S$ and $j_1\theta = {^gj_2\theta}$. Using \cite[Lemma 3.4.10, Lemma 3.4.12]{Kaletha:Regular}, there exists $j' \in J(F)$ such that $(j'S,j'\theta)$ is extra regular in the sense of \cite[Definition 3.7.5]{Kaletha:Regular}. Now, as in the proof of Proposition \ref{prop:embeddings}, $j_1 = \Ad(h_1)\circ j'$ and $j_2=\Ad(h_2)\circ j'$ for some $h_1,h_2\in G(F^{\mathrm{sep}})$ such that $\Ad(h_1)$ and $\Ad(h_2)$ (as maps of $j'S$) are defined over $F$. Thus, $j_1S = \Ad(g)j_2S$ and $j_1\theta = {^gj_2\theta}$ if and only if $h_1^{-1}gh_2 \in N_G(j'S)$ and $j'\theta = {^{h_1^{-1}gh_2}}j'\theta$. Because $\Ad(h_1^{-1}gh_2)$ is defined over $F$, it is an easy exercise to show that $\sigma(h_1^{-1}gh_2)^{-1}(h_1^{-1}gh_2) \in C_G(j'S) = j'S$ for all $\sigma\in \Gamma$. 
It follows that $(h_1^{-1}gh_2)j'S \in \Omega(j'S,G)(F) = \left(N_G(j'S)/j'S\right)(F)$. By the extra regularity of $j'\theta$, we conclude that $h_1^{-1}gh_2\in j'S$, and therefore $\Ad(h_1^{-1}gh_2)\circ j' = j'$. Thus, $\Ad(g)\circ j_2 = j_1$.

The converse is a direct consequence of \cite[Corollary 3.7.10]{Kaletha:Regular}, and is built into the definition of $\Pi_\param$.
\end{proof}

Suppose as usual that $\underT{\varphi} \in \Phi_{\mathrm{sc}}(\GT)$ and $\varphi = {^L\eta\circ \underT{\varphi}} \in \Phi_{\mathrm{sc}}(G)$. It is natural to ask under what conditions $\varphi$ and $\underT{\varphi}$ are both regular. The following lemma and corollary address this question from the perspective of regular supercuspidal $L$-packet data.

\begin{lemma}\label{lem:regularpairs}
Let $(S,\theta)$ and $(\ST,\thetaT)$ be $F$-non-singular elliptic pairs of $G$ and $\GT$, respectively, satisfying $\eta(S) \subset \ST$ and $\theta = \thetaT \circ \eta$. If $(S,\theta)$ is extra regular, then $(\ST,\thetaT)$ is also extra regular.
\end{lemma}

\begin{corollary}\label{cor:regularparams}
Let $\underT{\varphi}$ be a supercuspidal $L$-parameter of $\GT$ and let $\varphi = {^L\eta}\circ \underT{\varphi}$. If $\varphi$ is regular, then $\underT{\varphi}$ is also regular.
\end{corollary}

Before proving Lemma \ref{lem:regularpairs}, recall from Section \ref{sec:constructPackets} that $S$ is not a subtorus of $G$. Rather, as in \cite[p.1145]{Kaletha:Regular}, we have a structure on $S$ that is given to us by $\jhat$. This means that the action of $\Omega(S,G)$ on $S$ corresponds to the action of $\Omega(T,G)$ twisted by $i:S\rightarrow T$. More precisely, given $w\in \Omega(S,G) = \Omega(T,G)$, $^ws = i^{-1}({^wi(s)})$ for all $s\in S$. Here, $T$ and is the maximal torus from our fixed $\Gamma$-invariant pinning, and $i$ is as per the proof of Proposition \ref{prop:embeddings}. The same is also true of $\ST$, for which we adopt analogous notation. Furthermore, since $\eta(T) = \eta(G)\cap \underT{T}$ and $\eta(G) \supset [\GT,\GT]$, one sees from \cite[Proposition 2.1.24]{thesis} that $\eta$ induces a $\Gamma$-equivariant isomorphism
$$\begin{aligned}
\eta_\Omega: \Omega(T,G) \xrightarrow{\eta} \Omega(\eta(T),\eta(G)) \rightarrow \Omega(\underT{T},\GT)
\end{aligned},$$ which sends $gT$ to $\eta(g)\underT{T}$ for all $g\in N_G(T)$.

\begin{proof}[Proof of Lemma \ref{lem:regularpairs}]
It is clear that the first two conditions in the definition of extra regularity \cite[Definition 3.7.5]{Kaletha:Regular} are satisfied for $(\ST,\thetaT)$ if and only if they are satisfied for $(S,\theta)$. We focus our attention on the third and final condition. That is, we assume that $\theta|_{S(F)_0}$ has trivial stabilizer for the action of $\Omega(S,G)(F)$ and show that $\thetaT|_{\ST(F)_0}$ has trivial stabilizer for the action of $\Omega(\ST,\GT)(F)$.

Recall from Proposition \ref{prop:characters} and (\ref{diagram2}) that $\theta = \tilde{\theta}\circ \eta_S$ and $\underT{i}\circ\eta_S = \eta\circ i$. Using these equalities in combination with the definitions of the actions of $\Omega(S,G)$ and $\Omega(\ST,\GT)$, we obtain 
\begin{align}\label{eq}
^w\theta = {^{\eta_\Omega(w)}\tilde{\theta}\circ \eta_S}  \text{ for all } w\in \Omega(S,G)(F).
\end{align}


Assume that $^{\underT{w}}\thetaT|_{\ST(F)_0} = \thetaT|_{\ST(F)_0}$ for some $\underT{w} \in \Omega(\ST,\GT)(F)$. By the above discussion, $\underT{w} = \eta_\Omega(w)$ for some $w\in \Omega(S,\G)(F)$. Using (\ref{eq}), it follows that
$$\theta|_{S(F)_0} = (\thetaT\circ \eta_S)|_{S(F)_0} = ({^{\underT{w}}}\thetaT\circ \eta_S)|_{S(F)_0} = {^w\theta}|_{S(F)_0}.$$ Given the assumption on $\theta$, we conclude that $w = 1$. Thus $\underT{w}=1$ and $\thetaT|_{\ST(F)_0}$ has trivial stabilizer for the action of $\Omega(\ST,\GT)(F)$.
\end{proof}


The converse of Lemma \ref{lem:regularpairs} and Corollary \ref{cor:regularparams} is not true in general. Consider the case of $G= \mathrm{SL}_2$ and $\underT{G} = \mathrm{GL}_2$, with $\eta$ being the inclusion map. Then, all irreducible supercuspidal representations of $\underT{G}(F)$ are extra regular \cite[Lemma 3.7.7]{Kaletha:Regular}, whereas there exist irreducible supercuspidal representations of $G(F)$ which are not regular (e.g. the four \emph{exceptional} supercuspidal representations from \cite{ADSS}). 
Given one such representation of $G(F)$, say $\pi$, the irreducible components of $\Ind_{G(F)}^{\underT{G}(F)}\pi$ are all extra regular, and thus correspond to extra regular elliptic pairs of $\underT{G}(F)$. The restrictions of these extra regular elliptic pairs to $G(F)$ can therefore not be extra regular (or even regular), as an application of Theorem \ref{th:fullTheorem} would then contradict the non-regularity of $\pi$.

It is worth pointing out that the instances for which the converse holds are not extremely rare. Indeed, assume that $\thetaT$ is extra regular and that $^w\theta|_{S(F)_0} = \theta|_{S(F)_0}$ for some $w\in \Omega(S,G^0)(F)$. Then (\ref{eq}) tells us $^{\eta_\Omega(w)}\thetaT \circ \eta_S|_{S(F)_0} = \thetaT \circ \eta_S|_{S(F)_0},$ or equivalently, $^{\eta_\Omega(w)}\thetaT|_{\eta_S(S(F)_0)} = \thetaT|_{\eta_S(S(F)_0)}.$ In order to conclude that $w=1$ (i.e. $\theta$ is extra regular), we must have $^{\eta_\Omega(w)}\thetaT|_{\ST(F)_0} = \thetaT|_{\ST(F)_0}.$ If the difference between $\ST(F)_0$ and $\eta_S(S(F)_0)$ is centralized by $\Omega(\ST,\GT)(F)$, then we can conclude that $w=1$, and so $\theta$ is extra regular when $\thetaT$ is extra regular. This happens for example if the equality $\underT{T} = Z(\GT) \, \eta(T)$ (or equivalently, $\ST = \underT{i}^{-1}(Z(\GT))\,\eta_S(S)$) remains true at the level of the $F$-points. 

\subsection{Functoriality and the Characters of the Dual Component Group}\label{sec:SolleveldConj}

\subsubsection{Admissible Embeddings and Characters of the Dual Component Group}
\label{embcompgp}

Let $\param \in \PhiSC(G)$ be a regular supercuspidal $L$-parameter. Up until now, the elements of $\Pi_\param$ have been parameterized by regular elliptic pairs. However, in Solleveld's conjecture \cite[Conjecture 2]{Solleveld:2020}, the elements of the packets are parameterized by $\param$ and characters of the associated component group. Thus, the first step for proving Theorem~\ref{th:Solleveld} is to reconcile the two parameterizations.

As discussed in Section \ref{sec:regularParams}, the $L$-parameter
$\varphi$ corresponds to a regular supercuspidal $L$-packet datum $(S,
\jhat, \chi_{0}, \theta)$, and the 
(equivalence classes of) irreducible representations in
$\Pi_{\varphi}$ correspond bijectively to the ($G(F)$-conjugacy
classes of ) admissible embeddings in $J(F)$
\begin{equation}
  \label{diag-1}
  \Pi_{\varphi} \longleftrightarrow J(F).
  \end{equation}

Thus far, there has been no need to specify a particular irreducible
representation or 
admissible embedding in (\ref{diag-1}) as being special.  However, it
shall be necessary 
to specify a particular representation or embedding to correspond to
the identity character of the dual component group.  Kaletha discusses what the canonical choice of representation and embedding should be in some settings \cite[Sections 5.3, 6.2]{Kaletha:Regular}\cite[Lemma 4.2.1]{Kaletha:SCPackets}.  At the time of writing, a canonical choice is not available for regular supercuspidal $L$-packets. For this reason,
we arbitrarily fix $j \in J(F)$ and thereby its corresponding irreducible
representation $\pi_{(jS,j\theta)} \in \Pi_{\varphi}$.  The bijection
in (\ref{diag-1}) is now an isomorphism of pointed sets.

The dual group attached to
$\varphi\in \PhiSC(G)$ is $\mathrm{Cent}(\varphi(W_{F}),
\widehat{G})$, and  according to \cite[Lemma
  5.3.4]{Kaletha:Regular} it is naturally isomorphic to the
fixed-point subgroup $\widehat{S}^{\Gamma}$.   We denote the finite
abelian component group 
$\widehat{S}^{\Gamma}/ (\widehat{S}^{\Gamma})^{\circ}$ by
$\pi_{0}(\widehat{S}^{\Gamma})$, and denote its group of characters by
$\pi_{0}(\widehat{S}^{\Gamma})^{D}$.
What we wish to do here supplement (\ref{diag-1}) with an inclusion
\begin{equation}
  \label{diag0}
  \Pi_{\varphi} \longleftrightarrow J(F) \hookrightarrow
  \pi_{0}(\widehat{S}^{\Gamma})^{D} 
\end{equation}
and to describe the image of this inclusion.

Recall from (\ref{eq:Hmap}) that there is a bijection between $J(F)$ and
$\ker(H^{1}(F,jS) \rightarrow H^{1}(F,G))$. An arbitrary element of 
$\ker(H^{1}(F,jS) \rightarrow H^{1}(F,G))$ is represented by a cocycle
$$z_{g}(\sigma) = g^{-1}\,\sigma(g), \quad\sigma \in
\Gamma,$$
for some $g \in G$.  
By fixing $j \in J(F)$ as we have
above, we obtain a canonical bijection from $\ker(H^{1}(F,jS) \rightarrow
H^{1}(F,G))$ to $J(F)$ given by
\begin{equation}
  \label{canbij}
  z_{g} \mapsto \Ad(g)\circ j.
\end{equation}

The desired inclusion of (\ref{diag0}) is given through this canonical
bijection and the commutative diagram
\begin{equation}
  \label{kottwitz86}
\xymatrix{ H^{1}(F, jS)  \ar[r] \ar[d]&
  \pi_{0}(\widehat{S}^{\Gamma})^{D} \ar[d]\\
 H^{1}(F, G)  \ar[r] &
  \pi_{0}(Z(\widehat{G})^{\Gamma})^{D}
}
\end{equation}
of \cite[Theorem 1.2]{Kottwitz:1986} and \cite[Theorem
  2.1]{Thang:2011}. In this diagram the upper and lower maps are 
bijections which arise from perfect pairings in Tate-Nakayama duality.
The map on the left is given by the inclusion $jS 
\subset G$, and the map on the right is given by restriction to 
$Z(\widehat{G})^{\Gamma} \subset \widehat{S}^{\Gamma}$. 
 Since
(\ref{kottwitz86}) is commutative,  $\ker(H^{1}(F,jS) \rightarrow
 H^{1}(F,G))$ is in bijection with the 
kernel of the restriction map on the right of the diagram.  The kernel of the
restriction map is isomorphic
to $\pi_{0}(\widehat{S}^{\Gamma}/Z(\widehat{G})^{\Gamma})^{D}$.

Combining these observations with (\ref{canbij}), we obtain a
bijection between $J(F)$ and
$\pi_{0}(\widehat{S}^{\Gamma}/Z(\widehat{G})^{\Gamma})^{D}$ given
by the map
$$\Ad(g) \circ j \mapsto \tau_{g}$$
in which $\Ad(g) \circ j$ is a representative of a $G(F)$-conjugacy
class in $J(F)$ and 
$\tau_{g} \in
\pi_{0}(\widehat{S}^{\Gamma}/Z(\widehat{G})^{\Gamma})^{D}$ is obtained
through $z_{g}$ and Tate-Nakayama duality.  

In summary, the desired arrangement of (\ref{diag0}) takes the shape of three bijections
\begin{equation}
  \label{diag1}
\Pi_{\varphi} \longleftrightarrow J(F)  \longleftrightarrow
\left( \ker( H^{1}(F,jS) \rightarrow
H^{1}(F,G)) \right) \longleftrightarrow 
\pi_{0}(\widehat{S}^{\Gamma}/Z(\widehat{G})^{\Gamma})^{D}.
\end{equation}
On the level of elements the bijections have the form
$$\pi_{({^g}jS, {^g}j\theta)} \longleftrightarrow  \Ad(g) \circ j
\longleftrightarrow z_{g}
\longleftrightarrow \tau_{g},$$
where  $g \in G$  and $z_{g} \in Z^{1}(F,jS)$.  
 We can go one step further and obtain an alternative description for $\ker( H^{1}(F,jS) \rightarrow H^{1}(F,G))$ as follows.

Given a maximal torus $S'$ of $G$ which is defined over $F$, we have in particular that $S'$ is a closed subgroup of
$G$. The quotient $G/S'$ is therefore a variety defined
over $F$.
An element $gS' \in G/S'$ belongs to $(G/S')(F)$ if and only if $gS'
= \sigma(g)S'$ for all $\sigma \in \Gamma$.
The group $G(F)$ acts on $(G/S')(F)$ by left-multiplication.  Let
$G(F) \backslash \, (G /S')(F)$ denote the set of $G(F)$-orbits. 

\begin{lemma}
  \label{firstlem}
  Let $gS' \in (G/S')(F)$.  Then the map $z_{g} : \Gamma \rightarrow
  S'$ defined by
  $$z_{g}(\sigma) = g^{-1}\, \sigma(g), \quad \sigma \in \Gamma$$
  is a cocycle in $Z^{1}(F, S')$.  In addition, the map $g \mapsto
  z_{g}$ induces a bijection from $G(F) \backslash \, (G /S')(F)$ to
  $\ker( H^{1}(F,S') \rightarrow H^{1}(F,G))$.
\end{lemma}
\begin{proof}
For any $\sigma \in \Gamma$ we have $gS' = \sigma(g)S'$ and so 
$g^{-1}\sigma(g) \in S'$.  It is elementary to verify that $z_{g}$
satisfies the cocycle condition.  This yields a map from $(G/S')(F)$ to
$H^{1}(F,S')$.  The image of this map lies in $\ker( H^{1}(F,S') \rightarrow H^{1}(F,G))$, as
$z_{g}$ is a coboundary in $H^{1}(F,G)$.  On the other hand any
element in the kernel is a coboundary in $H^{1}(F,G)$ and therefore of
the form $z_{g}$ as above.  Consequently  $(G/S')(F)$ maps
onto  $\ker( H^{1}(F,S') \rightarrow H^{1}(F,G))$.
Finally if $g_{1}S', g_{2}S' \in (G/S')(F)$ and $z_{g_{1}}$ is
equivalent to $z_{g_{2}}$ in $H^{1}(F,S')$ then there exists $s \in S'$
such that
$$g_{1}^{-1} \sigma(g_{1}) = s^{-1}g_{2}^{-1} \sigma(g_{2})\sigma(s),
\quad \sigma \in \Gamma.$$
This in turn implies $g_{2}sg_{1}^{-1} = \sigma(g_{2}sg_{1}^{-1})$ and
$g_{2}sg_{1}^{-1} = g_{3}$ for some $g_{3} \in G(F)$.  It follows that
$g_{3}^{-1} g_{2}S' = g_{1}S'$, which is to say that $g_{1}S'$ and $g_{2}S'$
are in the same $G(F)$-orbit of $G(F) \backslash \, (G /S')(F)$.  This
proves the injectivity of the induced map.
\end{proof}

We may now rewrite (\ref{diag1}) using the bijection of Lemma
 \ref{firstlem} with $S' = jS$.
 \begin{equation}
   \label{diag2}
 \xymatrix{ \Pi_{\varphi} \ar@{<->}[r] & J(F)   \ar@{<->}[r] & G(F)
   \backslash \, (G /jS)(F) 
    \ar@{<->}[r] &
   \pi_{0}(\widehat{S}^{\Gamma}/Z(\widehat{G})^{\Gamma})^{D} \\ 
  \pi_{({^{g^{-1}}}jS, {^{g^{-1}}}j\theta)}   \ar@{<->}[r] &
   \Ad(g) \circ j  \ar@{<->}[r] & g \,jS
    \ar@{<->}[r] & \tau_{g}
   }
 \end{equation}

\subsubsection{A Commutative Diagram}

Assume that $\underT{\varphi}\in \Phi_{\mathrm{sc}}(\GT)$, $\varphi = {^L\eta} \circ \underT{\varphi}$ and that both $\paramT$ and $\param$ are regular. In light of the discussion following Corollary \ref{cor:regularparams}, it is not sufficient to assume that only $\paramT$ is regular.

We begin this section by describing a commutative diagram 

\begin{equation}
  \label{comdiag}
\xymatrix{  \Pi_{\underT{\varphi}} 
  \ar@{<->}[r]  & \underT{J}(F) 
  \ar@{<->}[r]^>(0.7){\mathrm{Ad}(\cdot)\circ 
    \underT{j}} &  \GT(F) \backslash \,
  (\GT / \underT{j}\underT{S})(F)
  \ar@{<->}[r]^<(0.2){\underT{\tau}} 
  &\pi_{0}(\widehat{\underT{S}}^{\Gamma} / Z(\widehat{\GT})^{\Gamma})^{D}
  \\
   \Pi_{\varphi}  \ar@{<->}[r] \ar@{->>}[u] & J(F)
   \ar@{<->}[r]_>(0.7){\mathrm{Ad}(\cdot) \circ j} 
   \ar@{->>}[u]_{\bar{\eta}}  &  G(F) \backslash \, (G / jS)(F)
   \ar@{<->}[r]_<(0.2){\tau}  \ar@{->>}[u]_{\bar{\eta}}
   &\pi_{0}(\widehat{S}^{\Gamma} / Z(\widehat{G})^{\Gamma})^{D}
   \ar@{->>}[u]_{\circ \hat{\eta}}
}
\end{equation}

Once this diagram is in place, the decomposition formula for $\underT{\pi}_{(\underT{j}\ST,\underT{j}\thetaT)}\circ \eta$ that one would obtain by applying Theorem~\ref{th:fullTheorem} may be transferred to the right-hand square of the diagram, which is key for proving Theorem \ref{th:Solleveld}.

The starting point is  the top row of (\ref{comdiag}).  Replacing $G$
with $\GT$ in 
Section \ref{embcompgp}, we fix (the
$\GT(F)$-conjugacy class of) an admissible
embedding $\underT{j} \in \JT(F)$ arbitrarily.  The top row is
then the sequence of bijections in (\ref{diag2}), in which we have
restored the distinction between $\ST$ and
$\underT{j}\ST$. By Proposition 
\ref{prop:embeddings}, there exists $j \in J(F)$ such 
that $\eta(jS) \subset \underT{j}\ST$ and ${j\theta =
  \underT{j}\thetaT\circ\eta}$.   For lack of a canonical choice,
we arbitrarily fix one such $j$.  The bottom row is then given by
(\ref{diag2}).

We continue by describing the middle two vertical maps of (\ref{comdiag}).
Since  $\eta(jS) \subset
\underT{j}\ST$  the map $\bar{\eta}$ sending $g\, jS \in G/jS$ to
$\eta(g)\, \underT{j}\underT{S} \in \GT/ \underT{j}\underT{S}$ is
well-defined.  Furthermore,  $\eta$ is defined over $F$ so that
$$\sigma(\bar{\eta}(g\, jS)) = \eta(\sigma(g))\, \underT{j}\underT{S} =
\bar{\eta}(\sigma (g\,jS)), \quad \sigma \in \Gamma,$$
This means that $\bar{\eta}$ is defined over $F$ and that
$\bar{\eta}$ passes to a map ${G(F) \backslash \, (G/jS)(F) \rightarrow
\GT(F) \backslash \, (\GT/\underT{j}\underT{S})(F)}$.
This defines the second map from the right in (\ref{comdiag}).  We
define the map $J(F) \rightarrow \underT{J}(F)$ to its left as the map
which takes $\Ad(g) \circ j$ to $\Ad(\eta(g)) \circ \underT{j}$.   In
this way the middle square in (\ref{comdiag}) commutes.  The
arguments in the proof of Proposition \ref{prop:embeddings} imply the
surjectivity of the two vertical maps in the middle square.

The vertical map  $\circ \widehat{\eta}$ on the right of (\ref{comdiag}) is defined by
composition with $\widehat{\eta}$ (see Proposition
\ref{prop:SSprime}).  
The commutativity of the right-hand square of (\ref{comdiag}) is explained as follows. The functoriality of \cite[Theorem 1.2]{Kottwitz:1986} and
\cite[Theorem 2.1]{Thang:2011} imply that the following diagram commutes:
$$
\xymatrix{ H^{1}(F, \underT{j}\underT{S})  \ar[r] &
  \pi_{0}(\widehat{\underT{S}}^{\Gamma})^{D} \\
 H^{1}(F, jS)  \ar[r] \ar[u]^{\eta \circ}&
  \pi_{0}(\widehat{S}^{\Gamma})^{D} \ar[u]_{\circ \hat{\eta}}
}
$$
Therefore, the restriction
$$
\xymatrix{ \mathrm{ker}(H^{1}(F, \underT{j}\underT{S})\rightarrow H^1(F,\GT))  \ar[r] &
  \pi_{0}(\widehat{\underT{S}}^{\Gamma}/Z(\hat{\GT})^\Gamma)^{D} \\
 \mathrm{ker}(H^{1}(F, jS) \rightarrow H^1(F,G))  \ar[r] \ar[u]^{\eta \circ}&
  \pi_{0}(\widehat{S}^{\Gamma}/Z(\hat{G})^\Gamma)^{D} \ar[u]_{\circ \hat{\eta}}
}
$$
also commutes. Applying Lemma \ref{firstlem} then gives us the right-hand square of (\ref{comdiag}) as a commutative diagram.
As a consequence, the
surjectivity of  $\circ
\widehat{\eta}$ follows from the surjectivity of $\bar{\eta}$.

The leftmost vertical arrow of (\ref{comdiag}) is defined as the unique map
which makes the leftmost square of (\ref{comdiag}) commute.
It is defined by  
$$ \pi_{({^gj}S, \,{^gj\theta})} \mapsto
\pi_{({^{\eta(g)}\underT{j}}\underT{S},
  \,{^{\eta(g)}\underT{j}\underT{\theta}})}$$
for all $\Ad(g) \circ j \in J(F)$.  This map is surjective, as $\bar{\eta}$
is surjective.  
Given $\Ad(\underT{g})\circ \underT{j}\in \JT(F)$ the preimage of $\pi_{((\Ad(\underT{g}) \circ \underT{j})\ST, (\Ad(\underT{g})\circ \underT{j})\thetaT)} \in \Pi_{\paramT}$
under this map is 
\begin{align}\label{eq:newParam1}
\{\pi_{({^gj}S, \,{^gj\theta})} : \Ad(g) \circ j \in
\bar{\eta}^{-1}(\Ad(\underT{g}) \circ \underT{j}) \}.
\end{align}

\subsubsection{The Proof of Theorem \ref{th:Solleveld}}

The commutative diagram (\ref{comdiag}) from the previous section is going to allow us to rewrite the decomposition formula provided in the proof of Theorem \ref{th:desideratum} with respect to the reparameterization in terms of characters of the component groups.

Let $\param,\paramT, \underT{j},j$ be as in the previous section.
Recall from the proof of Theorem \ref{th:desideratum} that
\begin{align}\label{eq:decompFormula}
\pi_{(\underT{j}\ST,\underT{j}\thetaT)} \circ \eta \simeq \underset{\underT{c} \in \underT{C}}{\oplus} \pi_{((\Ad(c^{-1})\circ j)S,(\Ad(c^{-1})\circ j)\theta)},
\end{align}
where $\underT{C}$ is a set of coset representatives of $\eta(G(F))\setminus \GT(F)/\underT{j}\ST(F)$ and $c\in G$ is such that $\underT{c} = \eta(c)z$ for some $z\in Z(\GT)$. Note                                                                                                                                                                                                                                                                                                                                                                                                                                                                                                                                                               that, by construction, $\Ad(c^{-1})\circ j \in \bar{\eta}^{-1}(\Ad(\underT{c}^{-1})\circ \underT{j})$ and $\Ad(\underT{c}^{-1})\circ \underT{j}$ and $\underT{j}$ belong to the same $\GT(F)$-equivalence class.

The following proposition tells us that the set of representations (\ref{eq:newParam1}) coincides with the irreducible components of the decomposition formula (\ref{eq:decompFormula}).
\begin{proposition}
  \label{leftsquare}
  Suppose $\underT{g} \, \underT{j}\ST \in (\GT/\underT{j}\underT{S})(F)$.  Then
\begin{equation}
  \label{rightsquare}
\pi_{({^{\underT{g}}\underT{j}}\underT{S},
  \,{^{\underT{g}}\underT{j}\underT{\theta}})} \circ \eta \cong \bigoplus_{g\, jS \in
\bar{\eta}^{-1}(\underT{g}\,\underT{j}\underT{S})}  \pi_{({^gj}S,
  \,{^gj\theta})}.
\end{equation}
Equivalently,  an
irreducible representation $\pi$ of $G(F)$ is a subrepresentation 
of $\pi_{({^{\underT{g}}\underT{j}}\underT{S},
  \,{^{\underT{g}}\underT{j}\underT{\theta}})} \circ \eta$ if 
and only if $\pi \cong   \pi_{({^gj}S, \,{^gj\theta})}$  for $g\,jS \in
\bar{\eta}^{-1}(\underT{g}\,\underT{j}\underT{S})$ (\emph{i.e.} $\Ad(g) \circ j$ in
the fibre of 
$\bar{\eta}$ over $\Ad(\underT{g}) \circ \underT{j}$).

\end{proposition}

The key to proving (\ref{rightsquare}) is to identify the fibre over $\underT{j}$ with the set $\eta(G(F))\setminus\GT(F)/\underT{j}\ST(F)$
in (\ref{eq:decompFormula}), which is done via the following lemma.

\begin{lemma}
  \label{lastcor}
Let $\bar{\eta}^{-1}(\underT{j})$ be the fibre of $J(F)
\stackrel{\bar{\eta}}{\rightarrow} \underT{J}(F)$ over $\underT{j}$.
The map of Lemma \ref{firstlem} and the map $\eta$ induce horizontal maps in
the commutative diagram
$$
\xymatrix{ J(F)  \ar[r] &
  G(F) \backslash \, (G /jS)(F) \ar[r]&
  \eta(G(F))  \backslash
  \, (\GT   / \underT{j}\underT{S})(F)\\
 \bar{\eta}^{^{-1}}(\underT{j})  \ar[r] \ar@{^{(}->}[u]&
 G(F)  \backslash \,
 \bar{\eta}^{-1}(\GT(F)\underT{j}\underT{S}/\underT{j}\underT{S})
 \ar@{^{(}->}[u] \ar[r] & 
    \eta(G(F))\backslash \, \GT(F)/ \underT{j}\underT{S}(F)  \ar@{^{(}->}[u]
}
$$
In addition, the horizontal maps are bijections.
\end{lemma}
\begin{proof}
We first prove the assertion for the square on the left.  The upper
map is bijective by Lemma \ref{firstlem} and the vertical maps are
inclusions.  For the lower horizontal map, suppose $\mathrm{Ad}(g)
\circ j \in J$ is a representative 
of some 
$G(F)$-orbit in $J(F)$ where $g\,jS \in (G/jS)(F)$, and $\bar{\eta} ( 
\mathrm{Ad}(g) \circ j) = \mathrm{Ad}(\underT{g}) \circ \underT{j}$ for
some $\underT{g} \in \GT(F)$.  The bijectivity of the lower horizontal map
follows from the  equivalences
$$
  \mathrm{Ad}(\underT{g}^{-1} \eta(g)) \circ \underT{j}
  = \underT{j}
   \Leftrightarrow \underT{g}^{-1} \eta(g) \in \underT{j}\underT{S}
 \Leftrightarrow \eta(g)  \underT{j}\underT{S} = \underT{g}\,
 \underT{j}\underT{S}
   \Leftrightarrow g\,jS \in \bar{\eta}^{-1}( \GT(F)
  \underT{j}\ST/\underT{j}\ST) . 
  $$
  We continue by examining the upper horizontal map in the square on the
  right. This map induced by $\eta$ is clearly well-defined.  
For its surjectivity, suppose $\underT{g} \in \GT$.  Since
$\eta(G) \supset [\GT, \GT]$, we have
$\GT = \eta(G) \, Z(\GT)$.  It follows in turn that $\underT{g} = \eta(g)
z$ for some $g \in G$ and $z \in Z(\GT) \subset \underT{j}\underT{S}$, and
$$\eta(G(F)) \, \eta(g) \, \underT{j}\underT{S} = \eta(G(F)) \, \underT{g} z^{-1}\,
\underT{j}\underT{S} = \eta(G(F)) \, \underT{g} \, \underT{j}\underT{S}.$$
If $\underT{g}\, \underT{j}\underT{S} \in (\GT/\underT{j}\underT{S})(F)$ then
it is fixed by any $\sigma \in \Gamma$ and
$$\eta(\sigma(g))\,\underT{j}\underT{S} = \sigma(\eta(g)\, 
\underT{j}\underT{S}) = \sigma(\underT{g} \,\underT{j}\underT{S}) = \underT{g} \,
\underT{j}\underT{S} = 
\eta(g)\, \underT{j}\underT{S}.$$
The latter implies that $\eta(\sigma(g)g^{-1}) \in  \underT{j}\underT{S}
$ and in turn that $\sigma(g)\, jS = g\, jS$. 
This proves that $g\,jS \in (G/jS)(F)$ and that the induced map is surjective.

To prove that the map is injective, suppose $g_{1}\,jS , g_{2}\,jS \in (G/jS)(F)$ and
$$\eta(G(F)) \, \eta(g_{1}) \, \underT{j}\underT{S} = \eta(G(F))\,  \eta(g_{2}) 
\, \underT{j}\underT{S}.$$
Then $\eta(g_{1}) = \eta(g) \, \eta(g_{2}) \underT{s}$ for some $g \in G(F)$ and
$\underT{s} \in \underT{j}\underT{S}$.  This implies
$$
\eta(g_{2}^{-1}g^{-1}g_{1}) = \underT{s} \in \underT{j}\underT{S}\\
   \Rightarrow g_{2}^{-1}  g^{-1} g_{1} \in jS\\
   \Rightarrow G(F)\, g_{1}\, jS = G(F)\, g_{2}\, jS.
$$
   This establishes the bijectivity of the upper horizontal map.

Finally, the lower horizontal map is defined by restricting the upper
horizontal map.  This yields a bijection
$$G(F) \backslash \, \bar{\eta}^{-1}(\GT(F)
  \underT{j}\underT{S}/\underT{j}\underT{S} ) 
  \rightarrow \eta(G(F)) \backslash \,
  \GT(F)\underT{j}\underT{S}/\underT{j}\underT{S},$$ 
and the set on the right may be identified with $ \eta(G(F))\backslash
\, \GT(F)/ \underT{j}\underT{S}(F)$ under the apparent bijection
  $$\GT(F)/ \underT{j}\underT{S}(F) \rightarrow  \GT(F)
  \underT{j}\underT{S}/ 
  \underT{j}\underT{S}.$$ 
\end{proof}

The following lemma is an immediate consequence of Lemma
\ref{lastcor} and decomposition formula (\ref{eq:decompFormula}).
\begin{lemma}
  \label{rightsquare1}
An irreducible representation $\pi$ of $G(F)$ is a subrepresentation
of $\pi_{(\underT{j}\underT{S}, \underT{j}\underT{\theta})} \circ \eta$ if
and only if $\pi \cong   \pi_{({^gj}S, \,{^gj\theta})}$  for $gjS \in
\bar{\eta}^{-1}(\underT{j}\underT{S})$ ($\Ad(g) \circ j$ is in the fibre of
$\bar{\eta}$ over $\underT{j}$).   In particular, (\ref{rightsquare})
holds for $\underT{g} = 1$.
\end{lemma}

We are now ready to prove Proposition \ref{leftsquare} 
\begin{proof}[Proof of Proposition \ref{leftsquare}]
Let $\underT{j}' = \Ad(\underT{g}) \circ \underT{j}$.  Arguing as in the
proof of Proposition \ref{prop:embeddings} there exists $g'S \in
(G/S)(F)$ such that $j' = \Ad(g') \circ j$ is sent to $\underT{j}'$
under $\bar{\eta}$.  We may replace $\underT{j}$ with $\underT{j}'$, and
$j$ with $j'$ in the earlier results. Lemma \ref{rightsquare1}
then tells us that the irreducible subrepresentations of
$$\pi_{({^{\underT{g}}\underT{j}}\underT{S},
  \,{^{\underT{g}}\underT{j}\underT{\theta}})} \circ \eta =
\pi_{(\underT{j}'\underT{S}, 
  \,\underT{j}'\underT{\theta})} \circ \eta $$
are $\pi_{({^hj'}S, \,{^hj'\theta})} = \pi_{({^{hg}j}S, \,{^{hg}j\theta})}$, where $hj'S \in
\bar{\eta}^{-1}(\underT{j}'\underT{S})$.  The corollary now follows from 
$$hj'S \in \bar{\eta}^{-1}(\underT{j}'\underT{S}) \Leftrightarrow
\eta(hg'jS(g')^{-1}) \subset \underT{g} \underT{j}\underT{S} \underT{g}^{-1}
\Leftrightarrow \eta(hg')\underT{j}\underT{S} =
\underT{g}\underT{j}\underT{S}
\Leftrightarrow (hg')jS \in \bar{\eta}^{-1}(\underT{g}\underT{j}
\underT{S})$$
and setting $g = hg'$.
\end{proof}

A simple consequence of Proposition \ref{leftsquare} and diagram
(\ref{comdiag}) is the following corollary, analogous to \cite[Corollary 5.8]{Solleveld:2020}.

\begin{corollary}
The $L$-packet $\Pi_{\varphi}$ consists of the irreducible
representations appearing on the right of (\ref{rightsquare}).  More precisely,
$$\Pi_{\varphi} = \coprod_{\Ad(\underT{g}) \circ \underT{j} \in
  \underT{J}(F)} \{\pi_{({^gj}S, \,{^gj\theta})} : \Ad(g) \circ j \in 
\bar{\eta}^{-1}(\Ad(\underT{g}) \circ \underT{j}) \},$$
or equivalently, $\Pi_{\param} = \{ [\underT{\pi}\circ \eta] : \underT{\pi}\in \Pi_{\paramT}\}.$  
\end{corollary}
\begin{proof}
The corollary follows from the commutativity of (\ref{comdiag}) and the
partition
$$J(F) =  \coprod_{\Ad(\underT{g}) \circ \underT{j} \in
  \underT{J}(F)}  \bar{\eta}^{-1}(\Ad(\underT{g}) \circ \underT{j}).$$ 
\end{proof}

We continue by expressing Proposition \ref{leftsquare}, which concerns
the left-hand 
side of diagram (\ref{comdiag}), in terms of the characters 
of the dual groups, which appear on the right of the diagram.
\begin{corollary}
  \label{sollconj1}
Suppose $\underT{g}\, \underT{j}\ST \in (\GT/\underT{j}\underT{S})(F)$
and $g \, jS \in
(G/jS)(F)$.  Then $\pi_{({^gj}S, \,{^gj\theta})}$ is a subrepresentation 
of $\pi_{({^{\underT{g}}\underT{j}}\underT{S},
  \,{^{\underT{g}}\underT{j}\underT{\theta}})} \circ \eta$ if 
and only if $\tau_{g} \circ \widehat{\eta} = \underT{\tau}_{\underT{g}}$.  In
addition,
$$\pi_{({^{\underT{g}}\underT{j}}\underT{S},
  \,{^{\underT{g}}\underT{j}\underT{\theta}})} \circ \eta \cong \bigoplus_{g
  \in  G(F) \backslash \, (G / jS)(F)}   \mathrm{Hom}(
\underT{\tau}_{\underT{g}}, \tau_{g} \circ \widehat{\eta})
\otimes \, \pi_{({^gj}S, \,{^gj\theta})}.$$ 
\end{corollary}

\begin{remark}
Since $\underT{\tau}_{\underT{g}}$ and $\tau_g$ are characters, we
have that $\dim \mathrm{Hom}(\underT{\tau}_{\underT{g}}, \tau_{g}\circ
\hat{\eta})$ is either equal to $1$ or $0$. As such,
$\pi_{({^{\underT{g}}\underT{j}}\underT{S}, 
  \,{^{\underT{g}}\underT{j}\underT{\theta}})} \circ \eta$ is multiplicity free for all $\underT{g}\, \underT{j}\ST \in (\GT/\underT{j}\underT{S})(F)$. One can also prove that the decomposition is multiplicity free using tools directly from the classification theory of supercuspidal representations such as \cite{HM:2008,Murnaghan:2011}  (as done for instance in \cite[Section 6]{thesisPaper}).
\end{remark}

Yet another manner of expressing the decompositions of the corollaries
is to set $\underT{\varrho} = \underT{\tau}_{\underT{g}}$ and set
$\pi(\underT{\varphi}, \underT{\varrho}) =  \pi_{({^{\underT{g}}\underT{j}}\underT{S}, 
  \,{^{\underT{g}}\underT{j}\underT{\theta}})}$.  Then the decomposition
of Corollary \ref{sollconj1} reads as
\begin{equation}
  \label{plaineq}
\pi(\underT{\varphi}, \underT{\varrho}) \circ \eta \cong \bigoplus_{
  \varrho \in \pi_{0}(\widehat{S}^{\Gamma} /
  Z(\widehat{G})^{\Gamma})^{D}}   \mathrm{Hom}(
\underT{\varrho}, \varrho  \circ \widehat{\eta})
\otimes \, \pi(\varphi, \varrho),
\end{equation}
thus completing the proof of Theorem \ref{th:Solleveld}.

This form of the decomposition is the one proposed by Solleveld in
\cite[Conjecture 2]{Solleveld:2020} when $\eta$ is chosen to preserve
fixed pinnings of $G$ and $\GT$ as we have done in 
Section \ref{sec:packetData}.

In Section \ref{sec:packetData} we also
remarked that dropping the requirement of preserving the pinnings, but
keeping the dual homomorphism $\widehat{\eta}$ fixed, allows one to replace
$\eta$ with $\eta' = \Ad(\underT{t}') \circ \eta$ where
$\underT{t}'Z(\GT) \in (\underT{T}/Z(\GT))(F)$.  
It is convenient to write $\underT{t} = (\underT{t}')^{-1}$, for in this
arrangement 
$$
\pi_{({^{\underT{g}}\underT{j}}\underT{S}, 
  \,{^{\underT{g}}\underT{j}\underT{\theta}})} \circ \eta'  =
\pi_{({^{\underT{g}}\underT{j}}\underT{S}, 
  \,{^{\underT{g}}\underT{j}\underT{\theta}})} \circ \Ad(\underT{t}')
\circ \eta
= \pi_{({^{\underT{t}\underT{g}}\underT{j}}\underT{S},
  \,{^{\underT{t}\underT{g}}\underT{j}\underT{\theta}})} \circ \eta$$
and Corollary \ref{sollconj1} yields
\begin{equation}
 \label{sollconj1.5}
\pi_{({^{\underT{g}}\underT{j}}\underT{S}, 
  \,{^{\underT{g}}\underT{j}\underT{\theta}})} \circ \eta' \cong \bigoplus_{g
  \in  G(F) \backslash \, (G / jS)(F)}   \mathrm{Hom}(
\underT{\tau}_{\underT{t}\underT{g}}, \tau_{g} \circ \widehat{\eta})
\otimes \, \pi_{({^gj}S, \,{^gj\theta})}.
\end{equation}
This decomposition can be rephrased in terms of characters on the dual
group as in (\ref{plaineq}).  The introduction of $\Ad(\underT{t}')$ on the
left of (\ref{plaineq}) introduces the character
$\underT{\tau}_{\underT{t}'}$ on the right, as one sees in the following corollary.
\begin{corollary}
  \label{sollconj2}
Suppose  $\underT{g} \in (\GT/\underT{j}\underT{S})(F)$ and $\eta'
= \Ad(\underT{t}') \circ \eta$, where 
$\underT{t}'Z(\GT) \in (\underT{T}/Z(\GT))(F)$.  Set
$\underT{\varrho} = \underT{\tau}_{\underT{g}}$,
$\pi(\underT{\varphi},\underT{\varrho}) = \pi_{({^{\underT{g}}\underT{j}}\underT{S}, 
  \,{^{\underT{g}}\underT{j}\underT{\theta}})}$.  Then
$$
\pi(\underT{\varphi}, \underT{\varrho}) \circ \eta' \cong \bigoplus_{
  \varrho \in \pi_{0}(\widehat{S}^{\Gamma} /
  Z(\widehat{G})^{\Gamma})^{D}}   \mathrm{Hom}(
\underT{\varrho}, (\varrho  \circ \widehat{\eta}) \otimes \underT{\tau}_{\underT{t}'})
\otimes \, \pi(\varphi, \varrho).
$$
\end{corollary}
\begin{proof}
The character
$\underT{\tau}_{\underT{t}\underT{g}}$ appearing in  (\ref{sollconj1.5})
corresponds to the  
element $\underT{t}\underT{g} \in (\GT/\underT{j}\underT{S})$.
According to Lemma \ref{firstlem} (applied to $\GT$ and $\underT{j}\ST$), the element $\underT{t}\underT{g}$
represents the cocycle $z_{\underT{t}\underT{g}}$ defined by
$$z_{\underT{t}\underT{g}}(\sigma) = (\underT{t} \underT{g})^{-1} \,
\sigma(\underT{t} \underT{g}), \quad \sigma \in \Gamma.$$ 
Since $\underT{t} Z(\GT) \in (\underT{T}/Z(\GT))(F)$, we
have in turn that $\underT{t}^{-1} \sigma(\underT{t}) \in Z(\GT)$
and 
$$(\underT{t} \underT{g})^{-1} \,\sigma(\underT{t} \underT{g}) =
\underT{g}^{-1} \, (\underT{t}^{-1} \sigma(\underT{t}) ) \,
\sigma(\underT{g}) =  \underT{t}^{-1} \sigma(\underT{t})  \,
\underT{g}^{-1}\sigma(\underT{g}).$$
Therefore $z_{\underT{t}\underT{g}} = z_{\underT{t}} \,z_{\underT{g}}$ in
the group $Z^{1}(F, jS)$.  Applying $\underT{\tau}$, we obtain
$\underT{\tau}_{\underT{t} \underT{g}} = \underT{\tau}_{\underT{t}}
\otimes \underT{\tau}_{\underT{g}}$.  A similar argument leads to $1 =
\underT{\tau}_{\underT{t}} \otimes \tau_{\underT{t}^{-1}}$, and so
$\underT{\tau}_{\underT{t}}^{-1} = \underT{\tau}_{\underT{t}'}$.  By
setting $\varrho = \tau_{g}$ in (\ref{sollconj1.5}) we see that
$$ \mathrm{Hom}(\underT{\tau}_{\underT{t}\underT{g}}, \tau_{g} \circ
\widehat{\eta}) = \mathrm{Hom}(\underT{\tau}_{\underT{t}}
\underT{\tau}_{\underT{g}}, \varrho \circ \widehat{\eta}) =
\mathrm{Hom}(\underT{\varrho}, (\varrho \circ \widehat{\eta}) \otimes
\underT{\tau}_{\underT{t}'}). 
$$
The corollary now follows from (\ref{sollconj1.5}) and the
commutativity of (\ref{comdiag}).
\end{proof}
The decomposition of Corollary \ref{sollconj2} resembles the one appearing in
Solleveld's conjecture  \cite[Conjecture 2]{Solleveld:2020}.  The only
difference is the appearance of ${^S\eta'}^{*}(\varrho)$ instead of
$(\varrho  \circ \widehat{\eta}) \otimes \underT{\tau}_{\underT{t}'}$.
Translating into our setting,  \cite[(5.4)]{Solleveld:2020} expresses
${^S\eta'}^{*}(\varrho)$ as
$${^S\eta'}^{*}(\varrho) = (\varrho \circ \widehat{\eta}')\otimes
\tau_{\underT{\varphi}}(t') = (\varrho \circ \widehat{\eta})\otimes
\tau_{\underT{\varphi}}(t')$$ 
where $\tau_{\underT{\varphi}}: \GT(F) \backslash \,
(\GT/Z(\GT))(F) \rightarrow
\pi_{0}(\widehat{\underT{S}}^{\Gamma}/Z(\widehat{\GT})^{\Gamma})^{D}$
is a homomorphism defined in \cite[(2.12), Lemma 2.1]{Solleveld:2020}. 

\begin{proposition}\label{prop:otherpinning}
Under the regularity assumptions
on $\param$ and $\paramT$, Corollary \ref{sollconj2} coincides with
\cite[Conjecture 2]{Solleveld:2020}. That is,  
$(\varrho \circ \widehat{\eta}) \otimes \underT{\tau}_{\underT{t}'} = {^S\eta'}^{*}(\varrho),$
and therefore,
$$\begin{aligned}
 \pi(\underT{\varphi}, \underT{\varrho}) \circ \eta' &\cong \bigoplus_{
  \varrho \in \pi_{0}(\widehat{S}^{\Gamma} /
  Z(\widehat{G})^{\Gamma})^{D}}   \mathrm{Hom}(
\underT{\varrho}, (\varrho  \circ \widehat{\eta}) \otimes \underT{\tau}_{\underT{t}'})
\otimes \, \pi(\varphi, \varrho)\\
&= \bigoplus_{
  \varrho \in \pi_{0}(\widehat{S}^{\Gamma} /
  Z(\widehat{G})^{\Gamma})^{D}}   \mathrm{Hom}(
\underT{\varrho}, {^S\eta'}^{*}(\varrho))
\otimes \, \pi(\varphi, \varrho)
\end{aligned}$$
\end{proposition}

\begin{proof}
All we need to show is
\begin{equation}
  \label{taus}
  \tau_{\underT{\varphi}}(\underT{t}') = \underT{\tau}_{\underT{t}'}.
\end{equation}
Under the assumption that the characteristic of $F$ is zero, this
identity is given  in the proof of \cite[Lemma
  4.2]{Kaletha:2013} in which Kaletha 
writes $\tau_{\underT{\varphi}}(\underT{t}')$ as $(\mathfrak{w},
\mathfrak{w}')$ for a pair of
Whittaker data conjugate under $\underT{t}'$ \cite[\emph{pp.}
  2454-2455]{Kaletha:2013}, and $\underT{\tau}_{\underT{t}'}$ is
expressed in terms of the cocycle $z = z_{\underT{t}'}$ and the
Tate-Nakayama pairing (Lemma \ref{firstlem}).  His argument references
the work in  \cite[Appendix A]{ks} on
the hypercohomology of complexes of tori of length two.  The argument is the
same in positive characteristic, and to convince the reader that
nothing runs awry we offer a sketch.

We write a complex of $F$-tori of length two simply as $\mathcal{T}
\rightarrow \mathcal{S}$, concentrated in degrees $0$ and $1$.
Let $\mathcal{T}$ be a maximal torus is $\underT{G}$ which is defined
over $F$.  Let $\underT{Z} = Z(\underT{G})$ and 
$\mathcal{T}_{\mathrm{ad}} = \mathcal{T}/\underT{Z}$.  Then
$\mathcal{T}_{\mathrm{ad}}$ is a maximal torus in
$\underT{G}_{\mathrm{ad}} = \underT{G}/\underT{Z}$, and its Langlands dual is
$(\hat{\mathcal{T}})_{\mathrm{sc}}$, a maximal torus in the
simply-connected dual group $(\hat{\underT{G}})_{\mathrm{sc}} =
\widehat{\underT{G}_{\mathrm{ad}}}$. The
map $(\hat{\underT{G}})_{\mathrm{sc}} \rightarrow [\hat{\underT{G}},
  \hat{\underT{G}}] \rightarrow \hat{G}$ induces a map
$(\hat{\mathcal{T}})_{\mathrm{sc}} \rightarrow \hat{\mathcal{T}}$ with
kernel denoted by $\hat{\underT{Z}}$.

The sequence
$$0 \longrightarrow (0 \rightarrow \mathcal{T}_{\mathrm{ad}})
\stackrel{(0,id)}{\longrightarrow} (\mathcal{T} \rightarrow
\mathcal{T}_{\mathrm{ad}}) \stackrel{(id,0)}{\longrightarrow} (\mathcal{T}
\rightarrow 0 ) \longrightarrow 0$$
is a short exact sequence, and therefore gives rise
to a long exact sequence of Galois hypercohomology.  The first
hypercohomology portion of this long exact sequence appears in the second row
of the following diagram.
\begin{equation}
 \label{bigdiag}
  \xymatrix@1{\mathcal{T}_{\mathrm{ad}}(F) \ar[r] \ar[d]_{\cong} &
    H^{1}(F, \underT{Z})      \ar[r] \ar[d]_{\cong} & H^{1}(F,
    \mathcal{T}) \ar[d]_{\cong}\\  
H^{1}(F,0 \rightarrow \mathcal{T}_{\mathrm{ad}}) \ar[r] \ar[d] &
H^{1}(F, \mathcal{T} \rightarrow \mathcal{T}_{\mathrm{ad}}) \ar[r]
\ar[d] & H^{1}(F,  \mathcal{T} \rightarrow 0)  \ar[d] \\
H^{1}(W_{F}, (\hat{\mathcal{T}})_{\mathrm{sc}} \rightarrow 0)^{D}
\ar[r] \ar[d]_{\cong} &  H^{1}(W_{F}, (\hat{\mathcal{T}})_{\mathrm{sc}} \rightarrow
\hat{\mathcal{T}})^{D} \ar[r] \ar[d]_{\cong} & H^{1}(W_{F}, 0
\rightarrow \hat{\mathcal{T}})^{D} \ar[d]_{\cong}\\
H^{1}(W_{F}, (\hat{\mathcal{T}})_{\mathrm{sc}})^{D} \ar[r] & H^{1}(W_{F},
\hat{\underT{Z}})^{D} \ar[r] & (\hat{\mathcal{T}}^{W_{F}})^{D}
  }
\end{equation}
The third row of the diagram is the Pontryagin dual of the analogous
sequence of the dual tori together with the action of the Weil group on the dual
tori.  The vertical arrows between the second and third rows are given
by the pairing \cite[(A.3.15)]{ks} (\emph{cf.} \cite[Proposition A.4]{Dillery}).  

The first row of the diagram is defined as follows.  Any element in
$\mathcal{T}_{\mathrm{ad}}(F)$ may be written as $t'\underT{Z}$ where $t'
\in \mathcal{T}$.  The first horizontal map sends
$t'\underT{Z}$ to the class of the 1-cocycle $z_{t'}$ defined by
\begin{equation}
  \label{1cocyle}
  z_{t'}(\sigma) = (t')^{-1} \, \sigma(t'), \quad \sigma \in \Gamma.
  \end{equation}
 The second horizontal map
 in the top row carries $z_{t'}$ to itself.
The vertical isomorphisms between the the first and second rows are
canonical and left as  exercises (\emph{cf.} A.1 \cite{ks}).

The maps of the fourth row and the
isomorphisms with the the third 
row follow just as the ones for the first and second rows.  Starting
with $h \in \hat{\mathcal{T}}^{W_{F}}$, we choose $h' \in
(\hat{\mathcal{T}})_{\mathrm{sc}}$ so that $h = h' \hat{\underT{Z}}$.
We then define an  element $c_{h'}$ in $H^{1}(W_{F},
\hat{\underT{Z}})$ or $H^{1}(W_{F}, (\hat{\mathcal{T}})_{\mathrm{sc}})$ by
imitating (\ref{1cocyle})
(\emph{cf.}~\cite[(2.8)]{Solleveld:2020}).  The maps in the fourth
row are the ones dual to those just defined.

Diagram (\ref{bigdiag}) is commutative, due to the functoriality of
the vertical morphisms (\emph{cf.} \cite[(A.3.5)]{ks}).  By making a
comparison with the cohomology crossed modules, one can also see that
the map $H^{1}(F, \underT{Z}) \rightarrow H^{1}(W_{F},
\hat{\underT{Z}})^{D}$ in the middle of (\ref{bigdiag}) is independent
of the choice of maximal torus $\mathcal{T}$ (see the proof of
\cite[Proposition 5.19]{Kaletha:2015}).
Combining these facts with $\mathcal{T} = \underT{T}$ and
$\mathcal{T} =j\underT{S}$,  we obtain the commutative diagram 
$$
\xymatrix@1{\underT{T}_{\mathrm{ad}}(F)  \ar[d] \ar[r]& 
H^{1}(F, \underT{Z} )  \ar[r] \ar[d]
 & H^{1}(F, j \underT{S})  \ar[d] \\
H^{1}(W_{F}, (\hat{\underT{T}})_{\mathrm{sc}})^{D}
\ar[r]  &  H^{1}(W_{F}, \hat{\underT{Z}})^{D} \ar[r]  &
(\hat{\underT{S}}^{\varphi(W_{F})} )^{D} 
}$$
Consider the element $t'\underT{Z} \in
\underT{T}_{\mathrm{ad}}(F)$ in the top left of this diagram.  If we
trace $t' \in \underT{T}$ through the vertical map on the left and the
lower horizontal maps then, we arrive at
$\tau_{\underT{\varphi}}(t')$ (\cite[(A.3.13)]{ks},
\cite[(2.11)]{Solleveld:2020}).  Alternatively, tracing $t' \in
\underT{T}$ through the upper horizontal maps followed by the vertical
map on the right we arrive at $\tau_{t'}$ (\cite[(A.3.14)]{ks}).  The
commutativity of the diagram yields the desired identity (\ref{taus}).

\end{proof}

\appendix
\section{Appendices}

\subsection{Some Intertwining Maps and Combining Cosets}\label{sec:appendix}

This appendix groups together some important results that were cited in the main text, but whose proofs did not necessarily fit with the flow of the main narrative.

\begin{proposition}\label{prop:conjInd}
Let $G$ and $\GT$ be as per Theorem \ref{th:desideratum}, and let $G_1\subseteq G_2 \subseteq G(F)$. Let $\underT{g}\in\GT(F)$ and let $\AdT(\underT{g})$ be the automorphism of $G(F)$ as defined in Section \ref{sec:MainTheorem}. Then, given a representation $\pi$ of $G_1$,
$$\Ind_{\AdT(\underT{g})(G_1)}^{\AdT(\underT{g})(G_2)}\left( \pi \circ \AdT(\underT{g}^{-1})\right) \simeq \Ind_{G_1}^{G_2}\pi \circ \AdT(\underT{g}^{-1}).$$
\end{proposition}

\begin{proof}
Let $V$ denote the vector space on which the representations $\pi$ and $\pi\circ \AdT(\underT{g}^{-1})$ act. Then, $\Ind_{\AdT(\underT{g})(G_1)}^{\AdT(\underT{g})(G_2)}\left( \pi\circ \AdT(\underT{g}^{-1}) \right)$ and $\Ind_{G_1}^{G_2}\pi\circ\AdT(\underT{g}^{-1})$ act on the vector spaces
$$\underT{W} = \{ \underT{f} : \AdT(\underT{g})(G_2)\rightarrow V \text{ locally constant } | \underT{f}(gh) = (\pi\circ\AdT(\underT{g}))(h)^{-1}\underT{f}(g) \text{ for all } h\in \AdT(\underT{g})(G_1) \}$$
and
$$W = \{ f: G_2\rightarrow V \text{ locally constant } | f(gh) = \pi(h)^{-1}f(g) \text{ for all } h\in G_1\},$$
respectively. Define a linear map $\mathcal{F}$ as follows.
$$\begin{aligned}
\mathcal{F}: W &\rightarrow \underT{W}\\
f &\mapsto f\circ \AdT(\underT{g}^{-1}).
\end{aligned}$$
The map $\mathcal{F}$ is bijective as a consequence of $\AdT(\underT{g}^{-1}): \AdT(\underT{g})(G_2) \rightarrow G_2$ being bijective. It is then an easy computation to verify that $\mathcal{F}$ intertwines the representations $\Ind_{\AdT(\underT{g})(G_1)}^{\AdT(\underT{g})(G_2)}\left(\pi\circ \AdT(\underT{g}^{-1}) \right)$ and $\Ind_{G_1}^{G_2}\circ \AdT(\underT{g}^{-1})$. 
\end{proof}

\begin{proposition}\label{prop:induction}
Let $\mu: H'_2 \rightarrow H_2$ be a morphism of groups, $H_1\subset H_2$ and $H_1'\subset H_2'$ subgroups such that $\mu(H_2')\cap H_1 = \mu(H_1')$ and $\ker(\mu)\subset H_1'$. Let $\pi$ be a representation of $H_1$. Then,
$$\Ind_{H_1'}^{H_2'}(\pi\circ \mu) \simeq \left( \Ind_{\mu(H_1')}^{\mu(H_2')}\Res_{\mu(H_1')}^{H_1}\pi \right) \circ \mu.$$  
\end{proposition}

\begin{proof}
Let $V$ denote the vector space on which the representations $\pi$ and $\pi\circ\mu$ act. Then, the representations $\left(\Ind_{\mu(H_1')}^{\mu(H_2')}\Res^{H_1}_{\mu(H_1')}\pi\right)\circ \mu$ and $\Ind_{H_1'}^{H_2'}(\pi\circ\mu)$ act on the vector spaces 
$$W_\mu = \{f_\mu: \mu(H_2') \rightarrow V \text{ locally constant } | f_\mu(gh) = \pi(h)^{-1}f_\mu(g) \text{ for all } h\in \mu(H_1')\}$$ and $$W = \{f: H_2' \rightarrow V \text{ locally constant } | f(gh) = \pi\circ\mu(h)^{-1}f(g) \text{ for all } h\in H_1'\},$$ respectively. Define a linear map $\mathcal{F}$ as follows:
\begin{align*}
\mathcal{F}: W_\mu &\rightarrow W\\
f_\mu &\mapsto f_\mu\circ \mu.
\end{align*} 
One sees that the map $\mathcal{F}$ is injective, as $\mu : H_2' \rightarrow \mu(H_2')$ is surjective. Since $\ker\mu \subset H_1'$, elements of $W$ are constant on the cosets of $H_2'/\ker\mu$, allowing us to define an inverse map and conclude that $\mathcal{F}$ is bijective. It is then an easy computation to verify that $T$ intertwines the representations $\left(\Ind_{\mu(H_1')}^{\mu(H_2')}\Res_{\mu(H_1')}^{H_1}\pi \right)\circ \mu$ and $\Ind_{H_1'}^{H_2'}(\pi\circ\mu)$.

\end{proof}

\begin{lemma}\label{lem:modifiedisothm}
Let $A,B,C$ be subgroups such that $A\subseteq B$. Then 
$$\bigslant{AC/C}{BC/C} \simeq AC/BC \simeq \bigslant{A}{B(A\cap C)}.$$
\end{lemma}

\begin{proof}
The first isomorphism comes from the third isomorphism theorem. For the second isomorphism, consider the map
$$\begin{aligned}
A &\rightarrow AC/ BC \\
a &\mapsto aBC.
\end{aligned}$$ This map is clearly surjective. Furthermore, its kernel is given by
$$\{a\in A : aBC = BC\} = \{a\in A : a\in BC\} = A\cap BC = B(A\cap C).$$ The conclusion then follows from the first isomorphism theorem.
\end{proof}

\begin{lemma}\label{lem:combineCosets}
Let $A,B,C,\bar{A},\bar{B},\bar{C}$ be groups that satisfy the following conditions. $A,C\subseteq B$, $\bar{A},\bar{C}\subseteq\bar{B}$ and $\bar{C}\subseteq C$. Furthermore, assume $A$ and $\bar{A}$ are normal subgroups of $B$ and $\bar{B}$, respectively. Let $L$ be a set of coset representatives of $A\setminus B / C$ and $\mathcal{L}$ be a set of coset representatives of $\bar{A}\setminus \bar{B} / \bar{C}$. Write $L\mathcal{L} = \{l\ell : l\in L, \ell \in \mathcal{L}\}$.
\begin{enumerate}
\item[1)] If $\bar{B}\subseteq C$, $A\cap \bar{B} \subseteq \bar{A}$ and $\bigslant{C}{(A\cap C)\bar{C}}\simeq \bigslant{\bar{B}}{\bar{A}\bar{C}}$, then $L\mathcal{L}$ is a set of coset representatives of $A\setminus B / \bar{C}$.
\item[2)] If $\bar{A}\subseteq A$, $A\cap C = \bar{C}$, $A = \bar{B}$ and $\bar{A}$ is a normal subgroup of $B$, then $L\mathcal{L}$ is a set of coset representatives of $\bar{A} \setminus B/C$.
\end{enumerate}
\end{lemma}

\begin{proof}
Since $A$ is normal in $B$, and $\bar{A}$ is normal in $\bar{C}$, we note that $A\setminus B / C = B/AC, A\setminus B /\bar{C} = B/A\bar{C}$ and $\bar{A}\setminus\bar{B} /\bar{C} = \bar{B}/\bar{A}\bar{C}$.

To prove part 1), we use the normality of $A$, the inclusion $\bar{C}\subseteq C$ and the isomorphism $\bigslant{C}{(A\cap C)\bar{C}}\simeq \bigslant{\bar{B}}{\bar{A}\bar{C}}$, along with the second and third isomorphism theorems to show that 
$$B/AC \simeq \bigslant{B/A\bar{C}}{\bar{B}/\bar{A}\bar{C}},$$ and thus
$$|B /AC|\cdot |\bar{B} /\bar{A}\bar{C}| = | B /A\bar{C}|.$$ We then use the inclusions $\bar{B}\subseteq C$ and $A\cap\bar{B}\subseteq \bar{A}$ to show that the map
$$\begin{aligned}
\mu: L \times \mathcal{L} &\rightarrow  B /A\bar{C}\\
(l,\ell) &\mapsto l\ell A\bar{C}
\end{aligned}$$ is injective.

The proof of part 2) follows the exact same strategy as the proof of part 1).

\end{proof}

\subsection{Virtual Characters Through Quotients}\label{sec:virtualCharacters}
 
Virtual characters are required in the construction of the depth-zero part of a supercuspidal representation, as recalled in Figure~\ref{fig:summaryRho}. In this appendix, we derive a result involving the behaviour of virtual characters through central quotients.

A brief summary of virtual characters was provided in \cite[Appendix]{thesisPaper}. Following the notation therein, we let $G$ denote a connected reductive group over an algebraically closed field of characteristic $p$. Let $\mathrm{Fr}: G\rightarrow G$ be a Frobenius map (as per \cite[Section 1.17]{Carter:1985}). Note that when $G$ is defined over a finite field $\res$, a generator of $\Gal(\resun/\res)$ defines a Frobenius map $\mathrm{Fr}$ and $G^{\Gal(\resun/\res)} = G^{\mathrm{Fr}}$.

Given a maximal torus $T$ of $G$ which is $\mathrm{Fr}$-stable and an irreducible complex character $\theta$ of $T^{\mathrm{Fr}}$, there is a corresponding virtual character $R^G_{T,\theta}$ which maps $G^{\mathrm{Fr}}$ into $\mathbb{C}$. When the underlying group is clear, we omit the superscript $G$ and write $R_{T,\theta}$.

One formula for the virtual character is provided by \cite[Theorem 4.2]{DL:1976} and is described as follows. Given $g\in G^{\mathrm{Fr}}$, we let $g = su = us$ be its Jordan decomposition. Then
\begin{align*}
R_{T,\theta}(g) = \frac{1}{|(C_G(s)^\circ)^{\mathrm{Fr}}|}\sum_{\substack{x\in G^{\mathrm{Fr}}\\ x^{-1}sx \in T^{\mathrm{Fr}}}}\theta(x^{-1}sx)R_{xTx^{-1},1}^{C_G(s)^\circ}(u).
\end{align*}


It is known that virtual characters remain constant under isomorphism, and their behaviour under particular types of restrictions is given by \cite[Theorem A.2]{thesisPaper}. At the time of writing, we have not found a reference that discusses how virtual characters behave via central quotients, so we describe it with the following theorem.

\begin{theorem}\label{th:virtualCharQuotient}
Let $Z$ be a central subgroup of $G$ which is $\mathrm{Fr}$-stable, and consider the quotient map $q : G \rightarrow \GZ$, where $\GZ = G/Z$.  Let $T$ be a maximal torus of $G$ which is $\mathrm{Fr}$-stable and let $T_Z = q(T)$. Let $\theta$ and $\theta_Z$ be irreducible complex characters of $T^{\mathrm{Fr}}$ and $T_Z^{\mathrm{Fr}}$, respectively, satisfying $\theta = \theta_Z\circ q$. Then $R_{T,\theta} = R_{T_Z,\theta_Z} \circ q$.
\end{theorem}

In order to prove this theorem, we will require the following lemma on cohomology.

\begin{lemma}\label{lem:cohomology}
Let $Z, \GZ$ and $q$ be as in Theorem \ref{th:virtualCharQuotient}. Let $H$ be a connected subgroup of $G$ which contains $Z$ and is $\mathrm{Fr}$-stable. Then $H^1(\mathrm{Fr},Z) \simeq q(H)^{\mathrm{Fr}}/\left(H^{\mathrm{Fr}}/Z^{\mathrm{Fr}} \right)$. In particular, if $T$ is a maximal torus of $G$, then there exists a set of coset representatives of $\GZ^{\mathrm{Fr}}/(G^{\mathrm{Fr}}/Z^{\mathrm{Fr}})$ that is contained in $q(T)^{\mathrm{Fr}}$.
\end{lemma}

\begin{proof}
Consider the exact sequence
$$1 \rightarrow Z \rightarrow H \rightarrow q(H) \rightarrow 1.$$
Given that $H$ is connected, Lang's theorem implies $H^1(\mathrm{Fr},H) = 1$, giving us the following exact cohomology sequence \cite[Theorem 12.3.4]{Springer:LAG}:
$$1\rightarrow Z^{\mathrm{Fr}} \rightarrow H^{\mathrm{Fr}}\rightarrow q(H)^{\mathrm{Fr}} \rightarrow H^1(\mathrm{Fr},Z) \rightarrow 1.$$ The exactness of the sequence implies $H^1(\mathrm{Fr},Z) \simeq q(H)^{\mathrm{Fr}}/\left( H^{\mathrm{Fr}}/Z^{\mathrm{Fr}}\right)$.

For the second part, let $D$ be a set of coset representatives of $\GZ^{\mathrm{Fr}}/\left( G^{\mathrm{Fr}}/Z^{\mathrm{Fr}} \right)$. By what precedes above, one has $H^1(\mathrm{Fr},Z) \simeq \GZ^{\mathrm{Fr}}/\left( G^{\mathrm{Fr}}/Z^{\mathrm{Fr}} \right)$ and $H^1(\mathrm{Fr},Z)\simeq q(T)^{\mathrm{Fr}}/\left( T^{\mathrm{Fr}}/Z^{\mathrm{Fr}}\right)$. From \cite[Section 12.3.3]{Springer:LAG}, one sees from the definitions of the maps $\GZ^{\mathrm{Fr}}\rightarrow H^1(\mathrm{Fr},Z)$ and $q(T)^{\mathrm{Fr}}\rightarrow H^1(\mathrm{Fr},Z)$ that the elements of $D$ can be chosen from $q(T)^{\mathrm{Fr}}$ without loss of generality.
\end{proof}

\begin{proof}[Proof of Theorem \ref{th:virtualCharQuotient}]
Let $g\in G^{\mathrm{Fr}}$ and let $g=su=us$ be its Jordan decomposition. We must show that $R_{T,\theta}(g) = R_{T_Z,\theta_Z}\circ q(g)$. We start by studying the expression $R_{T,\theta}(g)$.

Let $C$ be a set of coset representatives of $G^{\mathrm{Fr}}/Z^{\mathrm{Fr}}$. Then, $G^{\mathrm{Fr}} = CZ^{\mathrm{Fr}}$. We thus obtain the following chain of equalities.
$$
\begin{aligned}
R_{T,\theta}(g) &= \frac{1}{|(C_G(s)^\circ)^{\mathrm{Fr}}|}\sum_{\substack{x\in G^{\mathrm{Fr}}\\ x^{-1}sx\in T^{\mathrm{Fr}}}}\theta(x^{-1}sx)R_{xTx^{-1},1}^{C_G(s)^\circ}(u) \\
&=\frac{1}{|(C_G(s)^\circ)^{\mathrm{Fr}}|}\sum_{\substack{c\in C\\ z\in Z^{\mathrm{Fr}} \\ (cz)^{-1}s(cz)\in T^{\mathrm{Fr}}}}\theta((cz)^{-1}s(cz))R_{(cz)T(cz)^{-1},1}^{C_G(s)^\circ}(u) \\
&= \frac{|Z^{\mathrm{Fr}}|}{|(C_G(s)^\circ)^{\mathrm{Fr}}|}\sum_{\substack{c\in C\\ c^{-1}sc \in T^{\mathrm{Fr}}}} \theta(c^{-1}sc)R_{cTc^{-1},1}^{C_G(s)^\circ}(u) \\
&= \frac{|Z^{\mathrm{Fr}}|}{|(C_G(s)^\circ)^{\mathrm{Fr}}|}\sum_{\substack{c\in C\\ c^{-1}sc \in T^{\mathrm{Fr}}}} \theta_Z(q(c)^{-1}q(s)q(c))R_{cTc^{-1},1}^{C_G(s)^\circ}(u).
\end{aligned}
$$
We claim that $R_{cTc^{-1},1}^{C_G(s)^\circ}(u) = R_{q(c)T_Zq(c)^{-1},1}^{C_{G_z}(q(s))^\circ}(q(u))$. Indeed, one sees from \cite[Theorem 3.5.3]{Carter:1985} that $q(C_G(s)^\circ) = C_{\GZ}(q(s))^\circ$. Given that $Z$ is in the centre of $C_G(s)^\circ$, it follows that $C_G(s)^\circ$ and $C_{\GZ}(q(s))^\circ$ have the same adjoint group, and the adjoint map of $C_G(s)^\circ$ is the composition of $q$ and the adjoint map of $C_{\GZ}(q(s))^\circ$. According to \cite[Formula 4.1.1]{DL:1976}, the virtual character remains constant under the adjoint map. Thus, it follows that $R_{cTc^{-1},1}^{C_G(s)^\circ}(u) = R_{q(c)T_Zq(c)^{-1},1}^{C_{G_z}(q(s))^\circ}(q(u))$ as claimed. Returning to our above computation, it then follows that
$$
\begin{aligned}
R_{T,\theta}(g) &= \frac{|Z^{\mathrm{Fr}}|}{|(C_G(s)^\circ)^{\mathrm{Fr}}|}\sum_{\substack{c\in C\\ c^{-1}sc \in T^{\mathrm{Fr}}}} \theta_Z(q(c)^{-1}q(s)q(c))R_{q(c)T_Zq(c)^{-1},1}^{C_{\GZ}(q(s))^\circ}(q(u)) \\
&= \frac{|Z^{\mathrm{Fr}}|}{|(C_G(s)^\circ)^{\mathrm{Fr}}|}\sum_{\substack{\bar{c}\in G^\mathrm{Fr}/Z^{\mathrm{Fr}}\\ \bar{c}^{-1}q(s)\bar{c} \in T^{\mathrm{Fr}}/Z^{\mathrm{Fr}}}} \theta_Z(\bar{c}^{-1}q(s)\bar{c})R_{\bar{c}T_Z\bar{c}^{-1},1}^{C_{\GZ}(q(s))^\circ}(q(u)).
\end{aligned}$$

Next, we look at the expression for $R_{T_Z,\theta_Z}\circ q(g)$. By \cite[Theorem 15.4]{Humphreys:LAG}, $q(g) = q(s)q(u) = q(u)q(s)$ is the Jordan decomposition of $q(g)$. Therefore,
$$
\begin{aligned}
R_{T_Z,\theta_Z}\circ q(g) = \frac{1}{|(C_{\GZ}(q(s))^\circ)^{\mathrm{Fr}}|}\sum_{\substack{\bar{x}\in \GZ^{\mathrm{Fr}}\\ \bar{x}^{-1}q(s)\bar{x}\in T_Z^{\mathrm{Fr}}}}\theta_Z(\bar{x}^{-1}q(s)\bar{x})R_{\bar{x}T_Z\bar{x}^{-1},1}^{C_{\GZ}(q(s))^\circ}(q(u)).
\end{aligned}
$$
Now, let $D$ be a set of coset representatives of $\GZ^{\mathrm{Fr}}/(G^{\mathrm{Fr}}/Z^{\mathrm{Fr}})$. By Lemma \ref{lem:cohomology}, we may assume that $D\subset T_Z^{\mathrm{Fr}}$. 
Furthermore, by setting $H = C_G(s)^\circ$ in Lemma \ref{lem:cohomology}, we have $H^1(\mathrm{Fr},Z) \simeq (q(C_G(s)^\circ))^{\mathrm{Fr}}/((C_G(s)^\circ)^{\mathrm{Fr}}/Z^{\mathrm{Fr}})$. This isomorphism then implies
$$\frac{|H^1(\mathrm{Fr},Z)|\cdot |(C_G(s)^\circ)^{\mathrm{Fr}}|}{|Z^{\mathrm{Fr}}|} = |(C_{\GZ}(q(s))^\circ)^{\mathrm{Fr}}|.$$ This allows us to rewrite the formula of $R_{T_Z,\theta_Z}(q(g))$ as follows.
$$
\begin{aligned}
 R_{T_Z,\theta_Z}\circ q(g) &= \frac{|Z^{\mathrm{Fr}}|}{|(C_{G}(s)^\circ)^{\mathrm{Fr}}|}\frac{1}{|H^1(\mathrm{Fr},Z)|}\sum_{\substack{d\in D \\ \bar{c}\in G^{\mathrm{Fr}}/Z^{\mathrm{Fr}}\\ (\bar{c}d)^{-1}q(s)(\bar{c}d)\in T_Z^{\mathrm{Fr}}}}\theta_Z((\bar{c}d)^{-1}q(s)(\bar{c}d))R_{(\bar{c}d)T_Z(\bar{c}d)^{-1},1}^{C_{\GZ}(q(s))^{\circ}}(q(u)) \\
 &= \frac{|Z^{\mathrm{Fr}}|}{|(C_{G}(s)^\circ)^{\mathrm{Fr}}|}\sum_{\substack{\bar{c}\in G^{\mathrm{Fr}}/Z^{\mathrm{Fr}}\\ \bar{c}^{-1}q(s)\bar{c}\in T_Z^{\mathrm{Fr}}}}\theta_Z(\bar{c}^{-1}q(s)\bar{c})R_{\bar{c}T_Z\bar{c}^{-1},1}^{C_{\GZ}(q(s))^{\circ}}(q(u)) \\
 &=\frac{|Z^{\mathrm{Fr}}|}{|(C_{G}(s)^\circ)^{\mathrm{Fr}}|}\sum_{\substack{\bar{c}\in G^{\mathrm{Fr}}/Z^{\mathrm{Fr}}\\ \bar{c}^{-1}q(s)\bar{c}\in T^{\mathrm{Fr}}/Z^{\mathrm{Fr}}}}\theta_Z(\bar{c}^{-1}q(s)\bar{c})R_{\bar{c}T_Z\bar{c}^{-1},1}^{C_{\GZ}(q(s))^{\circ}}(q(u)).
\end{aligned}
$$

This last expression is precisely the same expression we had for $R_{T,\theta}(g)$ above. Thus $R_{T,\theta} = R_{T_Z,\theta_Z}\circ q$.
\end{proof}



\begin{thebibliography}{ADSS11}

\bibitem[Adl98]{Adler:1998}
Jeffrey~D. Adler.
\newblock Refined anisotropic {$K$}-types and supercuspidal representations.
\newblock {\em Pacific J. Math.}, 185(1):1--32, 1998.

\bibitem[ADSS11]{ADSS}
Jeffrey~D. Adler, Stephen DeBacker, Paul~J. Sally, Jr., and Loren Spice.
\newblock Supercuspidal characters of {${\rm SL}_2$} over a {$p$}-adic field.
\newblock In {\em Harmonic analysis on reductive, {$p$}-adic groups}, volume
  543 of {\em Contemp. Math.}, pages 19--69. Amer. Math. Soc., Providence, RI,
  2011.

\bibitem[AMS18]{AMS}
Anne-Marie Aubert, Ahmed Moussaoui, and Maarten Solleveld.
\newblock Generalizations of the {S}pringer correspondence and cuspidal
  {L}anglands parameters.
\newblock {\em Manuscripta Math.}, 157(1-2):121--192, 2018.

\bibitem[Bor79]{Borel:1979}
Armand Borel.
\newblock Automorphic {$L$}-functions.
\newblock In {\em Automorphic forms, representations and {$L$}-functions
  ({P}roc. {S}ympos. {P}ure {M}ath., {O}regon {S}tate {U}niv., {C}orvallis,
  {O}re., 1977), {P}art 2}, Proc. Sympos. Pure Math., XXXIII, pages 27--61.
  Amer. Math. Soc., Providence, RI, 1979.

\bibitem[Bor91]{Borel:LAG}
Armand Borel.
\newblock {\em Linear algebraic groups}, volume 126 of {\em Graduate Texts in
  Mathematics}.
\newblock Springer-Verlag, New York, second edition, 1991.

\bibitem[Bou20]{thesis}
Ad\`ele Bourgeois.
\newblock On the restriction of supercuspidal representations: an in-depth
  exploration of the data.
\newblock Doctoral thesis, University of Ottawa,
  http://hdl.handle.net/10393/40901, 2020.

\bibitem[Bou21]{thesisPaper}
Ad\`ele Bourgeois.
\newblock Restricting supercuspidal representations via a restriction of data.
\newblock {\em Pacific J. Math.}, 312(1):1--39, 2021.

\bibitem[BT65]{BorelTits:1965}
Armand Borel and Jacques Tits.
\newblock Groupes r\'{e}ductifs.
\newblock {\em Inst. Hautes \'{E}tudes Sci. Publ. Math.}, (27):55--150, 1965.

\bibitem[BT72]{BT:1972}
François Bruhat and Jacques Tits.
\newblock Groupes r\'{e}ductifs sur un corps local.
\newblock {\em Inst. Hautes \'{E}tudes Sci. Publ. Math.}, (41):5--251, 1972.

\bibitem[BT84]{BT:1984}
François Bruhat and Jacques Tits.
\newblock Groupes r\'{e}ductifs sur un corps local. {II}. {S}ch\'{e}mas en
  groupes. {E}xistence d'une donn\'{e}e radicielle valu\'{e}e.
\newblock {\em Inst. Hautes \'{E}tudes Sci. Publ. Math.}, (60):197--376, 1984.

\bibitem[Car93]{Carter:1985}
Roger~W. Carter.
\newblock {\em Finite groups of {L}ie type}.
\newblock Wiley Classics Library. John Wiley \& Sons, Ltd., Chichester, 1993.
\newblock Conjugacy classes and complex characters, Reprint of the 1985
  original, A Wiley-Interscience Publication.

\bibitem[Dil23]{Dillery}
Peter Dillery.
\newblock Rigid inner forms over local function fields.
\newblock {\em Advances in Mathematics}, 430:109204, October 2023.

\bibitem[DL76]{DL:1976}
Pierre Deligne and George Lusztig.
\newblock Representations of reductive groups over finite fields.
\newblock {\em Ann. of Math. (2)}, 103(1):103--161, 1976.

\bibitem[DR09]{DR}
Stephen DeBacker and Mark Reeder.
\newblock Depth-zero supercuspidal {$L$}-packets and their stability.
\newblock {\em Ann. of Math. (2)}, 169(3):795--901, 2009.

\bibitem[FKS21]{FKS}
Jessica Fintzen, Tasho Kaletha, and Loren Spice.
\newblock A twisted {Y}u construction, {H}arish-{C}handra characters, and
  endoscopy.
\newblock {\em ArXiv e-prints}, 2021.

\bibitem[HM08]{HM:2008}
Jeffrey Hakim and Fiona Murnaghan.
\newblock Distinguished tame supercuspidal representations.
\newblock {\em Int. Math. Res. Pap. IMRP}, (2):Art. ID rpn005, 166, 2008.

\bibitem[Hum75]{Humphreys:LAG}
James~E. Humphreys.
\newblock {\em Linear algebraic groups}.
\newblock Springer-Verlag, New York-Heidelberg, 1975.
\newblock Graduate Texts in Mathematics, No. 21.

\bibitem[Hum95]{Humphreys:Conjugacy}
James~E. Humphreys.
\newblock {\em Conjugacy classes in semisimple algebraic groups}, volume~43 of
  {\em Mathematical Surveys and Monographs}.
\newblock American Mathematical Society, Providence, RI, 1995.

\bibitem[Jan03]{Jantzen}
Jens~Carsten Jantzen.
\newblock {\em Representations of algebraic groups}, volume 107 of {\em
  Mathematical Surveys and Monographs}.
\newblock American Mathematical Society, Providence, RI, second edition, 2003.

\bibitem[Kal13]{Kaletha:2013}
Tasho Kaletha.
\newblock Genericity and contragredience in the local {L}anglands
  correspondence.
\newblock {\em Algebra Number Theory}, 7(10):2447--2474, 2013.

\bibitem[Kal15]{Kaletha:2015}
Tasho Kaletha.
\newblock Epipelagic {$L$}-packets and rectifying characters.
\newblock {\em Invent. Math.}, 202(1):1--89, 2015.

\bibitem[Kal16]{Kaletha:2016}
Tasho Kaletha.
\newblock Rigid inner forms of real and {$p$}-adic groups.
\newblock {\em Ann. of Math. (2)}, 184(2):559--632, 2016.

\bibitem[Kal19]{Kaletha:Regular}
Tasho Kaletha.
\newblock Regular supercuspidal representations.
\newblock {\em J. Amer. Math. Soc.}, 32(4):1071--1170, 2019.

\bibitem[Kal21]{Kaletha:SCPackets}
Tasho Kaletha.
\newblock {Supercuspidal {L}-packets}.
\newblock {\em ArXiv e-prints}, 2021.

\bibitem[Kot82]{Kottwitz:1982}
Robert~E. Kottwitz.
\newblock Rational conjugacy classes in reductive groups.
\newblock {\em Duke Math. J.}, 49(4):785--806, 1982.

\bibitem[Kot86]{Kottwitz:1986}
Robert~E. Kottwitz.
\newblock Stable trace formula: elliptic singular terms.
\newblock {\em Math. Ann.}, 275(3):365--399, 1986.

\bibitem[KP23]{KalethaPrasad}
Tasho Kaletha and Gopal Prasad.
\newblock {\em Bruhat-{T}its theory---a new approach}, volume~44 of {\em New
  Mathematical Monographs}.
\newblock Cambridge University Press, Cambridge, 2023.

\bibitem[KS99]{ks}
Robert~E. Kottwitz and Diana Shelstad.
\newblock Foundations of twisted endoscopy.
\newblock {\em Ast\'{e}risque}, (255):vi+190, 1999.

\bibitem[LS87]{LS:1987}
Robert~P. Langlands and Diana Shelstad.
\newblock On the definition of transfer factors.
\newblock {\em Math. Ann.}, 278(1-4):219--271, 1987.

\bibitem[Mur11]{Murnaghan:2011}
Fiona Murnaghan.
\newblock Parametrization of tame supercuspidal representations.
\newblock In {\em On certain {$L$}-functions}, volume~13 of {\em Clay Math.
  Proc.}, pages 439--469. Amer. Math. Soc., Providence, RI, 2011.

\bibitem[Nev15]{Nevins:2015}
Monica Nevins.
\newblock Restricting toral supercuspidal representations to the derived group,
  and applications.
\newblock {\em J. Pure Appl. Algebra}, 219(8):3337--3354, 2015.

\bibitem[Sil79]{Silberger:1979}
Allan~J. Silberger.
\newblock Isogeny restrictions of irreducible admissible representations are
  finite direct sums of irreducible admissible representations.
\newblock {\em Proc. Amer. Math. Soc.}, 73(2):263--264, 1979.

\bibitem[Sol20]{Solleveld:2020}
Maarten Solleveld.
\newblock Langlands parameters, functoriality and {H}ecke algebras.
\newblock {\em Pacific J. Math.}, 304(1):209--302, 2020.

\bibitem[Spr79]{Springer:1979}
Tonny~A. Springer.
\newblock Reductive groups.
\newblock In {\em Automorphic forms, representations and {$L$}-functions
  ({P}roc. {S}ympos. {P}ure {M}ath., {O}regon {S}tate {U}niv., {C}orvallis,
  {O}re., 1977), {P}art 1}, Proc. Sympos. Pure Math., XXXIII, pages 3--27.
  Amer. Math. Soc., Providence, RI, 1979.

\bibitem[Spr09]{Springer:LAG}
Tonny~A. Springer.
\newblock {\em Linear algebraic groups}.
\newblock Modern Birkh\"{a}user Classics. Birkh\"{a}user Boston, Inc., Boston,
  MA, second edition, 2009.

\bibitem[Tha11]{Thang:2011}
Nguyen~Quoc Thang.
\newblock On {G}alois cohomology and weak approximation of connected reductive
  groups over fields of positive characteristic.
\newblock {\em Proc. Japan Acad. Ser. A Math. Sci.}, 87(10):203--208, 2011.

\bibitem[Vog93]{Vogan:1993}
David~A. Vogan, Jr.
\newblock The local {L}anglands conjecture.
\newblock In {\em Representation theory of groups and algebras}, volume 145 of
  {\em Contemp. Math.}, pages 305--379. Amer. Math. Soc., Providence, RI, 1993.

\bibitem[Yu01]{Yu:2001}
Jiu-Kang Yu.
\newblock Construction of tame supercuspidal representations.
\newblock {\em J. Amer. Math. Soc.}, 14(3):579--622, 2001.

\end{thebibliography}

\textsc{School of Mathematics and Statistics, Carleton University,
Ottawa, ON, Canada K1S 5B6}

\textit{E-mail addresses:} abour115@uottawa.ca, mezo@math.carleton.ca

\end{document}